%% file: GLE-v7.tex
\documentclass[11pt]{article}
\usepackage{epsfig,amsmath,amsthm,amssymb,graphicx,amsfonts,psfrag}
\usepackage[letterpaper,margin=1in]{geometry}
\usepackage{enumitem}

\usepackage{tikz}

\usepackage[colorlinks=true]{hyperref}

\usepackage{url}
\makeatletter
\g@addto@macro{\UrlBreaks}{\UrlOrds}

\newcommand{\xx}[1]{\,\text{ #1 }\,}
\newcommand{\Xx}[1]{\quad\text{ #1 }\,}
\newcommand{\XX}[1]{\quad\text{ #1 }\quad}

\theoremstyle{plain}
\newtheorem{thm}{Theorem}[section]
\newtheorem{cor}[thm]{Corollary}
\newtheorem{prop}[thm]{Proposition}
\newtheorem{lem}[thm]{Lemma}

\theoremstyle{definition}
\newtheorem{defn}[thm]{Definition}

\theoremstyle{remark}
\newtheorem{remark}[thm]{Remark}

\def\R{\mathbb{R}}

\def\N{\mathbb{N}}

\newcommand{\be}{\begin{equation}}
\newcommand{\ee}{\end{equation}}
\newcommand{\bea}{\begin{eqnarray}}
\newcommand{\eea}{\end{eqnarray}}
\newcommand{\beann}{\begin{eqnarray*}}
\newcommand{\eeann}{\end{eqnarray*}}
\newcommand{\benn}{\begin{equation*}}
\newcommand{\eenn}{\end{equation*}}

\def\ra{\rightarrow}
\def\I{\infty}

\newcommand{\cA}{{\mathcal A}}  % calligraphic A
\newcommand{\cE}{{\mathcal E}}  % calligraphic E
\newcommand{\cF}{{\mathcal F}}  % calligraphic F
\newcommand{\cK}{{\mathcal K}}  % calligraphic K
\newcommand{\cS}{{\mathcal S}}  % calligraphic S
\renewcommand{\d}[1][x]{\,\operatorname{d}\!#1}

\newcommand{\dx}{\d[x]}
\newcommand{\dy}{\d[y]}

\newcommand{\euler}{\mathrm e}

\newcommand{\Riesz}{D_0^\alpha}  % Riesz operator
\newcommand{\RieszFeller}{D_\theta^\alpha}  % Riesz-Feller operator
\newcommand{\Green}{G_\theta^\alpha}  % Green's function associated to Riesz-Feller operator
\newcommand{\Fourier}{\mathcal{F}}
\newcommand{\FourierInv}{\mathcal{F}^{-1}}
\newcommand{\SchwartzTF}{\mathcal{S}} % Schwartz test functions
\newcommand{\integral}[3]{\int_{#1} \, {#2} \, \d[#3]\,}
\newcommand{\integrall}[4]{\int_{#1}^{#2} \, {#3} \, \d[#4]\,}
\newcommand{\abs}[1]{|#1|}
\newcommand{\Abs}[1]{\bigg|#1\bigg|}
\newcommand{\difff}[3]{\tfrac{\partial^{#3} {#1}}{\partial {#2}^{#3}}}
\newcommand{\diff}[2]{\tfrac{\partial {#1}}{\partial {#2}}}
\newcommand{\Diff}[2]{\tfrac{d {#1}}{d {#2}}}

\newcommand{\norm}[1]{\| {#1} \|}
\newcommand{\Norm}[1]{\bigg\| {#1} \bigg\|}
\newcommand{\set}[2]{\{\, {#1} \,|\, {#2} \,\}}
\newcommand{\Set}[2]{\big\{\, {#1} \,\big|\, {#2} \,\big\}}

\newcommand{\Bb}{B_b} %(\R^n)}
\newcommand{\Cb}[1][]{C_b^{#1}}%(\R^n)}
\newcommand{\Co}{C_0}%(\R^n)}
\newcommand{\SG}{(S_t)_{t\geq 0}}
\newcommand{\SGextension}{(\tilde S_t)_{t\geq 0}}

\newcommand{\epsI}{\bar\epsilon}
\newcommand{\speedI}{\bar C}

\newcommand{\um}{u_-}
\newcommand{\up}{u_+}
\newcommand{\upm}{u_\pm}

\DeclareMathOperator{\Div}{div}

\DeclareMathOperator{\Id}{Id}

\DeclareMathOperator{\sgn}{sgn}
\begin{document}

\author{Franz Achleitner\footnotemark[1] and Christian Kuehn\footnotemark[1]}

\renewcommand{\thefootnote}{\fnsymbol{footnote}}
\footnotetext[1]{%
Institute for Analysis and Scientific Computing, 
Vienna University of Technology, 
1040 Vienna, Austria. \\
e-mail: franz.achleitner@tuwien.ac.at; ck274@cornell.edu
} 
\renewcommand{\thefootnote}{\arabic{footnote}}
 
\title{Traveling waves for a bistable equation with nonlocal-diffusion}

\maketitle

\begin{abstract}
We consider a single component reaction-diffusion equation in one dimension with bistable nonlinearity
 and a nonlocal space-fractional diffusion operator of Riesz-Feller type. 
Our main result shows the existence, uniqueness (up to translations) and stability of a traveling wave solution 
 connecting two stable homogeneous steady states. 
In particular,
 we provide an extension to classical results on traveling wave solutions involving local diffusion.
This extension to evolution equations with Riesz-Feller operators requires several technical steps. 
These steps are based upon an integral representation for Riesz-Feller operators, a comparison principle,
 regularity theory for space-fractional diffusion equations,
 and control of the far-field behavior. 
\end{abstract}


\section{Introduction}  
\label{sec:intro}
\input{GLE-Introduction}

\section{Riesz-Feller Operators}
\label{sec:RF}
\input{GLE-RieszFeller-L2}

\subsection{Extensions to Bounded Continuous Functions}
\input{GLE-RieszFeller-Cb}

%\subsection{A Comparison Principle}
%\label{ssec:comparison}
%\input{GLE-ComparisonPrinciple}

\section{Cauchy Problem and Comparison Principle}
\label{ap:cp}
\input{GLE-CauchyProblem}

\section{Traveling Wave Problem}
\label{sec:TWP}

We consider the traveling wave problem for the local reaction-nonlocal diffusion equation
\begin{equation} \label{RD:TWP}
  \diff{u}{t} = \RieszFeller u + f(u) \,,\quad x\in\R \,,\quad t\in (0,\infty) \,,
\end{equation}
whereat $1<\alpha\leq 2$, $|\theta| \leq \min\{\alpha, 2-\alpha\}$
 and $f\in C^\infty(\R)$ is a bistable function in the sense of~\eqref{As:f:0}.

\begin{defn}
  A traveling wave solution of~\eqref{RD:TWP} is a solution of the form $u(t,x)=U(\xi)$, 
  for some constant wave speed $c\in\R$, a traveling wave variable $\xi:=x-c t$,
  and a function~$U$ connecting different endstates $\lim_{\xi\to\pm\infty} U(\xi) = \upm$.
\end{defn}
The profile $U$ has to satisfy the traveling wave equation
 \[ - c U'(\xi) = \RieszFeller U + f(U) \]
 where $\RieszFeller$ has to be understood in the sense of the singular integral in Theorem~\ref{thm:RieszFeller:extension}
 which is well-defined for $C^2_b(\R)$ functions due to Proposition~\ref{prop:RieszFeller:estimate}. 

\subsection{Xinfu Chen's Approach and Results}
\label{sec:Chen}
\input{GLE-TWP-Chen}
\subsection{The Bistable Case with Nonlocal Diffusion}
\label{subsec:TWP}
\input{GLE-TWP}

\appendix

\section{Proof - Uniqueness}
\label{ap:uniq}
\input{GLE-Uniqueness}

\section{Proof - Stability}
\label{ap:stab}
\input{GLE-Stability}

% \textbf{TODO: put fully adapted proof for stability here; check carefully: it may not be exponentially stable... see numerics.}

\section{Proof - Existence}
\label{ap:exist}
\input{GLE-Existence}

\section{L\'evy Processes and Semigroups}
\label{ap:LevyProcesses+Semigroups}
\input{GLE-LevyProcesses+Semigroups}

\bibliographystyle{plain}
\bibliography{AK}

\end{document}

%% file: GLE-Introduction.tex
We consider partial integro-differential equations
 \begin{equation} \label{RD}
  \diff{u}{t} = \RieszFeller u + f(u) \,,\quad x\in\R \,,\quad t\in(0,\infty) \,,
 \end{equation}
 where $f\in C^1(\R)$ is a nonlinear function of bistable type,
 i.e. $f$ has precisely three roots $\um~<~a~<~\up$ in the interval $[\um,\up]$ such that 
 \begin{equation} \label{As:f:0} 
  \quad f(\um)=f(a)=f(\up)=0\,, \quad f'(\um)<0\,, \quad f'(\up)<0\,, % \quad f'(a)>0\,.
 \end{equation}
 and $\RieszFeller$ is a Riesz-Feller operator
 for some fixed parameters $1<\alpha\leq 2$ and $\abs{\theta} \leq \min\{\alpha,2-\alpha\}$.
A Riesz-Feller operator can be defined as a Fourier multiplier operator
 $\Fourier[\RieszFeller u](\xi) = \psi^\alpha_\theta (\xi) \Fourier [u] (\xi)$
 with symbol $\psi^\alpha_\theta(\xi) = -|\xi|^\alpha \exp\left[i(\sgn(\xi)) \theta \frac{\pi}{2}\right]$ 
 for $0<\alpha\leq 2$ and $|\theta| \leq \min\{\alpha, 2-\alpha\}$.
Special cases are the second order derivative $D^2_0=\partial_x^2$
 and the fractional Laplacians $D^\alpha_0=-(-\partial^2_x)^{\alpha/2}$ for $0<\alpha\leq 2$ and $\theta=0$; 
 for details see Section~\ref{sec:RF}.

Our aim is to prove existence, uniqueness (up to translations) and stability
 of traveling wave solutions $u(x,t)=U(x-ct)$ of~\eqref{RD}. % satisfying~\eqref{As:TWS}.
To make sense of~$\RieszFeller U$,
 we use an extension of Riesz-Feller operators $\RieszFeller$ for $1<\alpha<2$ to $f\in C^2_b(\R)$
 in the form of a singular integral
 \[
 \RieszFeller f(x) = c_1 \integrall{0}{\infty}{ \tfrac{f(x+\xi)-f(x)-f'(x)\,\xi}{\xi^{1+a}} }{\xi}
  + c_2 \integrall{0}{\infty}{ \tfrac{f(x-\xi)-f(x)+f'(x)\,\xi}{\xi^{1+a}} }{\xi} 
 \]
 for some $c_1,c_2>0$, see also Theorem~\ref{thm:RieszFeller:extension}.

First we briefly review previous results on traveling wave solutions
 of classical bistable reaction-diffusion equations in Subsection~\ref{ssec:CRD}
 and of bistable reaction-diffusion equations with fractional Laplacian in Subsection~\ref{ssec:FRD}.
Then we will present our main results in Subsection~\ref{ssec:main}
 and conclude with a discussion in Subsection~\ref{ssec:discussion}.

\subsection{Classical Bistable Reaction-Diffusion equations}
\label{ssec:CRD}
In particular, equation~\eqref{RD} with $D^2_0=\partial_x^2$ and $f(u)=u(1-u)(u-a)$
 is a reaction-diffusion equation with bistable nonlinear reaction term,
\be
\label{eq:Nagumo}
\frac{\partial u}{\partial t}=\frac{\partial^2 u}{\partial x^2}+u(1-u)(u-a)\,,  \quad x\in\R\,, \quad t>0\,,
%+ b\integral{}{u}{t}
\ee
 which is known as Nagumo's equation to model propagation of signals~\cite{McKean:1970, Nagumo},
 as one-dimensional real Ginzburg-Landau equation (RGLE) to model long-wave amplitudes
 e.g. in case of convection in binary mixtures near the onset of instability, 
 as well as Allen-Cahn equation to model phase transitions in solids~\cite{Allen+Cahn:1979}.

Equation~\eqref{eq:Nagumo} has three homogeneous steady states (or equilibria) $0=\um< a< \up=1$,
 where $u=\upm$ are locally asymptotically stable and $u=a$ is unstable. 
It is natural to search for monotone traveling wave solutions $u(x,t)=U(x-ct)=U(\xi)$ of \eqref{eq:Nagumo}
 which connect two stable states
 \be \label{As:TWS}
  \lim_{\xi\ra -\I}U(\xi)=\um\,, \qquad \lim_{\xi\ra \I}U(\xi)=\up \quad\text{and}\quad U'(\xi)>0 \quad\text{for all } \xi\in\R\,.
 \ee
The existence of - up to translation unique - traveling wave solutions $u(x,t)=U(x-ct)$ of general reaction diffusion equations
 \be
 \label{eq:RD:local}
 \diff{u}{t} = \difff{u}{x}{2} + f(u)\,, \quad x\in\R\,, \quad t>0\,,
 \ee
 with bistable function $f\in C^1(\R)$ and their stability are well-known;
 see {e.g.}~\cite{Aronson+Weinberger:1974,Fife+McLeod:1977,Volpert+etal:1994} and references therein. 

It is important to highlight that phase plane methods may be used \cite{FifeMcLeod1} to study the existence
 and uniqueness of traveling wave solutions of \eqref{eq:RD:local}. 
In case of a partial integro-differential equation like~\eqref{RD}
 these classical geometric methods do not generalize immediately. 
A similar remark applies to the asymptotic stability of traveling wave solutions
 - with exponential rate of decay -
 which may be deduced from a special variational structure for~\eqref{eq:RD:local}.
%  functional
%  \begin{equation} \label{eq:AC:VS}
%     \integral{\R}{e^{c\xi} \Big(\diff{u}{t}\Big)^2 
%     + \Diff{}{t} e^{c\xi} \Big(\Big(\diff{u}{\xi}\Big)^2-2[F(u)-F(1)H(\xi)]\Big) }{\xi}=0
%  \end{equation}
% where $\xi=x-ct$, $F(u)=\integrall{0}{u}{f(v)}{v}$ and $H$ is the Heaviside function~\cite{Fife+McLeod:1977}.

Reaction-diffusion equations with bistable reaction term arise in various applications.
The potential $F(u)=F(\um)+\integrall{\um}{u}{f(v)}{v}$ indicates 
 which stable state - $\um$ or $\up$ - will replace the other one.
In particular, the speed $c$ of a traveling wave from $\um$ to $\up$ has to have the same sign as $-\integrall{\um}{\up}{f(v)}{v}$.
Thus in case of a balanced potential $\integrall{\um}{\up}{f(v)}{v}=0$ a stationary traveling wave will exist,
 {i.e.} both stable states will co-exist.
In contrast, in case of an unbalanced potential $\integrall{\um}{\up}{f(v)}{v}\ne 0$,
 the traveling wave will move in the direction of the stable state with lesser potential value $F(u)$,
 also called the metastable state.
\medskip

% However, on the level of particles the diffusion might not be a Brownian motion - with a Laplacian as its infinitesimal generator - 
%  but a more general L\'{e}vy process such as a (compound) Poisson process or a L\'{e}vy stable process.
% This observation has been the starting point to consider reaction-diffusion models with nonlocal diffusion
%  and their mathematical analysis~\cite{Bates+etal:1997, Chen:1997, Zanette:1997, Nec+etal:2008, Imbert+Souganidis:2009,
%  Volpert+etal:2010, Engler:2010, Cabre+Sire:2010, Cabre+Sire:2011, Gui:2012, Palatucci+etal:2013, Chmaj:2013}.
In some applications a reaction-diffusion model with nonlocal diffusion may be more appropriate,
 see the articles~\cite{Bates+etal:1997, Chen:1997, Zanette:1997, Nec+etal:2008, Imbert+Souganidis:2009,
 Volpert+etal:2010, Engler:2010, Cabre+Sire:2010, Cabre+Sire:2011, Gui:2012, Palatucci+etal:2013, Chmaj:2013}
 for mathematical analysis of the traveling wave problem and further references on applications.

For example, Bates, Fife, Ren and Wang~\cite{Bates+etal:1997} considered a non-local variant 
\begin{equation} \label{eq:RD:Bates}
 \diff{u}{t} = J \ast u - u + f(u)\,, \quad x\in\R\,, \quad t>0 \,,
\end{equation}
for even non-negative functions $J\in C^1(\R)$ with $\integral{\R}{J(y)}{y}=1$, $\integral{\R}{\abs{y} J(y)}{y}<\infty$,
 as well as $J'\in L^1(\R)$ and a bistable function $f\in C^2(\R)$.
The assumptions on $J$ ensure that $J \ast u - u$ has similar properties as the Laplacian
 - most notably a maximum principle -  
 and that the problem exhibits a free energy functional. % as  
% \[ E[u(x)] = \tfrac14 \integral{\R}{ \integral{\R}{ J(x-y)(u(x)-u(y))^2 }{x} }{y} + \integral{\R}{ F(u(x)) }{x}\,. \]
The existence of traveling wave solutions is proved via homotopy to a reaction-diffusion model~\eqref{eq:RD:local};
 again the traveling wave will move in the direction of the stable state with lesser potential value.
%Under additional assumptions on $f$ results on the regularity of traveling wave solutions are derived.
However, asymptotic stability of traveling wave solutions with exponential rate of decay is proven
 only for standing traveling wave solutions,
 i.e. $c=0$ or equivalently the balanced case $\integrall{\um}{\up}{f(v)}{v}=0$,
 since a generalization of the variational structure in~\eqref{eq:RD:local} seems not to be available for~\eqref{eq:RD:Bates}. 

Xinfu Chen~\cite{Chen:1997} established a unified approach to prove the existence, uniqueness
 and asymptotic stability with exponential decay of traveling wave solutions
 for a class of nonlinear nonlocal evolution equations including~\eqref{eq:RD:local} and~\eqref{eq:RD:Bates}
 and many more examples from the literature.
His approach is suitable for equations supporting a comparison principle
 and based on constructing suitable sub- and super-solutions.
In Section~\ref{sec:Chen} we recall his assumptions and results in more detail.

\subsection{Bistable Reaction-Diffusion equations with Fractional Laplacian}
\label{ssec:FRD}
We briefly review previous results on traveling wave solutions for equation~\eqref{RD}
 with fractional Laplacian~$\Riesz$ where $0<\alpha<2$.
% The special case of~\eqref{RD} with a fractional Laplacian $\Riesz$ for $1<\alpha\leq 2$ or sometimes $0<\alpha\leq 2$ 
%  has been considered in the literature.

Zanette~\cite{Zanette:1997} proposed a model
% Exact solutions in front propagation problems with superdiffusion
\begin{equation} \label{eq:RD:Zanette}
 \diff{u}{t} = \Riesz u + f(u)\,, \quad x\in\R\,, \quad t>0\,, 
\end{equation}
 with a fractional Laplacian $\Riesz$ for some $0<\alpha\leq 2$,
% \[ \Fourier_{x\to \xi} \big[\Riesz u\big](\xi) = - \abs{\xi}^{\alpha} \Fourier_{x\to\xi} [u](\xi) \Xx{for some} \alpha\in (0,2)\,, \] 
 and an explicit bistable function $f$ given as 
 $f:[0,1]\to\R$, $u\mapsto -k[u-H(u-w)]$, where $k>0$, $w\in(0,1)$, and $H$ is the Heaviside function.
 % is bistable $f(\upm)=0$ and $f'(\upm)<0$. 
His motivation is to study the effects of anomalous diffusion represented by a fractional Laplacian
 in combination with a simple reaction function introduced by McKean in his study~\cite{McKean:1970} of Nagumo's equation~\cite{Nagumo}.
Restricting to monotone traveling wave solutions the traveling wave equation simplifies 
 and a formal Fourier transform yields an explicit solution in integral form.
Moreover, the asymptotic behavior of front tails and the front width are studied. 

Volpert, Nec and Nepomnyashchy~\cite{Volpert+etal:2010} consider~\eqref{eq:RD:Zanette} for $\alpha\in[1,2]$
 and study (the existence of) monotone traveling waves $u(x,t)=U(x-ct)$ satisfying~\eqref{As:TWS}
% \[ \lim_{\xi\to-\infty} U(\xi)=0 \Xx{and} \lim_{\xi\to\infty} U(\xi)=1 \]
 among other types of solutions.
They notice that if for general reaction term $f$ a traveling wave solution exists,
 then its speed has to satisfy - like in the case of a classical reaction-diffusion equation -
 \begin{equation} \label{eq:WS} % wave speed
   c = -\integrall{\um}{\up}{f(w)}{w} \Big[ \integral{\R}{ \Big(\Diff{U}{x}\Big)^2 }{x} \Big]^{-1}\,,
 \end{equation} 
 in particular the wave speed $c$ is zero in the balanced case $\integrall{\um}{\up}{f(w)}{w}=0$.
%We remark that this identity will be true in our general reaction-diffusion equation.
The asymptotic behavior of traveling wave solutions for $1\leq\alpha<2$ is derived as
 \[ U(\xi)\sim 
     \begin{cases} 
       \um+ \tfrac{1}{k\pi \abs{\xi}^\alpha} \sin\left(\tfrac{\alpha\pi}{2}\right) \Gamma(\alpha) &\xx{for} \xi\to-\infty\,, \\
       \up- \tfrac{1}{k\pi \abs{\xi}^\alpha} \sin\left(\tfrac{\alpha\pi}{2}\right) \Gamma(\alpha) &\xx{for} \xi\to\infty\,,
     \end{cases}
 \] 
 in contrast to exponential decay to the endstates in case of $\alpha=2$.

%\paragraph{fractional Allen-Cahn equation}
Nec, Nepomnyashchy and Golovin~\cite{Nec+etal:2008} consider 
 \begin{equation} \label{eq:AC:Nec} 
   \diff{u}{t} = \Riesz u + u(1-u^2)\,, \quad x\in\R\,, \quad t>0\,,
 \end{equation}
 with fractional Laplacian $\Riesz$ for $\alpha\in (1,2)$ and bistable reaction function $f(u)=u(1-u^2)$.
They derive a variational formulation
 such that~\eqref{eq:AC:Nec} is the associated Euler-Lagrange equation.
% They extend the fractional Laplacian in form of a second-order derivative of a Riesz potential
%  \[
%    D^{\gamma}_{\abs{x}} u(x) = \tfrac{\sec (\pi \gamma/2)}{2\Gamma(2-\gamma)} \difff{}{x}{2}
%      \integral{\R}{ \tfrac{u(y)}{\abs{x-y}^{\gamma-1}} }{y}\,,
%  \]
%  which is not well-defined for constant solutions or traveling wave solutions approaching distinct constant endstates.
% However for suitable functions,
%  an associated variational formulation is given by 
%  \[ \Upsilon(t)=\integral{\R}{\upsilon(x,t)}{x} \]
%  with $\upsilon(x,t)=\tfrac14(1-u^2)^2 - \tfrac{\sec (\pi \gamma/2)}{2\Gamma(2-\gamma)}
%          ( \diff{u}{x} \diff{}{x} \integral{\R}{ \tfrac{u(y)}{\abs{x-y}^{\gamma-1}} }{y}  
%            + \tfrac{u}{2} \integral{\R}{ \difff{u}{x}{2}(y) \tfrac{1}{\abs{x-y}^{\gamma-1}} }{y} )$.
% The first variation $\delta \Upsilon$ satisfies 
%  \[ \delta\Upsilon = -\integral{\R}{ \big( D^{\gamma}_{\abs{x}} u + u(1-u^2) \big) \delta u }{x} \]
%  and equation~\eqref{eq:AC:Nec} is recovered as the associated Euler-Lagrange equation 
%  \[ \diff{u}{t} = \tfrac{\delta\upsilon}{\delta u} = D^{\gamma}_{\abs{x}} u + u(1-u^2)\,. \]
% Hence, they conclude that any solution of~\eqref{eq:AC:Nec} will decay asymptotically to a stationary solution due to
%  \[ \diff{\Upsilon}{t} = \integral{\R}{ \tfrac{\delta\upsilon}{\delta u} \diff{u}{t} }{x}
%       = - \integral{\R}{ \big(\diff{u}{t}\big)^2 }{x} < 0.
%  \]
To study traveling wave solutions, 
 first they perform an a-priori analysis of asymptotic behavior of front tails.
Then they devise three families of ansatz functions with the correct asymptotic far-field behavior
 and use the variational functional to identify the minimizer among each family of ansatz functions.
% They point out that
%  \emph{``the principal value integrals in $\upsilon(t)$ are regularized accordingly in each case''}. 
Finally the three minimizers are compared with a numerical solution. 
% with respect to value of the functional, front slope at the origin, and inverse domain wall mobility.
This concludes their analysis of the balanced case $\integrall{-1}{1}{f(w)}{w}=F(1)-F(-1)=0$
 with $F(u)=\tfrac14 (1-u^2)^2$ and roots $\upm=\pm 1$.
They also consider an unbalanced case $F'(\upm)=0$ and $F''>0$ such that $F(\up)>F(\um)$ 
 and establish the existence of non-stationary traveling wave solutions $\lim_{\xi\to\pm\infty} U(\xi)=\upm$
 with negative wave speed $c$ - see also~\eqref{eq:WS} - i.e. $\up$ will replace $\um$.
%  \[ \diff{u}{t} = D^{\gamma}_{\abs{x}} u + f(u)\,, \quad x\in\R\,, \quad t>0\,, \]
%  a Riesz fractional derivative
%  \[ D^{\gamma}_{\abs{x}} u(x) = \frac{\sec (\pi \gamma/2)}{2\Gamma(2-\gamma)} \difff{}{x}{2}
%      \integral{\R}{ \frac{u(y)}{\abs{x-y}^{\gamma-1}} }{y}\]
%  The Riesz fractional derivative is another extension of the $\Riesz$ to $C^2_b(\R)$.

 Cabr\'e and Sire~\cite{Cabre+Sire:2010, Cabre+Sire:2011} consider
  \begin{equation} \label{eq:RD:Cabre} 
    0 = \Riesz U + f(U) \Xx{in} \R^n 
  \end{equation}
  for $\alpha\in(0,2)$ and a function $f\in C^{1,\gamma}(\R)$ with $\gamma>\max(0,1-\alpha)$.
 They study the existence of layer solutions, i.e. 
  $\diff{U}{x_n}(x)>0 \xx{for any} x\in\R^n$ and $\lim_{x_n\to\pm\infty} U(x',x_n)=\upm$ for any $x'\in\R^{n-1}$.
 Due to a result by Caffarelli and Silvestre~\cite{Caffarelli+Silvestre:2007},
  they relate equation~\eqref{eq:RD:Cabre} to a nonlinear BVP 
 \[\begin{cases} 
  \Div(y^\beta \nabla v) = 0 &\xx{in} \R^{n+1}_+ = \set{(x,y)\in\R^n\times\R}{y>0}\,, \\
	(1+\beta) \diff{v}{\nu^\beta} = f(v) &\xx{in} \partial \R^{n+1}_+\,,
 \end{cases}\]
  where $\beta=1-\alpha$, $v=v(x,y)$ is real-valued
  and $\diff{v}{\nu^\beta} = -\lim_{y\to 0} y^\beta \diff{v}{y}$.
	% is generalized exterior normal derivative of v
 The most complete results are obtained for $n=1$,
  where a layer solution of~\eqref{eq:RD:Cabre} is a stationary traveling wave solution of~\eqref{eq:RD:Zanette}.
 They prove that a layer solution of~\eqref{eq:RD:Cabre} exists
  if and only if $f$ is function of bistable type $f(\um)=f(\up)=0$ with a balanced potential $\integrall{\um}{\up}{f(v)}{v}=0$;
 if - in addition - $f'(\upm)>0$ then they prove that a layer solution is unique up to translations;
 if - in addition - $f$ is odd then they establish that a layer solution is odd w.r.t. some point.
 Moreover they derive the asymptotic behavior of front tails.

 Palatucci, Savin and Valdinoci~\cite{Palatucci+etal:2013}
  investigate the existence, uniqueness and other geometric properties of the minimizers of the energy functional 
  \begin{equation} \label{functional:Palatucci}
    \cE(U,\Omega) := \cK(U,\Omega) + \integral{\Omega}{F(U(x))}{x}
  \end{equation}
  where $\cK(U,\Omega)$ can be viewed as the contribution in $\Omega$ of the squared $H^s$ semi-norm of $U$,
  and $F$ is a double-well potential with $F(\upm)=0$. 
 First they show that layer solutions of~\eqref{eq:RD:Cabre}, i.e. 
 \[ \diff{U}{x_n}(x)>0 \xx{for any} x\in\R^n \Xx{and} \lim_{x_n\to\pm\infty} U(x',x_n)=\upm \xx{for any} x'\in\R^{n-1}\,, \]
 are local minimizers of the functional $\cE(U,\R)$. 
 The most complete results are obtained for $n=1$,
  where - again - a layer solution of~\eqref{eq:RD:Cabre} is a stationary traveling wave solution of~\eqref{eq:RD:Zanette}.
 For a bistable function $f\in C^1(\R)$ with balanced potential $F(U)=F(\um)+\integrall{\um}{U}{f(v)}{v}$,
  they prove the existence of a unique (up to translations) nontrivial global minimizer $U$
%   \[ u^{(0)}\in \set{f\in L^1_{loc}(\R)}{\lim_{x\to\pm\infty} f(x)=\pm 1} \]
  of the energy $\cE$ which is strictly increasing.
 This minimizer $U$ solves~\eqref{eq:RD:Cabre} 
  and is unique (up to translations) also in the class of monotone solutions of this equation.
 Moreover, they establish that $U$ belongs to $C^2(\R)$
   and derive the asymptotic behavior of front tails.

 Thus, among other results and independently from another,
  Cabr\'e and Sire~\cite{Cabre+Sire:2010, Cabre+Sire:2011} as well as Palatucci, Savin and Valdinoci~\cite{Palatucci+etal:2013} establish
  the existence and uniqueness of stationary traveling wave solutions of~\eqref{eq:RD:Zanette}
  for bistable functions $f$ with balanced potential.  
\medskip

% \paragraph{nonlocal Allen-Cahn equation - unbalanced case}
Chmaj~\cite{Chmaj:2013} investigated the existence of traveling waves for 
 \begin{equation} \label{eqn:AC:Chmaj}
  \diff{u}{t} = \Riesz u + f(u)\,, \quad x\in\R\,, \quad t>0\,,
 \end{equation}
 where $\Riesz$ is the fractional Laplacian with $\alpha\in(0,2)$
 and $f$ is bistable but not necessarily balanced, i.e. $f(\upm) = 0$ and $\integrall{\um}{\up}{f(v)}{v}\in\R$. 
He proves existence of traveling wave solutions $u=U(x-ct)$ of~\eqref{eqn:AC:Chmaj} satisfying~\eqref{As:TWS}
 where $\sgn c=-\sgn \integrall{\um}{\up}{f(v)}{v}$.
The idea of Chmaj's proof is to note that a traveling wave solution $u=U(x-ct)$ satisfies
 the traveling wave equation $-cU' = \Riesz U + f(U)$.
The fractional Laplacian can be approximated with non-singular operators of the form $J_\epsilon \ast U - (\int J_\epsilon) U$
 such that 
 \[ J_\epsilon \ast U - \big(\integral{\R}{ J_\epsilon(x)}{x}\big) U \stackrel{\epsilon\to 0}{\longrightarrow} \Riesz U\,. \]
The associated (traveling wave) equations $-cU' = J_\epsilon \ast U - (\int J_\epsilon) U + f(U)$
 exhibit for all sufficiently small $\epsilon>0$ a unique monotone solution $(U_\epsilon,c_\epsilon)$ with $U_\epsilon'>0$
 see also~\cite{Bates+etal:1997, Chen:1997}.
Finally the existence of a limit $U = \lim_{\epsilon\to 0} U_\epsilon$ is established
 and that $U$ is a traveling wave solution of~\eqref{eqn:AC:Chmaj}.

Changfeng Gui announced~\cite{Gui:2012} % (see also \url{http://birs.ca/events/2012/5-day-workshops/12w5100/videos})
 a different proof together with Mingfeng Zhao for the existence and properties of traveling waves
 in the fractional bistable equation~\eqref{eqn:AC:Chmaj}.
They consider (unbalanced) double well potentials $F\in C^{2,\gamma}(\R)$
 and prove existence of unique traveling wave solutions $u\in C^2(\R)$ via homotopy to the balanced case.
Moreover they announce results on the asymptotic behavior of front tails.

The reaction-diffusion equation~\eqref{RD} with general Riesz-Feller operators has been considered by Hans Engler~\cite{Engler:2010}.
Starting from the fundamental solution of the associated fractional diffusion equations~\eqref{eq:linearRF},
 he constructs traveling wave solutions of~\eqref{RD} for some suitable bistable function~$f$.
In the following, he assumes the existence of traveling wave solutions 
 and proves that the wave speed is bounded.

\subsection{Main Results}
\label{ssec:main}
% We consider local reaction-nonlocal diffusion equation of the form~\eqref{RD},
%  \[ \diff{u}{t} = \RieszFeller u + f(u) \,,\quad x\in\R \,,\quad t\in(0,\infty) \,, \]
%  for some fixed parameters $1<\alpha\leq 2$, $\abs{\theta} \leq \min\{\alpha,2-\alpha\}$,
%  and a bistable function $f\in C^1(\R)$ with~\eqref{As:TWS}.
Our main result is summarized in the following theorem.
\begin{thm} \label{thm:main}
Suppose $1<\alpha\leq 2$, $\abs{\theta} \leq \min\{\alpha,2-\alpha\}$ and $f\in C^\infty(\R)$ satisfies~\eqref{As:f:0}.
Then equation~\eqref{RD} admits a traveling wave solution $u(x,t)=U(x-ct)$ satisfying~\eqref{As:TWS}.
In addition, a traveling wave solution of~\eqref{RD} is unique up to translations.
Furthermore, traveling wave solutions are globally asymptotically stable in the sense that
 there exists a positive constant $\kappa$ such that 
 if $u(x,t)$ is a solution of~\eqref{RD} with initial datum~$u_0\in C_b(\R)$ satisfying $0\leq u_0\leq 1$ and
 \[ \liminf_{x\to\infty} u_0(x) > a\,, \qquad \limsup_{x\to-\infty} u_0(x) < a\,, \]
 then, for some constants $\xi$ and $K$ depending on $u_0$,
 \[ \norm{u(\cdot,t)-U(\cdot- ct + \xi)}_{L^\infty(\R)} \leq K e^{-\kappa t} \qquad \forall t\geq 0\,. \] 
\end{thm} 
Our proof is structured as follows.
In Section~\ref{sec:RF},
 we introduce the Riesz-Feller operators as Fourier multiplier operators on Schwartz functions,
 then we extend the Riesz-Feller operators in form of singular integrals to functions in $C^2_b(\R)$.
We consider the semigroup generated by the Riesz-Feller operators $\RieszFeller$
 with the help of results from the theory of L\'{e}vy processes, see also Appendix~\ref{ap:LevyProcesses+Semigroups}.

In Section~\ref{ap:cp},
 we investigate the Cauchy problem for~\eqref{RD} with initial datum $u_0\in C_b(\R)$ such that $0\leq u_0\leq 1$.
We follow a standard approach, 
 to consider the Cauchy problem in its mild formulation
 and to prove the existence of a mild solution.
The Cauchy problem generates a nonlinear semigroup  
 which allows to prove uniform $C^k_b$ estimates via a bootstrap argument
 and to conclude that mild solutions are also classical solutions.
% \begin{thm}[Theorem~\ref{thm:CP:existence:bistable}] %\label{thm:CP:existence:bistable:slides}
% Suppose $1<\alpha\leq 2$, $|\theta| \leq \min\{\alpha,2-\alpha\}$ and $f\in C^\infty(\R)$ satisfies~\eqref{As:f:slides}.
% The Cauchy problem~\eqref{RD:slides} with initial condition $u(\cdot ,0)=u_0\in L^\infty(\R)$ and $0\leq u_0\leq 1$ a.e.
%  has a solution $u(x,t)$ in the following sense: for all $T>0$
%  \begin{enumerate} %[label=(DI\arabic*)] %\setcounter{enumi}{19}
%   \item %\label{As:DI:20}
%     $u\in C_b((0,T)\times\R)$ and 
%     $u\in C^\infty_b((t_0,T)\times\R)$ for all $t_0\in(0,T)$;
%   \item %\label{As:DI:21}
%     $u$ satisfies~\eqref{RD:slides} on $(0,T)\times\R$;
%   \item %\label{As:DI:22}
%     If $u_0\in C_b(\R)$ then $u(t,.)\to u_0$ uniformly on $\R$ as $t\to 0$;
%   \item \label{As:DI:xx}
% 	  $0\leq u(x,t) \leq 1$ for all $(x,t)\in\R\times(0,\infty)$;
%   \item $\forall k\in\N$ $\forall t_0>0$ $\exists C>0$ such that $\norm{u(\cdot ,t)}_{C^k_b(\R)}\leq C$ $\forall 0<t_0<t$. 
% % \[ \forall k\in\N: \qquad \exists C>0: \qquad \forall 0<t_0 \Xx{and} \forall t_0<t \qquad \norm{u}_{C^k_b(\R)}\leq C\,. \] 
%  \end{enumerate}
% \end{thm}
In Subsection~\ref{ssec:comparison} we establish a comparison principle for the partial integro-differential equation~\eqref{RD}
 and investigate the behavior of the spatial limits of solutions.
The comparison principle is essential to prove our result on the existence, uniqueness and stability of traveling wave solutions 
 and to allow for a larger class of admissible functions $f$ in the result for the Cauchy problem.
Moreover, in the existence proof we need to show that the (continuous) solution of the Cauchy problem with some prepared initial datum
 exhibits spatial limits at all times. 
Therefore, we prove Theorem~\ref{thm:far-field-behaviour} on the far-field behavior of solutions.

In Section~\ref{sec:TWP}, we consider the traveling wave problem for~\eqref{RD}.
First we recall the results by Xinfu Chen~\cite{Chen:1997}.
Then we study his necessary assumptions and notice that some estimates are not of the required form.
However Xinfu Chen's approach can be extended, which we prove in the Appendices~\ref{ap:uniq}--\ref{ap:exist}. 
Our main result in Theorem~\ref{thm:main}
 will follow from the separate results on uniqueness in Theorem~\ref{thm:uniqueness},
 on stability in Theorem~\ref{thm:stability}
 and on existence of a traveling wave solution in Theorem~\ref{thm:existence}.
The details are given in Subsection~\ref{subsec:TWP}.

\subsection{Discussion}
\label{ssec:discussion}
To our knowledge,
 we establish the first result on existence, uniqueness (up to translations) and stability of traveling wave solutions of~\eqref{RD}
 with Riesz-Feller operators~$\RieszFeller$ for $1<\alpha< 2$ and $|\theta| \leq \min\{\alpha,2-\alpha\}$.

The variational approach is - at the moment - restricted to symmetric diffusion operators such as fractional Laplacians 
 and to stationary traveling waves that means to balanced potentials. 
Whereas, the approach by Changfeng Gui and Mingfeng Zhao is based on a homotopy to a balanced potential.
%  e.g. using results from variational methods~\cite{Cabre+SolaMorales:2005, Cabre+Sire:2010, Cabre+Sire:2011, Palatucci+etal:2013}.
% In his presentation~\cite{Gui:2012} %\url{http://birs.ca/events/2012/5-day-workshops/12w5100/videos},
%  Changfeng Gui gives a personal account 
%  why some classical approaches are not applicable to a unbalanced bistable reaction and nonlocal diffusion equation.
It might be possible to modify Chmaj's approach
 to study also our reaction-diffusion equation~\eqref{RD} with Riesz-Feller operators.
However his approach is concerned with the existence of a traveling wave.
It may be desireable to establish uniqueness and stability of traveling wave solutions as well.
By following Xinfu Chen's approach we obtain all these properties
 for a traveling wave solution of~\eqref{RD} directly.  

In contrast, the existence of traveling wave solutions of equation~\eqref{eq:RD:Zanette}
 with bistable function $f$ and fractional Laplacian~$\Riesz$ with $0<\alpha\leq 1$
 has been established in case of balanced potentials~\cite{Cabre+Sire:2010, Cabre+Sire:2011, Palatucci+etal:2013}
 and in the unbalanced case by~\cite{Chmaj:2013, Gui:2012}.
However, to extend Xinfu Chen's approach, if this is possible,
 to the general case of Riesz-Feller operators with $0<\alpha\leq 1$ and $\abs{\theta} \leq \min\{\alpha,2-\alpha\}$
 remains an open problem.

The interest in reaction-diffusion equations of the form~\eqref{RD} arose at the same time~\cite{Zanette:1997},
 at which Xinfu Chen published his results~\cite{Chen:1997}.
This may explain why Xinfu Chen did not consider also these examples.

%% file: GLE-RieszFeller-L2.tex
We follow Mainardi, Luchko and Pagnini~\cite{MainardiLuchkoPagnini}
 in their definition of the Riesz-Feller fractional derivative as a Fourier multiplier operator.
They use a definition of the Fourier transform which is custom in probability theory.
For $f$ in the Schwartz space 
\be
\cS(\R)=\left\{f\in C^\I(\R):\sup_{x\in\R}\left|x^\beta 
\frac{\partial^\gamma f }{\partial x^\gamma}(x)\right|<\I,~\forall \beta,\gamma\in \N_0\right\}
\ee
the Fourier transform is defined as 
\begin{align*}
  \cF [f](\xi) &:= \int_\R e^{+i \xi x} f(x) dx \,, \qquad \xi\in\R \,,
\intertext{and the inverse Fourier transform as}
  \cF^{-1} [f](x) &:= \frac{1}{2\pi} \int_\R e^{-i \xi x} f(\xi) d\xi \,, \qquad x\in\R \,.
\end{align*}
In the following, $\Fourier$ and $\FourierInv$ will denote also their respective extensions to $L^2(\R)$.
Then the Riesz-Feller space-fractional derivative of order $\alpha$ and skewness $\theta$
 is the Fourier multiplier operator
\be
 \label{eq:RF:transform}
 \cF [D^\alpha_\theta f](\xi) = \psi^\alpha_\theta (\xi) \cF [f] (\xi) \,,\qquad  \xi\in\R \,,
\ee
with symbol
\be
 \label{eq:RF:symbol}
  \psi^\alpha_\theta(\xi) = -|\xi|^\alpha \exp\left[i(\text{sgn}(\xi)) \theta \frac{\pi}{2}\right] \,, \quad 
  0<\alpha\leq 2 \,, \quad |\theta| \leq \min\{\alpha, 2-\alpha\} \,.
\ee
The symbol $\psi^\alpha_\theta (\xi)$ is the logarithm of the characteristic function
 of a L\'{e}vy strictly stable probability density with index of stability~$\alpha$ and asymmetry parameter~$\theta$
 according to Feller's parameterization \cite{Feller2,GorenfloMainardi};
 see also Section \ref{ssec:linear}.

\subsection{The Linear Space-Fractional Diffusion Equation}
\label{ssec:linear}

To analyze the Cauchy problem for the reaction diffusion equation~\eqref{RD}
 we need to investigate the linear space-fractional diffusion equation
\be
\label{eq:linearRF}
\frac{\partial u}{\partial t}(x,t) = \RieszFeller [u(\cdot ,t)](x) \,,\quad (x,t)\in\R\times (0,\I)\,,
\ee
for $0<\alpha\leq 2$ and $\abs{\theta} \leq \min\{\alpha,2-\alpha\}$.
A formal Fourier transform of the associated Cauchy problem yields
\benn
\diff{}{t} \Fourier[u](\xi,t) = \psi^\alpha_\theta(\xi) \Fourier[u](\xi,t)\,, 
  \qquad \Fourier[u](\xi,0) = \Fourier[u_0](\xi)\,, 
\eenn 
 which has a solution $\Fourier[u](\xi,t)=e^{t\psi_\theta^\alpha(\xi)}\Fourier[u_0](\xi)$. 
Hence, a formal solution of the Cauchy problem is given by 
\be \label{SG}
 u(x,t) = (\Green(\cdot ,t) \ast u_0)(x)
\ee
with kernel (or Green's function) 
\be
\label{eq:Green}
\Green(x,t): = \mathcal{F}^{-1} \left[\exp( t \psi^\alpha_\theta(\cdot))\right] (x)\,.
\ee
To study the properties of the formal solution,
 first we investigate the kernel $\Green$ 
 and then we verify that~\eqref{SG} defines a semigroup of solutions. 

\begin{lem} \label{lem:SSPM} % L\'evy strictly stable probability measures
For $0<\alpha\leq 2$ and $\abs{\theta} \leq \min\{\alpha,2-\alpha\}$, $\Green(x,t)$
 is the probability measure of a L\'{e}vy strictly $\alpha$-stable distribution.

Moreover for $\abs{\theta}<1$ the probability measure $\Green$ is absolutely continuous with respect to the Lebesgue measure
 and possesses a probability density which will be denoted again by $\Green$.
Furthermore, for all $(x,t)\in \R \times (0,\I)$ 
%  $0<\alpha\leq 2$, $\abs{\theta} \leq \min\{\alpha,2-\alpha\}$
 the following properties hold;
  \begin{enumerate}[label=(G\arabic*)] 
    \item \label{K:SFDE:prop0} $\Green(x,t)\geq 0$,
    \item \label{K:SFDE:prop1} $\Green(x,t)=t^{-1/\alpha} \Green (x t^{-1/\alpha},1)$. 
    \item \label{K:SFDE:prop2} $\norm{\Green(\cdot,t)}_{L^1(\R)}=1$,
    \item \label{K:SFDE:prop3} $\Green(\cdot,s)\ast \Green(\cdot,t) = \Green(\cdot,s+t)$ for all $s,t\in(0,\I)$,
    \item \label{K:SFDE:prop4} $\norm{\Green(\cdot,t)}_{L^p(\R)}\leq \norm{\Green(\cdot,1)}_{L^p(\R)} 
    t^{\frac{1-p}{\alpha p}}$ for all $1\leq p <\infty$,
    \item \label{K:SFDE:prop5} $\Green\in C^\infty_0(\R\times (0,\I))$.
  \end{enumerate}
Moreover, for $1\leq \alpha\leq 2$, $\abs{\theta} \leq \min\{\alpha,2-\alpha\}$ and $\abs{\theta}<1$, 
  \begin{enumerate}[label=(G\arabic*)] \setcounter{enumi}{6}
    \item \label{K:SFDE:prop6} For all $m\geq 0$ there exists a constant $B_m\in(0,\I)$ such that 
      \be 
      \label{eq:prop6} 
	  \Abs{\difff{}{x}{m} \Green(x,t)} \leq t^{-(1+m)/\alpha} \frac{B_m}{1+t^{-2/\alpha} x^2} \,, 
          \qquad \forall (x,t)\in \R\times (0,\I) \,.
      \ee
   \item \label{K:SFDE:prop7} For all $t>0$, there exists a $\cK$ such that $\norm{\diff{G}{x} (\cdot,t)}_{L^1(\R)} = \cK t^{-1/\alpha}$.
   \item \label{K:SFDE:prop8} $\Green(\cdot,s)\ast \diff{\Green}{x} (\cdot,t) = 
   \diff{\Green}{x} (\cdot,s+t)$ for all $s,t\in(0,\I)$.
   \item \label{K:SFDE:prop9} For all $t>t_0>0$ and $u\in L^1(\R)$ we have $(\Green(\cdot,t)\ast u)\in C^\infty(\R)$.
  \end{enumerate}
  For $0<\alpha\leq 2$ and $\abs{\theta} \leq \min\{\alpha,2-\alpha\}$ and $\alpha\neq \pm \theta$ ({i.e.}~excluding the so-called extremal 
pdfs)
    \begin{enumerate}[label=(G\arabic*)] \setcounter{enumi}{10}
    \item \label{K:SFDE:prop10} $\Green(x,t)> 0$.
  \end{enumerate}
\end{lem}

\begin{proof}
Due to Theorem~\cite[Theorem 14.19]{Sato:1999}, 
 the function $e^{t\psi_\theta^\alpha(\xi)}$ is the characteristic function 
 of a random variable with L\'{e}vy strictly $\alpha$-stable distribution.
Thus $\Green$ is the scaled probability measure of a L\'{e}vy strictly $\alpha$-stable distribution.
In case of $(\alpha,\theta)\in\{(0,0),(1,1),(1,-1)\}$, 
 the probability measure $\Green$ is a delta distribution 
 \[ G^0_0(x,t) = \delta_x\,, \qquad G^1_1(x,t) = \delta_{x+t}\,, \qquad G^1_{-1}(x,t) = \delta_{x-t} \]
 and called trivial~\cite[Definition 13.6]{Sato:1999}.
In all other (non-trivial) cases,
 the probability measure $\Green$ is absolutely continuous with respect to the Lebesgue measure
 and has a continuous probability distribution density~\cite[Proposition 28.1]{Sato:1999}, which we will denote again by $\Green$.
A non-trivial strictly $\alpha$-stable probability density is pointwise non-negative~\ref{K:SFDE:prop0}
 satisfies the scaling property~\ref{K:SFDE:prop1} due to~\cite[Remark 14.18]{Sato:1999}
 and hence the identity~\ref{K:SFDE:prop2} and the estimate~\ref{K:SFDE:prop4} follow.
The (semigroup)-property is satisfied by the defining property of strictly $\alpha$-stable probability density
 \cite[Definition 13.1]{Sato:1999}. 
Moreover a non-trivial strictly $\alpha$-stable probability density is $C^\infty$-smooth
 whose partial derivatives of all orders tend to $0$ in the limits $x\to\pm\infty$~\cite[Proposition 28.1; Example 28.2]{Sato:1999},
 hence~\ref{K:SFDE:prop5} holds.
Subsequently, the properties~\ref{K:SFDE:prop6}--\ref{K:SFDE:prop8} follow from direct calculations,
 see also~\cite{Droniou+Imbert:2006} for the special case of fractional Laplacian $D^\alpha_0$
 and~\cite{DebbiDozzi} for the general case $\alpha\in(1,\infty)\backslash\N$.
To prove~\ref{K:SFDE:prop9},
 we consider the basic definition of the derivative as the limit of a finite difference.
A function $f:\R\to\R$ has a derivative $f'(x)$ at $x$ if and only if
 $\frac1\epsilon (f-\tau_\epsilon f)$ has a limit as $\epsilon\to 0$,
 where $(\tau_\epsilon f)(x)=f(x+\epsilon)$.
Moreover for $t>t_0>0$, $\Green(\cdot ,t)\in C^\infty_b(\R) \cap W^{\infty,1}(\R)$ due to the estimate~\ref{K:SFDE:prop6}.
Thus $\frac1\epsilon (\Green(\cdot ,t)-\tau_\epsilon \Green(\cdot ,t))$ converges uniformly to $\diff{\Green}{x}(\cdot ,t)$, 
 i.e. with respect to the norm $\norm{.}_{L^\infty(\R)}$.
This fact and the Dominated Convergence Theorem imply that,
 $\frac1\epsilon (\Green(\cdot ,t)-\tau_\epsilon \Green(\cdot ,t)) \ast u$ converges uniformly to $\diff{\Green}{x}(\cdot ,t)\ast u$, too.
Finally, for $h(\cdot ,t):=\Green(\cdot ,t)\ast u$,
 the identity $\frac1\epsilon (\Green(\cdot ,t)-\tau_\epsilon \Green(\cdot ,t)) \ast u = \frac1\epsilon (h(\cdot ,t)-\tau_\epsilon h(\cdot ,t))$
 implies that the derivative $\diff{h}{x}(\cdot ,t)$ exists and is equal to $\diff{\Green}{x}(\cdot ,t)\ast u$.
A mathematical induction on the order of the derivative proves the general statement.   
Due to~\ref{K:SFDE:prop7} and a result by Sharpe~\cite{Sharpe:1969}, 
 the support of $\Green(\cdot ,t)$ is either all of $\R$ or a half-line for each $t>0$~\cite[Remark 28.8]{Sato:1999}.
Indeed only the strictly $\alpha$-stable probability densities with $0<\alpha<1$ and $\theta=-\alpha$ or $\theta=\alpha$
 have support on $(-\infty,0]$ and $[0,\infty)$, respectively;
all others have support $\R$~\cite[Property 1.2.14]{Samorodnitsky+Taqqu:1994}.
\end{proof}

Due to the properties of $\Green$,
 it is easy to show that $\RieszFeller$ generates a semigroup.

 \begin{prop} \label{prop:SFDE:semigroup}
  For $0<\alpha\leq 2$, $\abs{\theta} \leq \min\{\alpha,2-\alpha\}$ and $\abs{\theta}<1$,
  the Riesz-Feller operator~$\RieszFeller$ generates a strongly continuous, convolution semigroup 
  \benn
   S_t: L^p(\R) \to L^p(\R)\,, \quad u_0 \mapsto S_t u_0 = \Green(\cdot,t)\ast u_0 \,, 
  \eenn
  Moreover, the semigroup satisfies the dispersion property for $u\in L^1(\R)$
  \benn
    \norm{S_t u}_{L^p(\R)} \leq C_p\; t^{\frac{1-p}{\alpha p}} \norm{u}_{L^1(\R)} \quad 
  \eenn
  for all $1\leq p<\infty$ and some $C_p>0$.
 \end{prop}
\begin{proof}
The assumption $\abs{\theta}<1$ has to be made to exclude the cases $(\alpha,\theta)\in\{(1,1),(1,-1)\}$.
Due to Lemma~\ref{lem:SSPM},
 the probability measure $\Green$ is absolutely continuous with respect to the Lebesgue measure
 and possesses a probability distribution density which will be denoted again by $\Green$.
Thus~\ref{K:SFDE:prop2} and Young's inequality for convolutions imply
 $\norm{S_t u}_{L^p} \leq \norm{\Green(\cdot ,t)}_{L^1} \norm{u}_{L^p} = \norm{u}_{L^p}$ for all $u\in L^p(\R^n)$.
Therefore $S_t:L^p(\R)\ra L^p(\R)$ are well-defined bounded linear operators for all~$t\geq 0$. 
$\SG$ is a semigroup,
 since $S_{t+s}=S_t S_s$ for all $s,t\geq 0$ holds due to \ref{K:SFDE:prop3}.
Whereas the formal definition $S_0=\Id$ is justified,
 since~\ref{K:SFDE:prop1} and a standard result about convolutions~\cite[p.64]{LiebLoss} yield strong continuity of~$\SG$.
The dispersion property
 \benn
   \forall 1\leq p<\infty \quad \exists C_p>0 \,: \qquad 
   \norm{S_t u}_{L^p(\R)} \leq C_p\, t^{\frac{1-p}{\alpha p}} \norm{u}_{L^1(\R)} \quad
   \forall u\in L^1(\R)
 \eenn
 can be proved using \ref{K:SFDE:prop4} and Young's inequality \cite[p.98-99]{LiebLoss}.
\end{proof}
\begin{remark}
In addition, the semigroup $\SG$ is positivity preserving,
 since~\ref{K:SFDE:prop0} implies for $f\in L^p$ with $f\geq 0$ a.e. that $S_t f\geq 0$ a.e. for all $t\geq 0$.
Moreover, for $f\in L^p$ with $0\leq f\leq 1$ a.e. follows 
\[ 0\leq S_t f = \Green(\cdot ,t) \ast f \leq \norm{f}_{L^\infty} \norm{\Green(\cdot ,t)}_{L^1} = 1 \quad \text{a.e.} \]
Thus the semigroup $\SG$ is sub-Markovian and conservative,
 see also Definition~\ref{def:MarkovSG},
 hence it is an $L^p$-Markov semigroup for all $1\leq p<\infty$.
% As indicated in the proof, 
%  each Riesz-Feller operator generates a Markov semigroup, 
%  see Definition~\ref{def:MarkovSG} and Proposition~\ref{prop:MarkovSG} for more details.
\end{remark}

%% file: GLE-RieszFeller-Cb.tex
We are interested in traveling wave solutions 
 which will be $C^2_b(\R)$ functions in space.
Therefore we are going to derive an extension for the nonlocal operators $\RieszFeller$
 such that $\RieszFeller:C^2_b\to C_b$ and it generates a semigroup on $C_b$.
To deduce these properties,
 we will identify $\RieszFeller$ as the generator of a stochastic L\'evy process
 and use standard results from probability theory,
 which we collected in Appendix~\ref{ap:LevyProcesses+Semigroups}.

In Lemma~\ref{lem:SSPM}, we identified $\Green$ as L\'{e}vy strictly $\alpha$-stable distributions,
 which are a special case of infinitely divisible distributions.
For every infinitely divisible distribution $\mu$ on $\R^d$, such as $\Green$,
 there exists an associated L\'evy process $(X_t)_{t\geq 0}$.
In particular, every L\'evy process exhibits an associated strongly continuous semigroup on $C_0(\R^d)$,
 see also Theorem~\ref{thm:IG}.

% In Lemma~\ref{lem:SSPM}, we identified $\Green$ as L\'{e}vy strictly $\alpha$-stable distributions,
%  which are a special case of infinitely divisible distributions.
% For every infinitely divisible distribution $\mu$ on $\R^d$, such as $\Green$,
%  there exists a L\'evy process $(X_t)_{t\geq 0}$
%  such that $P_{X_1}=\mu$ and it is unique up to identity in law~\cite[Corollary 11.6]{Sato:1999}.
% This L\'evy process $(X_t)_{t\geq 0}$ is called the L\'evy process corresponding to $\mu$.
% The generating triplet $(A,\nu,\gamma)$ of $\mu$ is called
%  the generating triplet of the L\'evy process $(X_t)_{t\geq 0}$.
% In particular, every L\'evy process exhibits an associated strongly continuous semigroup on $C_0(\R^d)$,
%  see also Theorem~\ref{thm:IG}.

\subsubsection{A Representation Formula}
\label{ssec:representation}

%Although one may use the definition \eqref{eq:RF:transform} as a starting point,
%it is often more convenient to work with a more explicit integral representation formula. 
To study the traveling wave problem, 
 it is necessary to extend the nonlocal operator to $C^2_b(\R)$.
The following integral representations may be used to accomplish this task.
\begin{thm} \label{thm:RieszFeller:extension}
If $0<\alpha<1$ or $1<\alpha<2$ and $|\theta|\leq \min\{\alpha,2-\alpha\}$,
 then for all $f\in\SchwartzTF(\R)$ and $x\in\R$ 
\begin{multline} \label{eq:RieszFeller1}
\RieszFeller f(x) = \frac{c_1 - c_2}{1-\alpha} f'(x) +
  c_1 \integrall{0}{\infty}{ \frac{f(x+\xi)-f(x)-f'(x)\,\xi 1_{(-1,1)}(\xi)}{\xi^{1+\alpha}} }{\xi} \\
  + c_2 \integrall{0}{\infty}{ \frac{f(x-\xi)-f(x)+f'(x)\,\xi 1_{(-1,1)}(\xi)}{\xi^{1+\alpha}} }{\xi}
\end{multline}
for some constants $c_1, c_2 \geq 0$ with $c_1+c_2 >0$.
Alternative representations in special cases are
\begin{enumerate}
\item If $0<\alpha<1$, then 
 \begin{equation} \label{eq:RieszFeller1a}
 \RieszFeller f(x) = c_1 \integrall{0}{\infty}{ \frac{f(x+\xi)-f(x)}{\xi^{1+\alpha}} }{\xi}
  + c_2 \integrall{0}{\infty}{ \frac{f(x-\xi)-f(x)}{\xi^{1+\alpha}} }{\xi} \,.
 \end{equation}
\item If $1<\alpha<2$, then 
 \begin{multline} \label{eq:RieszFeller1b}
 \RieszFeller f(x) = c_1 \integrall{0}{\infty}{ \frac{f(x+\xi)-f(x)-f'(x)\,\xi}{\xi^{1+\alpha}} }{\xi}
  + c_2 \integrall{0}{\infty}{ \frac{f(x-\xi)-f(x)+f'(x)\,\xi}{\xi^{1+\alpha}} }{\xi} \,.
 \end{multline}
\end{enumerate}
\end{thm}
\begin{proof}
Due to Theorem~\cite[Theorem 14.19]{Sato:1999}, 
 the function $e^{t\psi_\theta^\alpha(\xi)}$ is the characteristic function 
 of a random variable with L\'{e}vy strictly $\alpha$-stable distribution.
Thus $\Green$ is the scaled probability measure of a L\'{e}vy strictly $\alpha$-stable distribution
 with generating triplet $(A,\nu,\gamma)=(0,\nu,\gamma)$.
Moreover $\nu$ is an absolutely continuous L\'evy measure~\eqref{eq:LevyMeasure1D}
 with constants  $c_1, c_2 \geq 0$ such that $c_1+c_2 >0$,
 see also Remark~\ref{remark:LevyMeasure1D}.
The Green functions $\Green$ are L\'{e}vy strictly $\alpha$-stable distributions due to Lemma~\ref{lem:SSPM},
 hence there exists a L\'evy process $(X_t)_{t\geq 0}$ (such that $P_{X_1}=\mu$),
 which is unique up to identity in law~\cite[Corollary 11.6]{Sato:1999}.
Due to Theorem~\ref{thm:IG}, the infinitesimal generator $L$ of the associated transition semigroup 
 has a representation~\eqref{eq:RieszFeller1}. 

In Remark~\ref{remark:AlternativeRepresentation}, 
 alternative representations - $(A,\nu,\gamma_0)_0$ and $(A,\nu,\gamma_1)_1$ - of the L\'evy--Khintchine formula are discussed.
Indeed, if $0<\alpha<1$,
 then the L\'evy measure~\eqref{eq:LevyMeasure1D} satisfies condition~\eqref{eq:condition:0},
 hence the characteristic function has a representation~\eqref{eq:LKF:0}
 with generating triplet $(A,\nu,\gamma_0)_0=(0,\nu,\gamma_0)$.
Due to~\cite[Theorem 14.7]{Sato:1999}, 
 a strictly $\alpha$-stable distribution for $0<\alpha<1$ satisfies $\gamma_0=0$
 which yields the representation~\eqref{eq:RieszFeller1a}.
If $1<\alpha<2$,
 then the L\'evy measure~\eqref{eq:LevyMeasure1D} satisfies condition~\eqref{eq:condition:1},
 hence the characteristic function has a representation~\eqref{eq:LKF:1} with generating triplet $(A,\nu,\gamma_1)_1=(0,\nu,\gamma_1)$.
Due to~\cite[Theorem 14.7]{Sato:1999}, 
 a strictly $\alpha$-stable distribution for $1<\alpha<2$ satisfies $\gamma_1=0$
 which yields the representation~\eqref{eq:RieszFeller1b}. 
Following Remark~\ref{remark:AlternativeRepresentation},
 $\gamma$ is determined in both cases 
 \[ \begin{cases} 
     0 = \gamma_0 = \gamma - \int_{(-1,1)}\, x \, \nu(\dx) = \gamma - \frac{c_1 - c_2}{1-\alpha} &\text{for } 0<\alpha<1\,, \\
     0 = \gamma_1 = \gamma + \int_{\R\backslash(-1,1)}\, x\, \nu(\dx) = \gamma + \frac{c_1 - c_2}{\alpha-1} &\text{for } 1<\alpha<2\,,
    \end{cases}
 \]
 as $\gamma=\frac{c_1 - c_2}{1-\alpha}$ which yields representation~\eqref{eq:RieszFeller1}. 
\end{proof}
\begin{remark} \label{rem:c1+c2}
Mainardi, Luchko and Pagnini~\cite{MainardiLuchkoPagnini} give the values of the coefficients as
\[ c_1=\frac{\Gamma(1+\alpha) \sin ( (\alpha+\theta) \frac{\pi}{2} )}{\pi} 
    \quad \text{and} \quad
    c_2=\frac{\Gamma(1+\alpha) \sin( (\alpha-\theta) \frac{\pi}{2} )}{\pi}\,. \]
However, their integral representation~\cite[(2.8)]{MainardiLuchkoPagnini} in the range $1<\alpha<2$
 has to be understood as the principal value of the integral. 
\end{remark}

These representations allow to extend the $\RieszFeller$ operator to $C^2_b(\R)$ such that $\RieszFeller C^2_b(\R)\subset C_b(\R)$.
\begin{prop} \label{prop:RieszFeller:estimate}
The integral representation~\eqref{eq:RieszFeller1b} of $\RieszFeller$ with $1<\alpha<2$ and 
$|\theta|\leq \min\{\alpha,2-\alpha\}$ is well-defined for functions $f\in C^2_b(\R)$ with
 \begin{equation} 
 \label{eq:estimate:RieszFeller}
  \sup_{x\in\R} |\RieszFeller f(x)|
      \leq \cK \|{f''}\|_{C_b(\R)} \frac{M^{2-\alpha}}{2-\alpha}
      + 4 \cK \|f'\|_{C_b(\R)} \frac{M^{1-\alpha}}{\alpha-1} < \infty
 \end{equation}
 for some positive constants $M$
 and $\cK=\frac{\Gamma(1+\alpha)}{\pi} \abs{ \sin ( (\alpha+\theta) \frac{\pi}{2} ) + \sin( (\alpha-\theta) \frac{\pi}{2} ) }$.
\end{prop}
\begin{proof}
 We consider the two summands in~\eqref{eq:RieszFeller1b} separately,
  starting with $\integrall{0}{\infty}{ \frac{f(x+\xi)-f(x)-f'(x)\xi}{\xi^{1+\alpha}} }{\xi}$ for any $f\in C^2_b(\R)$.
 The goal is to obtain an upper bound.
 Choose $M>0$ and consider 
 \begin{multline*}
   \integrall{0}{\infty}{ \frac{f(x+\xi)-f(x)-f'(x)\xi}{\xi^{1+\alpha}} }{\xi} = \\
    \integrall{0}{M}{ \frac{f(x+\xi)-f(x)-f'(x)\xi}{\xi^{1+\alpha}} }{\xi} 
    + \integrall{M}{\infty}{ \frac{f(x+\xi)-f(x)-f'(x)\xi}{\xi^{1+\alpha}} }{\xi} \,. 
 \end{multline*}
 The first integral is written as
 \begin{align*}
   \integrall{0}{M}{ \frac{f(x+\xi)-f(x)-f'(x)\xi}{\xi^{1+\alpha}} }{\xi} 
     &= \integrall{0}{M}{ \frac1{\xi^{1+\alpha}} \bigg[ \integrall{0}{1}{f'(x+\theta\xi)\, \xi}{\theta} -f'(x)\xi \bigg] }{\xi} \\
     &= \integrall{0}{M}{ \frac{\xi}{\xi^{1+\alpha}} \bigg[ \integrall{0}{1}{ \integrall{0}{1}{ f''(x+s\theta\xi)\, \theta\xi}{s} }{\theta} \bigg] }{\xi} \\
     &= \integrall{0}{M}{ \frac{\xi^2}{\xi^{1+\alpha}} \underbrace{\bigg[ \integrall{0}{1}{ \integrall{0}{1}{ f''(x+s\theta\xi)\, \theta}{s} }{\theta} \bigg]}_{\text{bounded by $\norm{f''}_{C_b}$}} }{\xi}
 \end{align*}
 where we use the shorthand notation $\|\cdot\|_{C_b}=\|\cdot\|_{C_b(\R)}$. Thus
 \benn
   \Abs{ \integrall{0}{M}{ \frac{f(x+\xi)-f(x)-f'(x)\xi}{\xi^{1+\alpha}} }{\xi} } 
     \leq \tfrac12 \norm{f''}_{C_b} \integrall{0}{M}{ \xi^{1-\alpha} }{\xi} 
     %= C \norm{f''}_{C_b} \, \frac{\xi^{2-\alpha}}{2-\alpha} \bigg|_{\xi=0}^{\xi=M} 
     = \tfrac12 \norm{f''}_{C_b} \frac{M^{2-\alpha}}{2-\alpha}  \,. 
 \eenn
 The second integral is written as 
 \begin{align*}
   \integrall{M}{\I}{ \frac{f(x+\xi)-f(x)-f'(x)\xi}{\xi^{1+\alpha}} }{\xi} 
     &= \integrall{M}{\I}{ \frac1{\xi^{1+\alpha}} \bigg[ \integrall{0}{1}{f'(x+\theta\xi)\, \xi}{\theta} -f'(x)\xi \bigg] }{\xi} \\
     &= \integrall{M}{\I}{ \frac{\xi}{\xi^{1+\alpha}} \underbrace{\bigg[ \integrall{0}{1}{f'(x+\theta\xi) -f'(x)}{\theta} \bigg]}_{\text{bounded by $2 \norm{f'}_{C_b}$}} }{\xi} 
 \end{align*}
 Thus 
 \benn
   \Abs{ \integrall{M}{\infty}{ \frac{f(x+\xi)-f(x)-f'(x)\xi}{\xi^{1+\alpha}} }{\xi} } 
     \leq 2 \norm{f'}_{C_b} \integrall{M}{\I}{ \xi^{-\alpha} }{\xi} 
     = 2 \norm{f'}_{C_b} \frac{M^{1-\alpha}}{\alpha-1} \,.
 \eenn
 Summarizing we estimate
 \[ \Abs{\integrall{0}{\infty}{ \frac{f(x+\xi)-f(x)-f'(x)\xi}{\xi^{1+\alpha}} }{\xi} }
      \leq \tfrac12 \norm{f''}_{C_b} \frac{M^{2-\alpha}}{2-\alpha} + 2 \norm{f'}_{C_b} \frac{M^{1-\alpha}}{\alpha-1} < \infty
 \]
 and similarly
 \[ \Abs{\integrall{0}{\infty}{ \frac{f(x-\xi)-f(x)+f'(x)\xi}{\xi^{1+\alpha}} }{\xi} }
      \leq \tfrac12 \norm{f''}_{C_b} \frac{M^{2-\alpha}}{2-\alpha} + 2 \norm{f'}_{C_b} \frac{M^{1-\alpha}}{\alpha-1} < \infty
 \]
 for some $M>0$.
 Consequently the integral representation~\eqref{eq:RieszFeller1b} of $\RieszFeller$ satisfies estimate~\eqref{eq:estimate:RieszFeller}, where we use the expressions given in Remark~\ref{rem:c1+c2} to determine~$\cK$.  
\end{proof}

The estimate \eqref{eq:estimate:RieszFeller} shows that for $1<\alpha<2$ there exists a bound for 
$\RieszFeller$ involving first and second derivatives. This is one key estimate we are going to 
use to adapt the assumptions~\ref{As:Chen:B3} and~\ref{As:Chen:C3} discussed in Section~\ref{sec:Chen}.
\medskip

For a self-contained derivation of the representation of fractional Laplacians $D^\alpha_0$, $0<\alpha<2$, 
 see the work of Droniou and Imbert~\cite[Theorem 1]{Droniou+Imbert:2006}.
Their results on continuity~\cite[Proposition 1]{Droniou+Imbert:2006} and on sequences~\cite[Theorem 2]{Droniou+Imbert:2006} generalize 
 to Riesz-Feller operators with obvious modifications in their proofs.
\begin{prop} 
Let $0<\alpha<1$ or $1<\alpha<2$, $|\theta|\leq \min\{\alpha,2-\alpha\}$ and $f\in C^2_b(\R)$.
If $(f_n)_{n\geq 1} \in C^2_b(\R)$ is bounded in $L^\infty(\R)$ and $D^2 f_n \to D^2 f$ locally uniformly on $\R$,
 then $\RieszFeller f_n \to \RieszFeller f$ locally uniformly on $\R$.
\end{prop}

\begin{thm} \label{thm:RieszFeller:globalMax}
  Let $0<\alpha<1$ or $1<\alpha<2$, $|\theta|\leq \min\{\alpha,2-\alpha\}$ and $f\in C^2_b(\R)$.
  If $(x_k)_{k\in\N}$ is a sequence in $\R^n$ such that $f(x_k)\to \sup_{\R^n} f$ as $k\to\infty$,
   then $\lim_{k\to\infty} \nabla f(x_k) = 0$ and $\liminf_{k\to\infty} \RieszFeller[f](x_k)\geq 0$.
\end{thm}

\subsubsection{Semigroup Properties}

In particular, a non-degenerate Riesz-Feller operator generates a strongly continuous convolution semigroup on $C_0(\R)$,
 which can be extended to a convolution semigroup on $L^\infty(\R)$.
\begin{thm} \label{thm:Bb:semigroup}
For $1<\alpha< 2$ and $\abs{\theta} \leq \min\{\alpha,2-\alpha\}$,
 the Riesz-Feller operator~$\RieszFeller$ generates a convolution semigroup 
 $S_t: L^\infty(\R) \to L^\infty(\R)$, $u_0 \mapsto S_t u_0 = \Green(\cdot ,t)\ast u_0$,
 with kernel $\Green(x,t)$.
Moreover, the convolution semigroup with $u(x,t):=S_t u_0$ satisfies
\begin{enumerate}
\item $u\in C^\infty(\R\times(t_0,\infty))$ for all $t_0>0$;
\item $\frac{\partial u}{\partial t} = \RieszFeller u$ for all $(x,t)\in\R\times (t_0,\I)$ and any $t_0>0$;
\item $u(\cdot ,t)\stackrel{*}{\rightharpoonup} u_0$ for $t\searrow 0$ in the weak-$*$ topology of $L^\infty(\R)$;
\item If $u_0\in C_b(\R)$
  then $\lim_{\R\times (0,\infty)\ni(x,t)\to (x_0,0)} u(x,t)=u_0(x_0)$ for each $x_0\in\R$.
\end{enumerate}
\end{thm}
\begin{proof}
Due to the assumptions and Lemma~\ref{lem:SSPM},
 the kernel is a smooth probability density function with $\Green(\cdot ,t)\in L^1(\R)$.
This observation and Young's inequality for convolutions show that,
 $S_t:L^\infty(\R)\to L^\infty(\R)$ are well-defined bounded linear operators.
We define $S_0=\Id$ 
 and the semigroup property follows from property~\ref{K:SFDE:prop3} in Lemma~\ref{lem:SSPM}.
The semigroup $(S_t)_{t\geq 0}$ of bounded linear operators on $L^\infty(\R)$ is not necessarily strongly continuous,
 see also~\cite[page 427 ff.]{Jacob:2001}.
However $S_t u_0$ converges for $t\searrow 0$ in the weak-$*$ topology of $L^\infty(\R)$,
 see also~\cite{Tartar:2007}.

The function $u$ is smooth, since $u$ is a convolution of $u_0\in L^\infty(\R)$ with an integrable smooth function $\Green$
 having bounded integrable derivatives~\ref{K:SFDE:prop5}--\ref{K:SFDE:prop7}.
Furthermore, $u$ is a solution of~\eqref{eq:linearRF},
 since $\Green$ is a solution of~\eqref{eq:linearRF} for positive times.
Finally, $\Green$ is an approximate unit with respect to $t$ due to~\ref{K:SFDE:prop0}--\ref{K:SFDE:prop2}
 which is sufficient for the stated convergence to the initial datum $u_0$.  
\end{proof}

In the analysis of the traveling wave problem, 
 we are mostly interested in the evolution of initial data in $C_b$.
Therefore, it is important to notice the following corollary.
\begin{cor} \label{cor:Cb:semigroup}
For $1<\alpha< 2$ and $\abs{\theta} \leq \min\{\alpha,2-\alpha\}$,
 the Riesz-Feller operator~$\RieszFeller$ generates a convolution semigroup 
 $S_t: C_b(\R) \to C_b(\R)$, $u_0 \mapsto S_t u_0 = \Green(\cdot ,t)\ast u_0$,
 with kernel $\Green(x,t)$.
Moreover, the convolution semigroup with $u(x,t):=S_t u_0$ satisfies
\begin{enumerate}
\item $u\in C^\infty(\R\times(t_0,\infty))$ for all $t_0>0$;
\item $\frac{\partial u}{\partial t} = \RieszFeller u$ for all $(x,t)\in\R\times (t_0,\I)$ and any $t_0>0$;
\item If $u_0\in C_b(\R)$ then $u\in C_b(\R\times[0,T])$ for any $T>0$.
\end{enumerate}
\end{cor}
Since $S_t: C_b(\R) \to C_b(\R)$ is not a strongly continuous semigroup,
 the relation between the $C_b$-extension of the strongly continuous semigroup $(S_t)_{t\geq 0}$ on $C_0(\R)$
 and the $C^2_b$-extension of the Fourier multiplier operators $\RieszFeller$ is not obvious.
In Appendix~\ref{ap:LevyProcesses+Semigroups} we discuss this relationship in more detail.   
% In particular, the Riesz-Feller operators $\RieszFeller$ for $0<\alpha\leq 2$ and $\abs{\theta}\leq \min\{\alpha,2-\alpha\}$
%  generate conservative $C_b$-Feller semigroups. 
% This can be deduced from a criterion on the symbol of Fourier multiplier operators in~\cite{Schilling:1998}.

%% file: GLE-CauchyProblem.tex
We consider the Cauchy problem
\begin{equation} \label{CP:RD}
\begin{cases}
  \diff{u}{t} = \RieszFeller u + f(u) 	& \xx{for} (x,t)\in\R\times (0,\infty) \,, \\
  u(x,0) = u_0(x) 			& \xx{for} x\in\R \,,
\end{cases}
\end{equation}
for $1<\alpha\leq 2$, $|\theta| \leq \min\{\alpha, 2-\alpha\}$ and $f\in C^\infty(\R)$ satisfying~\eqref{As:f:0}.
% for an initial datum $u_0$. % with the help of semigroup theory~\cite{Henry:1981}.
We follow a standard approach, 
 and consider the Cauchy problem in its mild formulation
 to prove the existence of a mild solution.
The Cauchy problem generates a nonlinear semigroup  
 which allows to prove uniform $C^k_b$ estimates via a bootstrap argument
 and to conclude that mild solutions are also classical solutions.

In particular, Droniou and Imbert~\cite{Droniou+Imbert:2006} studied partial integro-differential equations
 \[ \diff{u}{t}(x,t) = \Riesz[u(\cdot ,t)](x) + F(t,x,u(x,t),\nabla u(x,t)) \Xx{for} x\in\R^n\,, t>0\,, \]
 involving the fractional Laplacian $\Riesz$ for some $0<\alpha<2$. 
First they introduce the fractional Laplacian $\Riesz$ as a Fourier multiplier operator on the Schwartz class $\SchwartzTF(\R)$,
 and then they extend it to $C^2_b(\R)$ functions in~\cite[Lemma 2]{Droniou+Imbert:2006}.

Next, we summarize relevant results from~\cite[Section 3]{Droniou+Imbert:2006}:
Under the assumption $\alpha\in (1,2)$, they consider the Cauchy problem
\begin{equation} \label{IVP:fractalHJequation}
 \begin{cases}
  \diff{u}{t}(x,t) = \Riesz[u(\cdot ,t)](x) + F(t,x,u,\nabla u) &\Xx{for} x\in\R^n\,, t>0\,, \\
  u(x,0) = u_0(x) &\Xx{for} x\in\R^n\,, 
 \end{cases}
\end{equation}
where $u_0\in W^{1,\infty}(\R^n)$ and $F\in C^\infty([0,\infty)\times \R^n\times \R\times \R^n)$ satisfies
\begin{enumerate}[label=(DI\arabic*)]
 \item \label{As:DI:16}
  $\forall T>0$, $\forall R>0$, $\forall k\in\N$, $\exists \cK_{T,R,k}$
  such that $\forall (t,x,v,\nu)\in [0,T]\times \R^n\times [-R,R]\times B_R$
  where $B_R:=\set{\nu\in\R^n}{\norm{\nu}<R}$
  and for all multi-indices $\beta\in \N^{2n+2}$ satisfying $\abs{\beta}\leq k$,
  \[ \abs{\partial^\beta F(t,x,v,\nu)}\leq \cK_{T,R,k} \,. \]
 \item \label{As:DI:17}
  $\forall T>0$, $\exists \Lambda_T:[0,\infty)\mapsto (0,\infty)$ continuous and non-decreasing
  such that $\integrall{0}{\infty}{ \frac{1}{\Lambda_T(a)} }{a} = \infty$
  and $\forall (t,x,v)\in [0,T]\times \R^n\times \R$,
  \[ \sgn(v) F(t,x,v,0) \leq \Lambda_T(\abs{v}) \,. \]
 \item \label{As:DI:18}
  $\forall T>0$, $\forall R>0$, $\exists \Gamma_{T,R}:[0,\infty)\mapsto (0,\infty)$ continuous and non-decreasing
  such that $\integrall{0}{\infty}{ \frac{1}{\Gamma_{T,R}(a)} }{a} = \infty$
  and $\forall (t,x,v,\nu)\in [0,T]\times \R^n\times [-R,R]\times \R^n$
  \[ \abs{\nu} \partial_v F(t,x,v,\nu) \leq \Gamma_{T,R}(\abs{\nu})\,, \quad
       \abs{\nabla_x F(t,x,v,\nu)} \leq \Gamma_{T,R}(\abs{\nu}) \,.
  \]
\end{enumerate}
The functions 
\begin{equation} \label{LT+GTR} 
    \mathcal{L}_T(a)=\integrall{0}{a}{ \frac{1}{\Lambda_T(b)} }{b} \XX{and}
    \mathcal{G}_{T,R}(a)=\integrall{0}{a}{ \frac{1}{2n \Gamma_{T,R}(b)} }{b}
\end{equation}
are defined, which are non-decreasing $C^1$-diffeomorphisms from $[0,\infty)$ to $[0,\infty)$,
due to the assumptions on $\Lambda_T$ and $\Gamma_{T,R}$.
\begin{thm}[{\cite[Theorem 3]{Droniou+Imbert:2006}}] \label{thm:CP:Solution:0}
 Let $\alpha\in(1,2)$, $u_0\in W^{1,\infty}(\R^n)$ and $F$ satisfy~\ref{As:DI:16}--\ref{As:DI:18}.
 There exists a unique solution of~\eqref{IVP:fractalHJequation} in the following sense: for all $T>0$
 \begin{enumerate}[label=(DI\arabic*)]
  \setcounter{enumi}{3}
  \item %\label{As:DI:20}
    $u\in C_b(\R^n \times (0,T))$, $\nabla u\in C_b(\R^n \times (0,T))^n$, and 
    for all $a\in(0,T)$ $u\in C^\infty_b(\R^n \times (a,T))$;
  \item %\label{As:DI:21}
    $u$ satisfies the partial integro-differential equation~\eqref{IVP:fractalHJequation} on $\R^n \times (0,T)$,
  \item %\label{As:DI:22}
    $u(\cdot ,t)\to u_0$ uniformly on $\R^n$ as $t\to 0$.
 \end{enumerate}
 There are also the following estimates on the solution:
 for all $0<t<T<\infty$,
 \begin{enumerate}[label=(DI\arabic*)]
  \setcounter{enumi}{6}
  \item %\label{As:DI:23}
    $\norm{u(\cdot ,t)}_{L^\infty(\R^n)} \leq (\mathcal{L}_T)^{-1} \big(t+\mathcal{L}_T(\norm{u_0}_{L^\infty(\R^n)})\big)$,
  \item \label{As:DI:24}
    $\norm{Du(\cdot ,t)}_{L^\infty(\R^n)} \leq (\mathcal{G}_{T,R})^{-1} \big(t+\mathcal{G}_{T,R}(\norm{D u_0}_{L^\infty(\R^n)})\big)$,
 \end{enumerate}
 where $\mathcal{L}_T$ and $\mathcal{G}_{T,R}$ are defined by~\eqref{LT+GTR},
 \[ \norm{Du(\cdot ,t)}_{L^\infty(\R^n)} = \sum_{i=1}^n \Norm{\diff{u}{x_i} (\cdot ,t)}_{L^\infty(\R^n)} \]
 and $R$ is any upper bound of $\norm{u}_{L^\infty(\R^n \times (0,T))}$.
\end{thm}

A smooth function $F=F(u)$ that depends only on $u$ and satisfies~\ref{As:DI:16} also satisfies~\ref{As:DI:18}.
In this case a simplified proof of Theorem~\ref{thm:CP:Solution:0} 
 allows to show the existence of a solution for the IVP with $u_0\in L^\infty(\R)$.
\begin{thm} \label{thm:CP:Solution}
 Let $\alpha\in(1,2)$, $u_0\in L^{\infty}(\R^n)$ and $F=F(u)$ satisfy~\ref{As:DI:16} and \ref{As:DI:17}.
 There exists a unique solution of~\eqref{IVP:fractalHJequation} in the following sense: for all $T>0$
 \begin{enumerate}[label=(DI\arabic*')]
  \setcounter{enumi}{3}
  \item \label{As:DI:20}
    $u\in C_b(\R^n \times (0,T))$ and 
    for all $a\in(0,T)$ $u\in C^\infty_b(\R^n \times (a,T))$;
  \item \label{As:DI:21}
    $u$ satisfies the partial integro-differential equation~\eqref{IVP:fractalHJequation} on $\R^n \times (0,T)$,
  \item \label{As:DI:22}
    If $u_0\in C_b(\R)$ then $u(\cdot ,t)\to u_0$ uniformly on $\R^n$ as $t\to 0$.
 \end{enumerate}
 There are also the following estimates on the solution:
 for all $0<t<T<\infty$,
 \begin{enumerate}[label=(DI\arabic*)]
  \setcounter{enumi}{6}
  \item \label{As:DI:23}
    $\norm{u(\cdot ,t)}_{L^\infty(\R^n)} \leq (\mathcal{L}_T)^{-1} \big(t+\mathcal{L}_T(\norm{u_0}_{L^\infty(\R^n)})\big)$,
 \end{enumerate}
 where $\mathcal{L}_T$ is defined by~\eqref{LT+GTR},
 and $R$ is any upper bound of $\norm{u}_{L^\infty(\R^n \times (0,T))}$.
\end{thm}

%\paragraph{extensions}
For our purposes we need
 to extend the result of Theorem~\ref{thm:CP:Solution} to the case of all Riesz-Feller operators~$\RieszFeller$ in~\eqref{CP:RD}
 and to adapt the result to admissible functions $f$ which do not satisfy the growth condition~\ref{As:DI:17}
 see also Figure~\ref{fig:f:modification}.

First, Droniou and Imbert note in~\cite[Remark 5]{Droniou+Imbert:2006}
 that their proof of Theorem~\ref{thm:CP:Solution:0} still applies
 if $\Riesz$ is replaced by more general operators
 which satisfy~\cite[Theorem 2]{Droniou+Imbert:2006}
 and whose associated kernel $K_\alpha(x,t)$ has the properties~\cite[(30)]{Droniou+Imbert:2006}
 \begin{enumerate}[label=(P\arabic*)]
  \item \label{As:P1} $K_\alpha\in C^\infty(\R^n \times (0,\infty))$ and $(K_\alpha(\cdot ,t))_{t\to 0}$ is an approximate unit
   (in particular, $K_\alpha\geq 0$ and, for all $t>0$, $\norm{K_\alpha(\cdot ,t)}_{L^1(\R^n)}=1$),
  \item \label{As:P2} $\forall t>0$, $\forall t'>0$, $K_\alpha(\cdot ,t+t')=K_\alpha(\cdot ,t)\ast K_\alpha(\cdot ,t')$,
  \item \label{As:P3} $\exists \mathcal K>0$, $\forall t>0$, $\norm{\nabla K_\alpha(\cdot ,t)}_{L^1(\R^n)}\leq \mathcal K t^{-1/\alpha}$,
 \end{enumerate}
 and~\cite[(59)]{Droniou+Imbert:2006}
 \begin{enumerate}[label=(P\arabic*)]
  \setcounter{enumi}{3}
  \item \label{As:P4} $(0,\infty)\ni t\mapsto K_\alpha(\cdot ,t)\in L^1(\R^n)$ is continuous.
 \end{enumerate}
The Riesz-Feller operators $\RieszFeller$ for $1<\alpha<2$ and $|\theta| \leq \min\{\alpha, 2-\alpha\}$
 satisfy the properties~\ref{As:P1}--\ref{As:P3}, due to Theorem~\ref{thm:RieszFeller:globalMax} and Lemma~\ref{lem:SSPM},
 whereat \ref{As:P4} follows from the regularity of $\Green$ and the scaling property~\ref{K:SFDE:prop1}.
Therefore the result of Theorem~\ref{thm:CP:Solution} still holds
 if the operator $\Riesz$ in~\eqref{IVP:fractalHJequation} is replaced by 
 a Riesz-Feller operator $\RieszFeller$ for $1<\alpha<2$ and $|\theta| \leq \min\{\alpha, 2-\alpha\}$.

Second, the prototype of a function $f$ satisfying assumption~\eqref{As:f:0} is a cubic polynomial of the form
 \[ f_1(u):= u\, (1-u)\, (u-a) \Xx{for some} a\in (0,1)\,, \]
 which satisfies $(\sgn v) f_1(v)\leq \max_{t\in [0,1]} f_1(t)$ for all $v\in\R$
 and hence assumption~\ref{As:DI:17} with a constant $\Lambda$.
\begin{figure}
 \begin{minipage}[b]{0.5\textwidth} 
% \begin{figure} \centering
  \begin{tikzpicture}[scale=1.5]
%  \draw[very thin,color=gray] (-1,-1) grid (2,1);
  \draw[->] (-1.5,0) -- (2.5,0) node[right] {$u$};
  \draw[->] (0,-1.5) -- (0,1.5); %node[above] {$f(u)$};
  \draw[-] (0.6,0) -- (0.6,-0.1) node[below] {$a$};
  \draw[-] (1,0) -- (1,-0.1) node[below] {$1$};

  \draw[very thick,domain=-0.7:1.7,color=blue] plot (\x,{-1*\x*(\x-0.6)*(\x-1)}) node[right] {$f_1(u)$}; %{$f_1(u)=u(1-u)(u-0.6)$};
%   \draw[thin,color=green,domain=-1.5:-0.5] plot (\x,{0.5*(-0.5-0.6)*(-0.5-1)});
%   \draw[thin,color=green,domain=-0.5:1.5] plot (\x,{-1*\x*(\x-0.6)*(\x-1)});
%   \draw[thin,color=green,domain=1.5:2.5] plot (\x,{-1.5*(1.5-0.6)*(1.5-1)}) node[right] {$\overline f_1(u)$};
  \end{tikzpicture}
% \end{figure}
 \end{minipage}
 \begin{minipage}[b]{0.5\textwidth} 
% \begin{figure} \centering
  \begin{tikzpicture}[scale=1.5]
%  \draw[very thin,color=gray] (-1,-1) grid (2,1);
  \draw[->] (-1.5,0) -- (2.5,0) node[right] {$u$};
  \draw[->] (0,-1.5) -- (0,1.5); %node[above] {$f(u)$};
  \draw[-] (0.6,0) -- (0.6,-0.1) node[below] {$a$};
  \draw[-] (1,0) -- (1,-0.1) node[below] {$1$};

  \draw[very thick,domain=-1.1:2.1,smooth,color=blue] plot (\x,{(\x+1)*\x*(\x-0.6)*(\x-1)*(\x-2)}) node[right] {$f_2(u)$}; %{$f_2(u)=(u+1)u(u-0.6)(u-1)(u-2)$};
%   \draw[thin,color=green,smooth,domain=-1.5:-0.5] plot (\x,{(0.5)*(-0.5)*(-1.1)*(-1.5)*(-2.5)});
%   \draw[thin,color=green,smooth,domain=-0.5:1.5] plot (\x,{(\x+1)*\x*(\x-0.6)*(\x-1)*(\x-2)});
%   \draw[thin,color=green,smooth,domain=1.5:2.5] plot (\x,{(2.5)*(1.5)*(0.9)*(0.5)*(-0.5)}) node[right] {$\overline f_2(u)$};
  \end{tikzpicture}
% \end{figure}
 \end{minipage}
 \caption{
  The functions $f_1(u)=u\, (1-u)\, (u-a)$ and $f_2(u)=(u+1)\, u\, (u-a)\, (u-1)\, (u-2)$ for any $a\in (0,1)$
   are bistable in the sense of~\eqref{As:f:0} and depicted in the left and the right figure, respectively.  
  Whereas $f_1$ satisfies the assumptions~\ref{As:DI:16}--\ref{As:DI:18},
   function $f_2$ does not satisfy~\ref{As:DI:17}.
%   However, we are interested in solutions taking values in $[0,1]$
%    and the partial integro-differential equation exhibits a comparison principle.
%   Therefore, we will modify any bistable function $f$ outside of $[0,1]$,
%    such that its modification $\overline f$ satisfies the assumptions~\ref{As:DI:16}--\ref{As:DI:18}.
  }
 \label{fig:f:modification}
\end{figure}
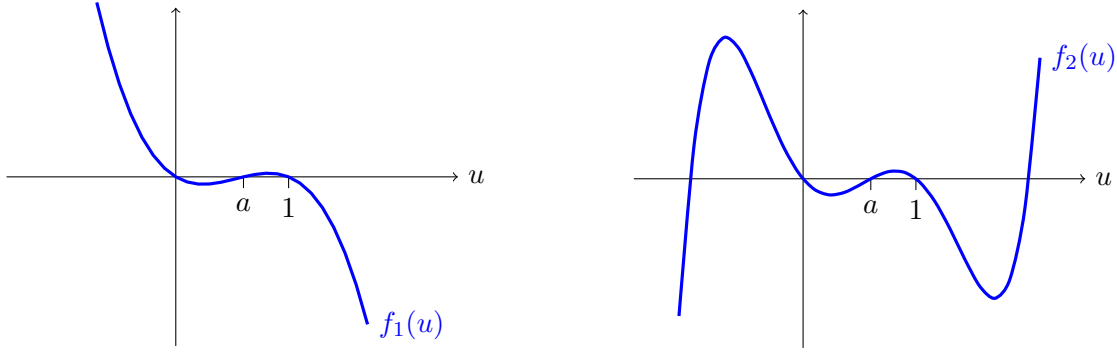
In contrast, other admissible function such as 
 \[ f_2(u):= (u+1)\, u\, (u-a)\, (u-1)\, (u-2) \Xx{for some} a\in (0,1) \]
 do not satisfy assumption~\ref{As:DI:17}.
 The estimate 
 \[ (\sgn v) f_2(v) \leq \Lambda_T(\abs{v})= c(\abs{v}+2)^5 \] 
 for some $c>0$, implies that
 \[
    \integrall{0}{\infty}{\frac1{\Lambda_T(u)}}{u} = \integrall{0}{\infty}{\frac1{c(u+2)^5}}{u} 
       = - \lim_{R\to\infty} \frac1{4c(u+2)^4} \bigg|_{u=0}^{u=R} = \frac1{c 2^6} <\infty \,.
 \]
However, we are interested in solutions taking values in $[0,1]$
 and the partial integro-differential equation exhibits a comparison principle see also Lemma~\ref{lem:comparison:Cb}.
Thus we will modify the function $f_2$ outside of $[0,1]$,
 such that it satisfies the assumptions~\ref{As:DI:16}--\ref{As:DI:18},
 see also Figure~\ref{fig:f:modification}.
Consequently (a generalization of) Theorem~\ref{thm:CP:Solution} applies to the associated Cauchy problem
 and the solution - taking values in $[0,1]$ -
 will be a solution of the original Cauchy problem~\eqref{CP:RD}.

\begin{thm} \label{thm:CP:existence:bistable}
Suppose $1<\alpha\leq 2$, $|\theta| \leq \min\{\alpha, 2-\alpha\}$ and $f\in C^\infty(\R)$ satisfies~\eqref{As:f:0}.
The Cauchy problem~\eqref{CP:RD} with initial condition $u(\cdot ,0)=u_0\in C_b(\R)$ and $0\leq u_0\leq 1$
 has a solution $u(x,t)$ in the sense of Theorem~\ref{thm:CP:Solution}
 satisfying $0\leq u(x,t) \leq 1$ for all $(x,t)\in\R\times(0,\infty)$.
 Moreover, for all $k\in\N$ and $t_0>0$ there exists a $\cK>0$ such that $\norm{u(\cdot ,t)}_{C^k_b(\R)}\leq \cK$ for all $0<t_0<t$. 
% \[ \forall k\in\N: \qquad \exists C>0: \qquad \forall 0<t_0 \Xx{and} \forall t_0<t \qquad \norm{u}_{C^k_b(\R)}\leq C\,. \] 
\end{thm}

\begin{proof}
The first assumption~\ref{As:DI:16} is satisfied,
 since $f$ is a smooth function,
 hence all derivatives are continuous and bounded on any compact interval $[-R,R]$.
The third assumption~\ref{As:DI:18} is satisfied,
 since $F(t,x,v,\nu)=f(v)$ implies 
 \[
    \abs{\nu} \diff{F}{v}(t,x,v,\nu) = \abs{\nu} f'(v) \leq \abs{\nu}\, \max_{s\in[-R,R]} \abs{f'(v)} =: \Gamma_{T,R}(\abs{\nu})\,,
      \qquad \Abs{\diff{F}{x}(t,x,v,\nu)} = 0 \leq \Gamma_{T,R}(\abs{\nu})\,,
 \]
 and
 \[ \integrall{0}{\infty}{\frac1{\Gamma_{T,R}(u)}}{u} = \integrall{0}{\infty}{\frac1{c\, u}}{u} 
       = \tfrac1c \lim_{\epsilon\to 0} \ln(u) \bigg|_{u=\epsilon}^{u=\frac1\epsilon} = \infty \,.
 \]
We are interested in solutions taking values in $[0,1]$.
Moreover, the partial integro-differential equation exhibits a comparison principle,
 such that classical solutions $u(x,t)$ of our Cauchy problem will satisfy $0\leq u\leq 1$.
Therefore, we can modify $f$
 in such a way that its modification $\tilde f$ satisfies assumption~\ref{As:DI:17}
 but does not change the dynamics as long as $u$ takes values in~$[0,1]$.
First we define $f_{min}:= \min_{u\in[0,1]} f(u)$, $f_{max}:= \max_{u\in[0,1]} f(u)$,
 and a bounded function $\overline f(u):= \max\{f_{min} \,, \min\{f(u)\,, f_{max}\}\}$.
Finally we consider a smooth function $\tilde f\in C^\infty(\R)$,
 such that $\tilde f(u)=\overline f(u)=f(u)$ for all $u\in[0,1]$
 and $\abs{\tilde f(u)}\leq \abs{\overline f(u)}$ for all $u\in\R$. 
Then, assumption~\ref{As:DI:17} holds for $\tilde f$, since
 $(\sgn s) \tilde f(v) \leq \Lambda_T(u) := \norm{\tilde f}_{\infty}<\infty$ and
 \[
     \integrall{0}{\infty}{\frac1{\Lambda_T(u)}}{u} = \integrall{0}{\infty}{\frac1{\norm{\tilde f}_{\infty}}}{u} 
       = \lim_{R\to\infty} \frac{u}{\norm{\tilde f}_{\infty}} \bigg|_{u=0}^{u=R} = \infty \,.
 \]
The other assumptions~\ref{As:DI:16} and~\ref{As:DI:18} continue to hold.
Thus, due to (a generalization of) Theorem~\ref{thm:CP:Solution},
 there exists a unique solution to the Cauchy problem 
 \begin{equation} \label{CP:RD:RieszFeller:modified} \begin{cases}
   \diff{u}{t} = \RieszFeller u + \tilde f(u) &\Xx{for} (x,t)\in\R\times(0,T]\,, \\
   u(\cdot ,0)=u_0 &\Xx{for} x\in\R\,. 
 \end{cases} \end{equation}
Due to the assumptions on the initial datum $0\leq u_0\leq 1$ and 
 a comparison principle - formulated in Lemma~\ref{lem:comparison:Cb} - 
 $0\leq u(x,t) \leq 1$ for all $(x,t)\in\R\times[0,T]$.
Thus the solution $u(x,t)$ is a solution of the original Cauchy problem,
 whose uniqueness has to be verified.
Suppose two solutions of~\eqref{CP:RD} with the stated properties exist,
 then they are solutions of the modified Cauchy problem~\eqref{CP:RD:RieszFeller:modified} as well.
 However, the modified Cauchy problem has a unique solution, 
 hence the two solutions are identical. 
 
Due to (a generalization of) Theorem~\ref{thm:CP:Solution} a solution $u$ exists for all $T>0$ on a time interval~$(0,T)$.
However the comparison principle proves that $0\leq u(\cdot,t) \leq 1$ for all $t\geq 0$,
 such that $\norm{u(\cdot ,t)}_{L^\infty(\R)}$ satisfies not only~\ref{As:DI:23}
 but also $\norm{u(\cdot ,t)}_{L^\infty(\R)}\leq 1$ for all $t\geq 0$.
%The solution exists for all time already due to the controlled growth of $\norm{u(\cdot ,t)}_{C_b(\R)}$ see~\ref{As:DI:23}.
The solution $u$ is also a mild solution and satisfies 
 \[ u(x,t) = (\Green(\cdot ,t)\ast u_0)(x) + \integrall{0}{t}{ [\Green(\cdot ,t-\tau)\ast f(u(\cdot ,\tau))](x) }{\tau} \]
 for $t\geq 0$.
For $t\geq t_0>0$ the solution is differentiable and satisfies the mild formulation
 \[ \diff{u}{x}(x,t) = \Big(\diff{\Green(\cdot ,t)}{x}\ast u_0\Big)(x)
     + \integrall{0}{t}{ \Big[\diff{\Green(\cdot ,t-\tau)}{x}\ast f(u(\cdot ,\tau))\Big](x) }{\tau} \]
and hence the estimate  
 \begin{equation} \label{estimate:first-derivative}
    \sup_{x\in\R} \Abs{\diff{u}{x}(x,t)} \leq \cK t^{-\tfrac1\alpha} \underbrace{\norm{u_0}_{L^\infty}}_{\leq 1}
     + \max_{u\in[0,1]} \abs{f(u)}\, \cK\, \frac{t^{1-\tfrac1\alpha}}{1-\tfrac1\alpha}
 \end{equation}
 due to $0\leq u(\cdot,t) \leq 1$ for all $t\geq 0$ and Lemma~\ref{lem:SSPM}.
In particular, assumption $1<\alpha\leq 2$ implies 
 \[ t^{-\tfrac1\alpha}\leq t_0^{-\tfrac1\alpha} \XX{and} t^{1-\tfrac1\alpha}\leq (2t_0)^{1-\tfrac1\alpha} \]
 for all $t\in[t_0,2t_0]$.
Thus, for $t\in[t_0,2t_0]$, estimate~\eqref{estimate:first-derivative} yields
 \[ \sup_{x\in\R} \Abs{\diff{u}{x}(x,t)} 
     \leq \cK t_0^{-\tfrac1\alpha} + \max_{u\in[0,1]} \abs{f(u)}\, \cK\, \frac{(2t_0)^{1-\tfrac1\alpha}}{1-\tfrac1\alpha}\,.
 \]
This gives an estimate on bounded intervals, but not a global estimate on $[t_0,\infty)$.
However, the IVP generates a nonlinear semigroup;
 the solution $u$ of the IVP with initial condition $u(\cdot ,0)=u_0(\cdot)$ 
 is equal to the solution $v$ of the IVP with initial condition $v(\cdot ,t_0)=u(\cdot ,t_0)$ on the time interval $[t_0,\infty)$.
 Hence, $u$ and its derivative $\diff{u}{x}(x,t)$ satisfy
 \[ u(x,t) = (\Green(\cdot ,t-t_0)\ast u(\cdot ,t_0))(x) + \integrall{t_0}{t}{ [\Green(\cdot ,t-\tau)\ast f(u(\cdot ,\tau))](x) }{\tau} \]
 and
 \[
   \diff{u}{x}(x,t)
     = \Big(\diff{\Green(\cdot ,t-t_0)}{x}\ast u(\cdot ,t_0)\Big)(x)
      + \integrall{t_0}{t}{ \Big[\diff{\Green(\cdot ,t-\tau)}{x}\ast f(u(\cdot ,\tau))\Big](x) }{\tau}
 \]
 for $t\geq t_0> 0$. The estimate now reads 
 \[ \sup_{x\in\R} \Abs{\diff{u}{x}(x,t)} \leq \cK (t-t_0)^{-\tfrac1\alpha} \norm{u(\cdot ,t_0)}_{L^\infty}
     + \max_{u\in[0,1]} \abs{f(u)}\, \cK\, \frac{(t-t_0)^{1-\tfrac1\alpha}}{1-\tfrac1\alpha} \]
 for $t\geq t_0> 0$ and $t\in [2 t_0, 3 t_0]$ we obtain again 
 \[
    \sup_{x\in\R} \Abs{\diff{u}{x}(x,t)} 
%       \leq \cK t_0^{-\tfrac1\alpha} \norm{u(\cdot ,t_0)}_{L^\infty} 
%         + \max_{u\in[0,1]} \abs{f(u)}\, \cK\, \frac{(2t_0)^{1-\tfrac1\alpha}}{1-\tfrac1\alpha} \\
      \leq \cK t_0^{-\tfrac1\alpha} 
        + \max_{u\in[0,1]} \abs{f(u)}\, \cK\, \frac{(2t_0)^{1-\tfrac1\alpha}}{1-\tfrac1\alpha} 
 \]
 due to $1<\alpha\leq 2$ and the uniform estimate on $u$.
Whence by induction we obtain the uniform estimate of $\diff{u}{x}(x,t)$ on $(x,t)\in\R\times[t_0,\infty)$,
 and in a similar way the uniform estimates for all other derivatives of $u$.     
\end{proof}

\subsection{Comparison principles and far-field behavior}
\label{ssec:comparison}
\begin{lem} \label{lem:comparison:Cb}
Assume $1<\alpha\leq 2$, $\abs{\theta} \leq \min\{\alpha,2-\alpha\}$, $T>0$
 and $u,v \in C_b(\R\times [0,T])\cap C^2_b(\R\times (t_0,T])$ for all $t_0\in(0,T)$ such that 
 \[
   \diff{u}{t} \leq \RieszFeller u + f(u) \XX{and} \diff{v}{t} \geq \RieszFeller v + f(v) \Xx{in} \R\times (0,T]\,. 
 \]
\begin{enumerate}[label=(\roman*)]
\item If $v(\cdot ,0)\geq u(\cdot ,0)$ then $v(x,t)\geq u(x,t)$ for all $(x,t)\in\R\times (0,T]$.
\item If $v(\cdot ,0) \gneqq u(\cdot ,0)$ then $v(x,t) > u(x,t)$ for all $(x,t)\in\R\times (0,T]$.
\item Moreover,
 there exists a positive continuous function 
 \[ \eta:[0,\infty)\times(0,\infty)\to(0,\infty)\,, \quad (m,t)\mapsto\eta(m,t)\,, \]
 such that if $v(\cdot ,0)\geq u(\cdot ,0)$ then for all $(x,t)\in\R\times(0,T)$
 \[ v(x,t)-u(x,t) \geq \eta(\abs{x},t) \integrall{0}{1}{ v(y,0)-u(y,0) }{y} \,. \]
% \item Moreover,
%  the function $\eta:[1,\infty)\to (0,\infty)$, $m\mapsto \min_{z\in [-m-1,m]} \Green(z,1)$, is positive and non-increasing
%  such that if $v(\cdot ,0)\geq u(\cdot ,0)$ then for some $C>0$ and every $m\geq 1$
%  \[ \min_{x\in [-m,m]} v(x,1)-u(x,1) \geq C \eta(m) \integrall{0}{1}{ v(y,0)-u(y,0) }{y} \,. \]
\end{enumerate}
\end{lem}
\begin{proof}
\begin{enumerate}[label=(\roman*)]
\item The function $w:=v-u$ satisfies $w\in C_b(\R\times [0,T])\cap C^2_b(\R\times (t_0,T])$ for all $t_0\in(0,T)$,
 $w(\cdot ,0)\geq 0$ in $\R$ and
 \begin{multline*}
  \diff{w}{t} 
  = \diff{}{t} (v-u)
    \geq \RieszFeller (v-u) + f(v) - f(u) \\
  = \RieszFeller (v-u) + \integrall{0}{1}{ f'(\theta v + (1-\theta)u) \, (v-u) }{\theta} 
  = \RieszFeller w + \underbrace{\bigg( \integrall{0}{1}{ f'(\theta v + (1-\theta)u) }{\theta} \bigg)}_{=:k(x,t)} w \,.
 \end{multline*}
In particular, $k:\R\times [0,T]\to\R$, $(x,t)\mapsto k(x,t)$, is a bounded continuous function,
 due to the properties of $u$ and $v$.
To prove $w\geq 0$ in $\R\times (0,T]$,
 we will derive a contradiction following~\cite[page 153]{Chen:1997}.
Assume $w$ takes negative values in $\R\times [0,T]$.
Due to $w\in C_b(\R\times [0,T])$ and $w(\cdot ,0)\geq 0$, for any $\kappa>0$ 
 there exist $\epsilon>0$ and $T\geq \tilde T>0$ such that
 \[ w(x,t)>-\epsilon \exp(2\kappa t) \Xx{in} \R\times [0,\tilde T) \Xx{and} \inf_{x\in\R} w(x,\tilde T) = -\epsilon \exp(2\kappa \tilde T) \,. \]
In the following we use again $T$ instead of $\tilde T$ 
 and assume without loss of generality $w(0,T)<-\tfrac78 \epsilon \exp(2\kappa T)$.
Consider $\omega(x,t):= -\epsilon (\tfrac34 +\sigma z(x)) \exp(2\kappa t)$
 where $\sigma>0$ and $z\in C^\infty(\R)$, $z(0)=1$, $\lim_{x\to\pm\infty} z(x)=3$,
 as well as $3\geq z\geq 1$, $\abs{z'}\leq 1$ and $\abs{z''}\leq 1$ in $\R$.
The function $\omega$ satisfies for $\sigma\geq 0$
 \[ \omega(x,t)= -\epsilon \left(\tfrac34 +\sigma z(x)\right) \exp(2\kappa t)
      \leq -\epsilon \left(\tfrac34 +\sigma\right) \exp(2\kappa t) \]
 where the upper bound is monotone decreasing with respect to $\sigma$.
Thus there exists a $\sigma^*\in (\tfrac18,\tfrac14]$ such that $w\geq\omega$ in $\R\times[0,T]$,
 where $\tfrac18<\sigma^*$ due to the restrictions at $x=0$.
Moreover
 \begin{multline*}
    \lim_{x\to\pm\infty} -\epsilon (\tfrac34 +\sigma^* z(x)) \exp(2\kappa t) = -\epsilon (\tfrac34 + 3 \sigma^*) \exp(2\kappa t) \\
      < -\epsilon \tfrac98 \exp(2\kappa t) < -\epsilon \exp(2\kappa t) \leq \liminf_{x\to\pm\infty} w(x,t)\,.
 \end{multline*}
In summary,
 there exists $\sigma^*\in (\tfrac18,\tfrac14]$ and $(x_0,t_0)\in \R\times(0,T]$
 such that $w\geq\omega$ in $\R\times[0,T]$ and $w(x_0,t_0)=\omega(x_0,t_0)$.
Thus $w-\omega\in C_b(\R\times [0,T])\cap C^2_b(\R\times (t_0,T])$ is a non-negative function
 which attains its minimum at $(x_0,t_0)\in \R\times(0,T]$, hence
 \[ \diff{w}{t}(x_0,t_0)\leq \diff{\omega}{t}(x_0,t_0)\,,
      \qquad \diff{w}{x}(x_0,t_0)=\diff{\omega}{x}(x_0,t_0)\,,
      \qquad \difff{w}{x}{2}(x_0,t_0)\geq \difff{\omega}{x}{2}(x_0,t_0)\,.
 \]
First
 we deduce from the integral representation of $\RieszFeller$ in Theorem~\ref{thm:RieszFeller:extension}
 the estimate $\RieszFeller[w(\cdot ,t_0)](x_0)\geq \RieszFeller[\omega(\cdot ,t_0)](x_0)$.
Second
 we deduce the estimate
\begin{align*}
 - \tfrac74 \epsilon \kappa \exp(2\kappa t_0)
   &\geq \diff{\omega}{t}(x_0,t_0) \geq \diff{w}{t}(x_0,t_0) \\
	 &\geq \RieszFeller[w(\cdot ,t_0)](x_0) + k(x_0,t_0)w(x_0,t_0) \\
	 &\geq \RieszFeller[\omega(\cdot ,t_0)](x_0) - \sup\abs{k}\; \abs{\omega(x_0,t_0)} \\ 
	 &\geq -\cK \|{\omega''}\|_{C_b(\R)} \frac{M^{2-\alpha}}{2-\alpha}
         -4\cK \|\omega'\|_{C_b(\R)} \frac{M^{1-\alpha}}{\alpha-1}
				 -\sup\abs{k}\; \epsilon (\tfrac34 + \sigma^* z(x_0)) \exp(2\kappa t_0) \\ 
	 &\geq -\cK \frac{M^{2-\alpha}}{2-\alpha} \epsilon\sigma^* \exp(2\kappa t_0)
         -4\cK \frac{M^{1-\alpha}}{\alpha-1} \epsilon\sigma^* \exp(2\kappa t_0)
				 -\sup\abs{k}\; \epsilon\tfrac64 \exp(2\kappa t_0)\,, 
\end{align*}
where we use Proposition~\ref{prop:RieszFeller:estimate} with some positive constants $M$ and $\cK$.
Thus if we choose $\kappa>0$ such that
 \[ -\tfrac74 \kappa< -\tfrac{\cK}4 \tfrac{M^{2-\alpha}}{2-\alpha} - \cK \tfrac{M^{1-\alpha}}{\alpha-1}  - \tfrac64 \sup\abs{k}\]
 then we obtain a contradiction. 
Therefore $w\geq 0$ in $\R\times (0,T]$.

\item For another constant $K_2\in\R$, the function $w_2:= \exp(K_2 t) w$ satisfies
 $w_2 \in C_b(\R\times [0,T])\cap C^2_b(\R\times (t_0,T])$ for all $t_0\in(0,T)$, $w_2\geq 0$ in $\R\times (0,T]$, and
 \[ \diff{w_2}{t} \geq \RieszFeller w_2 - c_2(x,t) w_2 \]
 with $c_2(x,t):=-(K_2 + k(x,t))$.
 Choosing $K_2\in\R$ such that $c_2(x,t)=-(K_2 + k(x,t))\leq 0$ and using $w_2\geq 0$ in $\R\times (0,T]$, yields
 \[ \diff{w_2}{t} \geq \RieszFeller w_2 - c_2(x,t) w_2 \geq \RieszFeller w_2 \,. \]
 Due to the first part,
 \begin{align*}
  w_2(x,t) 
   \geq [ \Green(\cdot ,t) \ast w_2(\cdot ,0) ](x) = \exp(K_2 t) [ \Green(\cdot ,t) \ast w(\cdot ,0) ](x) \,.
 \end{align*}
 The assumption $v(\cdot ,0) \gneqq u(\cdot ,0)$ implies that
 there exists $x_0\in\R$ % such that $w(x_0,0)>0$ 
 and $\epsilon>0$ such that $w(x,0)>0$ for all $x\in (x_0-\epsilon,x_0+\epsilon)$ due to continuity of $w$.
 Moreover, the nonlocal diffusion equation $\diff{w}{t} = \RieszFeller w$ generates a convolution semigroup
 with a positive convolution kernel $\Green(x,t)$, i.e. $\Green(x,t)>0$ in $\R\times (0,T]$, see Lemma~\ref{lem:SSPM}. %~\cite{Schneider:1986}.
Therefore, 
 \[ w_2(x,t) \geq \integral{U_\epsilon(x_0)}{ \Green(x-y,t) \, w_2(y,0)}{y}>0 \Xx{for all} (x,t)\in\R \times (0,T] \,, \]
 which implies $w(x,t)>0$ for all $(x,t)\in\R \times (0,T]$.

\item If $v(\cdot ,0)\geq u(\cdot ,0)$ then as before
 \begin{align*}
  w_2(x,t)
   &\geq [ \Green(\cdot ,t) \ast w_2(\cdot ,0) ](x) %= \exp(K_2 t) [ \Green(\cdot ,t) \ast w(\cdot ,0) ](x) \,.
    \geq \integrall{0}{1}{ [\Green(x-y,t) w_2(y,0)] }{y} \\
   &\geq \min_{y\in [0,1]} \Green(x-y,t) \integrall{0}{1}{w_2(y,0)}{y} \,,
 \end{align*}
  whereas the estimates follow from $\Green$ being an integrable positive smooth function,
  and $w_2(\cdot ,0)\geq 0$ in $\R$.
Thus
 \begin{align*}
   \exp(K_2 t) \big(v(x,t)-u(x,t)\big)
     &\geq \min_{z\in [-\abs{x}-1,\abs{x}]} \Green(z,t) \integrall{0}{1}{w(y,0)}{y} \\
     &= \tilde\eta(\abs{x},t) \integrall{0}{1}{ \big(v(y,0)-u(y,0)\big)}{y} \,,
 \end{align*}
 where $\tilde\eta(m,t)= \min_{z\in [-m-1,m]} \Green(z,t)$,
 is a positive continuous function, since $\Green(\cdot ,t)$ for $t>0$ is a positive smooth function.
Consequently the function $\eta:[0,\infty)\times(0,\infty)\to(0,\infty)$, $(m,t)\mapsto \exp(-K_2 t)\tilde\eta(\abs{x},t)$,
 is a positive continuous function, 
 and the statement follows.
\end{enumerate}
\end{proof}

We need to investigate the behavior of solutions in the limits $x\to\pm\infty$,
 see also~\cite[Theorem 5.2]{Volpert+etal:1994} for the case of a system of reaction-diffusion equations with local derivatives.
We consider the Cauchy problem 
\begin{equation} \label{CP:RD:general} \begin{cases}
 \diff{u}{t} = \RieszFeller u + F(u) &\Xx{for} (x,t)\in\R\times (0,\infty)\,, \\
 u(x,0) = u_0(x) &\Xx{for} x\in\R\,,                 
\end{cases} \end{equation}
 for some unknown function $u:\R\times (0,\infty)\to\R$
 and a given bounded continuous function $F:\R\to\R$, $u\mapsto F(u)$,
 satisfying a Lipschitz condition in $u$.

\begin{thm} \label{thm:far-field-behaviour}
Let $1<\alpha\leq 2$ and $\abs{\theta} \leq \min\{\alpha,2-\alpha\}$.
Suppose $u_0\in C_b(\R)$ and that the limits
 \[ \lim_{x\to\pm\infty} u_0(x) = u_{0,\pm} \]
 exist.
If $u\in C_b(\R\times [0,T])\cap C^2_b(\R\times (t_0,T])$ for all $t_0\in(0,T)$ is a solution of the Cauchy problem~\eqref{CP:RD:general}
 then the limits $\lim_{x\to\pm\infty} u(x,t) = \upm(t)$ exist and satisfy 
 \begin{equation} %\label{IVP:limits}
   \Diff{\upm}{t} = F (\upm) \Xx{for} t\in[0,T]\,, \qquad \upm(0)=u_{0,\pm}\,.
 \end{equation}
\end{thm}
\begin{proof}
The result is a variation of~\cite[Theorem 5.2]{Volpert+etal:1994} where the case $D^2_0=\difff{}{x}{2}$ is considered.
The proof holds verbatim, since for $1<\alpha\leq 2$ and $\abs{\theta} \leq \min\{\alpha,2-\alpha\}$
 the fundamental solution~$\Green$ of 
 \[ \diff{u}{t} = \RieszFeller u \Xx{for} (x,t)\in\R\times(0,T] \]
 is for all $t>0$ an integrable positive smooth function $\Green(\cdot ,t)\in L^1(\R)$ with finite mean,
 see Lemma~\ref{lem:SSPM}.
\end{proof}

\begin{proof}[Proof of Theorem~\ref{thm:far-field-behaviour}]
% Due to a Sobolev imbedding, $W^{1,\infty}(\R) \hookrightarrow C_b(\R)$, we consider $u_0\in C_b(\R)$.
The result is a variation of~\cite[Theorem 5.2]{Volpert+etal:1994} where the case $D^2_0=\difff{}{x}{2}$ is considered.
Again, for $1<\alpha\leq 2$ and $\abs{\theta} \leq \min\{\alpha,2-\alpha\}$ the fundamental solution~$\Green$ of 
 \[ \diff{u}{t} = \RieszFeller u \Xx{for} (x,t)\in\R\times(0,T] \]
 is for all $t>0$ an integrable positive smooth function $\Green(\cdot ,t)\in L^1(\R)$ with finite mean,
 see Lemma~\ref{lem:SSPM}.
Like in the proof of~\cite[Theorem 5.1]{Volpert+etal:1994},
 we obtain the unique mild solution as the limit of an iterated sequence
\begin{align*}
 u^0(x,t) &= \integrall{-\infty}{+\infty}{\Green(x-y,t) u_0(y)}{y} \\
 u^{k+1}(x,t) &= u^0(x,t) + \integrall{0}{t}{\integrall{-\infty}{+\infty}{\Green(x-y,t-\tau) F(u^k(y,\tau))}{y}}{\tau}
\end{align*}
for $k\in\N$.
The functions $u^k$ are bounded and continuous, hence measurable.
To study the limits of a solution $u$,
 we consider the limits of the functions $u^k$.
The dominated convergence theorem yields
\begin{multline*}
 u^0_{\pm}(t)
  := \lim_{x\to\pm\infty} u^0(x,t)
  = \lim_{x\to\pm\infty} \integrall{-\infty}{+\infty}{\Green(y,t) u_0(x-y)}{y} \\
  = \integrall{-\infty}{+\infty}{\Green(y,t) \lim_{x\to\pm\infty} u_0(x-y)}{y}
%    = \integrall{-\infty}{+\infty}{\Green(y,t) \lim_{x\to\pm\infty} u_0(x-y)}{y} \\
  = \integrall{-\infty}{+\infty}{\Green(y,t)\, u_{0,\pm}}{y} 
  = u_{0,\pm}\,.
\end{multline*}
A mathematical induction on $k\in\N$ proves that the limits of $u^k$ satisfy
\begin{align*}
 u^{k+1}_{\pm}(t)
  :&= \lim_{x\to\pm\infty} u^{k+1}(x,t) \\
   &= u_{0,\pm} + \lim_{x\to\pm\infty} \integrall{0}{t}{\integrall{-\infty}{+\infty}{\Green(x-y,t-\tau) F(u^k(y,\tau))}{y}}{\tau} \\
   &= u_{0,\pm} + \integrall{0}{t}{\integrall{-\infty}{+\infty}{\Green(y,t-\tau)\, \lim_{x\to\pm\infty} F(u^k(x-y,\tau))}{y}}{\tau} \\
   &= u_{0,\pm} + \integrall{0}{t}{F(u^k_{\pm}(\tau))}{\tau} \,.
\end{align*}
The sequence of functions $u^k_{\pm}(t)$ converges uniformly for $0<t\leq T$ to some function $u_{\pm}(t)$,
 by virtue of the uniform convergence of the sequence of functions $u^k(x,t)$, $k\in\N$.
Passing to the limit, we obtain 
\[ u_{\pm}(t) = u_{0,\pm} + \integrall{0}{t}{F(u_{\pm}(\tau))}{\tau} \,, \]
which is equivalent to the stated differential equation.
\end{proof}

%% file: GLE-TWP-Chen.tex
In this section we briefly review the results from \cite{Chen:1997} as they provide the basis for this work. 
Consider the evolution equation
\be
\label{eq:Chen}
\frac{\partial u}{\partial t}(x,t)=\cA [u(\cdot,t)](x),\qquad (x,t)\in\R\times[0,\I)\,,
\ee
where $\cA$ is a nonlinear operator. We shall also need the Fr\'{e}chet derivative of $\cA$ defined 
by 
\benn
\cA'[u](v):=\lim_{\epsilon\ra 0}\frac{\cA[u+\epsilon v]-\cA[u]}{\epsilon}.
\eenn
The basic assumptions on the operator $\cA$ are:

\begin{itemize}
 \item (\textbf{semigroup}) $\cA$ generates a semigroup on $L^\I(\R)$,
 \item (\textbf{translation invariance}) $\cA[u(\cdot +h)](x)=\cA[u(\cdot)](x+h)$ for all $x,h\in\R$,
 \item (\textbf{bistability}) there exists a function $f(\cdot)$
  such that $\cA[\alpha 1]=f(\alpha)1$ for all $\alpha\in\R$ with
 \be
 \label{As:f}
 f\in C^1(\R)\,,\qquad f(0)=0=f(1)\,,\qquad f'(0)<0\,,\qquad f'(1)<0\,,
 \ee
 \item (\textbf{comparison principle}) 
 \be
 \label{eq:comparison}
 \text{if } u_t\geq \cA[u], ~v_t\leq \cA[v],~ u(\cdot,0)\geq v(\cdot,0),~u(\cdot,0)\not\equiv v(\cdot,0)
~\text{ then }u(\cdot,t)> v(\cdot,t)~\forall t>0.
 \ee 
\end{itemize}

Xinfu Chen \cite{Chen:1997} studies the existence, uniqueness and stability of traveling fronts $u(x,t)=U(x-ct)$
 for \eqref{eq:Chen} connecting the two homogeneous stable states 
 {i.e.}~in a moving coordinate frame $\xi=x-ct$ one demands 
\be
\label{eq:wave}
\lim_{\xi\ra -\I}U(\xi)=0\,, \qquad \lim_{\xi\ra \I}U(\xi)=1 \qquad\text{and}\qquad \lim_{|\xi|\ra \I}U'(\xi)=0\,.
\ee
We state the three main results from \cite{Chen:1997}
 which will follow from the semigroup property,
 several variants of the other three properties and additional estimates for $\cA$. 

\begin{thm}(\textit{uniqueness}, \cite[Thm.~2.1]{Chen:1997}) \label{thm:Chen:1}
Suppose the following assumptions hold:
\begin{enumerate}[label=(A\arabic*)]
 \item \label{As:Chen:A1}
  $\cA$ is translation invariant and $f$ is bistable in the sense of \eqref{As:f}.
 \item \label{As:Chen:A2}
  $\cA$ satisfies the comparison principle \eqref{eq:comparison}.
 \item \label{As:Chen:A3}
  There exists constants $K_1>0$ and $K_2>0$ and a probability measure $\nu$ such that for any functions
	$u,v$ with $-1\leq u,v\leq 2$ and every $x\in\R$
	\be
	\label{eq:Chen:A3}
	\left|\cA'[u+v](1)(x)-\cA'[u](1)(x)\right|\leq K_1 \int_\R |v(x-y)|\nu(\dy)+K_2\|v(x+\cdot)\|_{C^0([-1,1])}.
	\ee
\end{enumerate}
Then monotonic traveling waves are unique up to translation. More precisely, suppose \eqref{eq:Chen} has 
a traveling wave $U\in C^1(\R)$ with speed $c$ satisfying \eqref{eq:wave} and $U'(\xi)>0~\forall \xi\in\R$, 
then any other traveling wave solution $(\tilde{U},\tilde{c})$ with $\tilde{U}\in C^0(\R)$ and 
$0\leq \tilde{U}\leq 1$ on $\R$ satisfies
\benn
c=\tilde{c}\qquad \text{and} \qquad \tilde{U}(\cdot)=U(\cdot+\xi_0)~\text{for some fixed $\xi_0\in\R$}
\eenn
{i.e.}~$\tilde{U}$ is a translate of the original wave $U$.
\end{thm}

To obtain stability of the traveling wave one has to extend the assumptions (A1)-(A3).

\begin{thm}(\textit{stability}, \cite[Thm.~3.1]{Chen:1997}) \label{thm:Chen:2}
Suppose (A1)-(A3) hold and, in addition, we have:
\begin{enumerate}[label=(B\arabic*)]
 \item \label{As:Chen:B1}
 There exist constants $a^-$ and $a^+$ with $0< a^- \leq a^+<1$ such that $f$ satisfies $f>0$ 
 in $(-1,0)\cup (a^+,1)$ and $f<0$ in $(0,a^-)\cup (1,2)$.
 \item \label{As:Chen:B2}
  There exists a positive non-increasing function $\eta(m)$ defined on $[1,\infty)$ such that 
	for any functions $u(x,t),v(x,t)$ satisfying $-1\leq u,v \leq 2$, $\frac{\partial u}{\partial t}\geq \cA[u]$, 
	$\frac{\partial v}{\partial t}\leq \cA[v]$ and $u(\cdot ,0)\geq v(\cdot ,0)$, there holds 
  \be
	\label{eq:Chen:B2}
    \min_{x\in [-m,m]} \left[ u(x,1)-v(x,1) \right]
    \geq \eta(m) \int_0^1  [u(y,0)-v(y,0)]~ \dy \quad \forall m\geq 1.
  \ee
 \item \label{As:Chen:B3}
  With $K_1$, $K_2$, $\nu$, $u$ and $v$ as in $(A3)$, there holds, for every $x\in\R$,
  \be
	\label{eq:Chen:B3}
    \left| \cA[u+v](x) - \cA[u](x) \right| \leq K_1 \int_\R |v(x-y)|~  \nu(\dy) + 
		K_2 \left\| v'' \right\|_{C^0(\R)} \,.
	\ee
\end{enumerate}
Then monotonic traveling waves are globally exponentially stable. More precisely, suppose
\eqref{eq:Chen} has a traveling wave $U\in C^1(\R)$ with speed $c$ satisfying \eqref{eq:wave} 
and $U'(\xi)>0~\forall \xi\in\R$. Then there exists a constant $\kappa$ such that for any 
$u_0\in L^\I(\R)$ satisfying $0\leq u_0\leq 1$ and
\benn
\liminf_{x\ra \I} u_0(x)>a^+,\qquad \limsup_{x\ra -\I} u_0(x)<a^-,  
\eenn 
the solution $u(x,t)$ of \eqref{eq:Chen} with initial value $u(\cdot,0)=u_0(\cdot)$ satisfies
the exponential stability estimate
\benn
\|u(\cdot,t)-U(\cdot -ct+\xi)\|_{L^\I(\R)}\leq K e^{-\kappa t}\qquad \text{for all $t\leq 0$},
\eenn
where $\xi$ and $K$ are constants depending on $u_0$.
\end{thm}

The strongest set of assumptions is required to show the existence of a traveling wave.

\begin{thm}(\textit{existence}, \cite[Thm.~4.1]{Chen:1997}) \label{thm:Chen:3}
Suppose the following assumptions are satisfied:
\begin{enumerate}[label=(C\arabic*)]
 \item \label{As:Chen:C1}
 $\cA$ is translation invariant and the function $f$ satisfies for some $a\in(0,1)$,
 \benn
 f>0\text{ in $(-1,0)\cup (a,1)$},\quad f<0\text{ in $(0,a)\cup (1,2)$},\quad 
 f'(0)<0,~f'(1)<0,~f'(a)>0.
 \eenn 
 \item \label{As:Chen:C2}
  There exists a positive continuous function $\eta(x,t)$ defined on $[0,\infty)\times (0,\I)$ 
  such that if $u(x,t),v(x,t)$ satisfy $-1\leq u,v \leq 2$, $\frac{\partial u}{\partial t}\geq \cA[u]$, 
	$\frac{\partial v}{\partial t}\leq \cA[v]$ and $u(\cdot ,0)\geq v(\cdot ,0)$, then
  \be
   \label{eq:Chen:C2}
   u(x,t)-v(x,t) \geq \eta(|x|,t) \int_0^1  [u(y,0)-v(y,0)]~ \dy \quad \forall (x,t)\in\R\times(0,\I).
  \ee
 \item \label{As:Chen:C3}
 There exit positive constants $K_1$, $K_2$, $K_3$, and a probability measure $\nu$ such that for 
 any $u,v\in L^\I(\R)$ with $-1\leq u,v\leq 2$, and $x\in\R$ we have
  \bea
    \label{eq:Chen:C3_1} && \left| \cA[u+v](x) - \cA[u](x) \right| \leq K_1 \int_\R |v(x-y)|~  \nu(\dy) 
      + K_2 \left\| v'' \right\|_{C^0([x-1,x+1])},\\
    \label{eq:Chen:C3_2} && \left| \cA[u+v]-\cA[u]-\cA'[u](v)\right|\leq K_3\|v\|^2_{C^0(\R)},\\
    \label{eq:Chen:C3_3} && \left| \cA'[u+v](1)(x) - \cA'[u](1)(x) \right| \leq K_1 \int_\R |v(x-y)|~\nu(\dy) 
      + K_2 \left\| v \right\|_{C^0([x-1,x+1])}. 
  \eea
 \item \label{As:Chen:C4}
 For any function $u_0(\cdot)$ satisfying $0\leq u_0\leq 1$ and $\|u_0\|_{C^3(\R)}<\I$, the solution 
 $u(x,t)$ of \eqref{eq:Chen} with initial condition $u(\cdot,0)=u_0(\cdot)$ satisfies 
 $\sup_{t\in[0,\I)}\|u(\cdot,t)\|_{C^2(\R)}<\I$.
\end{enumerate}
Then there exists a traveling wave $U\in C^1(\R)$ with speed $c$ satisfying \eqref{eq:wave} 
and $U'(\xi)>0~\forall \xi\in\R$.
\end{thm}

Observe that the assumption (C$i$) for $i=1,2,3$ implies (A$i$) as well as (B$i$). Furthermore, the 
first assumption for each theorem prescribes the nonlinear bistability behavior, the second one
is a comparison principle and the third assumption yields estimates on the nonlinear operator $\cA$
as well as on its linearization $\cA'$.\medskip

It is likely that Chen proved the Theorems \ref{thm:Chen:1}--\ref{thm:Chen:3}
 having in mind a general class of integro-differential evolution equations of the form
\be
\label{eq:IDE}
\frac{\partial u}{\partial t}=\delta\frac{\partial^2 u}{\partial x^2}+G(u,J_1\ast S^1(u),\ldots,J_n\ast S^n(u))
\ee
where $\delta\geq 0$ is the diffusion coefficient, $G$ and $S^k$ are smooth functions, $J_k\ast S^k(u)$
denotes the convolution $\int_\R J_k(x-y)S^k(u(y))\dy$ of $S^k(u)$ with a non-negative kernel $J_k\in C^1(\R)$ of unit mass
$\int_\R J_k(y)\dy=1$ and bounded total variation $\int_\R|J_k'(y)|\dy<\I$. 
In \cite[{Section 5}]{Chen:1997} further assumptions are specified 
 such that the conditions~\ref{As:Chen:C1}--\ref{As:Chen:C4} hold,
 which implies the existence, uniqueness and exponential stability of traveling wave solutions for these equations. 
It turns out that the approach does not apply directly
 when we replace the Laplacian in \eqref{eq:IDE} by a more general Riesz-Feller operator. 

%% file: GLE-TWP.tex
% We consider a scalar quantity $u: \R_+\times\R \to U\subset\R$, $(t,x) \mapsto u(t,x)$,
%  which is governed by the equation
%  \begin{equation} \label{RD}
%   \diff{u}{t} = \RieszFeller u + f(u) \,,\quad x\in\R \,,\quad t\in\R_+ \,,
%  \end{equation}
% for some function $f:\R\to\R$ and fixed parameters $1<\alpha\leq 2$ and $\abs{\theta} \leq \min\{\alpha,2-\alpha\}$.
% 
% The operator $\cA$ is identified as $\cA[u]:= \RieszFeller [u] + f(u)$.

The analysis of equation~\eqref{RD:TWP} in the Sections \ref{sec:RF} and~\ref{ap:cp} show
 that we only need a relatively mild generalization of Chen's results \cite{Chen:1997}
 which we reviewed in Section \ref{sec:Chen}.

First we identify the operator $\cA$ as $\cA[u]:= \RieszFeller u + f(u)$
 and take a look at the assumptions~\ref{As:Chen:C1}--\ref{As:Chen:C4}.
\begin{enumerate}[label=(C3')]
\item[(C1)]
 The Riesz-Feller operators~$\RieszFeller$ are translational invariant with respect to the spatial variable,
  which follows from their integral representation in Theorem~\ref{thm:RieszFeller:extension}.
 The nonlinearity $u\mapsto f(u)$ depends on the spatial variable only through the function $u$ itself,
  hence the operator is again translational invariant.
 Consequently, the operator $\cA$ is translational invariant, 
  since it is the sum of translational invariant operators.

 Due to translational invariance, the operator $\cA$ maps a constant function to a constant function.
 In particular, $\cA[c \mathbf{1}]= \RieszFeller [c \mathbf{1}] + f(c)\mathbf{1}= f(c)\mathbf{1}$ for all $c\in\R$,
  where $\mathbf{1}$ denotes the constant function $x\mapsto 1$.
 The additional assumptions on $f$ identify the admissible nonlinear functions.
\item[(C2)]
 The property follows Lemma~\ref{lem:comparison:Cb}.

\item \label{As:Chen:C3b} %[(C3b)]
 In the following, we consider $u,v\in L^\I(\R)$ with $-1\leq u,v\leq 2$, see assumption~\ref{As:Chen:C3}. 
 The quantity in~\eqref{eq:Chen:C3_1} is estimated as
 \begin{multline*}
  \abs{ \cA[u+v](x) - \cA[u](x) }
   = \abs{ \RieszFeller [u+v] + f(u+v) - \RieszFeller u - f(u) } \\
  \leq \abs{ \RieszFeller v }(x) + \abs{ f(u+v) - f(u) }(x)
   \leq K_2 \norm{v''}_{C(\R)} + K_4 \norm{v'}_{C(\R)} + K_1 \abs{v(x)}       
 \end{multline*}
 for some positive constants $K_1$, $K_2$ and $K_4$,
 due to Proposition~\ref{prop:RieszFeller:estimate} and 
 \[
  \abs{ f(u+v) - f(u) }(x) = \Abs{ \integrall{0}{1}{ f'(u+t v) }{t}\; v(x) } \leq \norm{f'}_{C([-2,4])} \abs{v(x)} \,.
 \]
 Note that the estimate involves $\norm{v''}_{C(\R)}$ instead of $\norm{v''}_{C([x-1,x+1])}$
  due to the estimate of the Riesz-Feller operator in Proposition~\ref{prop:RieszFeller:estimate}.
 The Fr\'{e}chet derivative $\cA'[u](v)$ of $\cA$ is $\cA'[u](v)= \RieszFeller v + f'(u)v$.
 The second estimate~\eqref{eq:Chen:C3_2} follows from
 \begin{align*}
  \abs{ \cA[u+v]-\cA[u]-\cA'[u](v) }
   &= \abs{ f(u+v) - f(u) - f'(u)v } \\
   &= \Abs{ \integrall{0}{1}{ \integrall{0}{t}{ f''(u+s v) }{s} }{t}\; v^2 }
    \leq \norm{f''}_{C([-2,4])} \abs{v(x)}^2 \,.
 \end{align*}
 The third  estimate~\eqref{eq:Chen:C3_3} follows from
 \begin{align*}
  \abs{ \cA'[u+v](1)(x) - \cA'[u](1)(x) } 
   &= \abs{ f'(u+v) - f'(u) } \\
   &= \Abs{ \integrall{0}{1}{ f''(u+t v) }{t}\; v(x) } \leq \norm{f''}_{C([-2,4])} \abs{v(x)} \,.
 \end{align*}
\item[(C4)]
Due to Theorem~\ref{thm:CP:existence:bistable},
 the Cauchy problem with initial datum $u_0\in C^3(\R)$ and $0\leq u_0\leq 1$
 has a solution $u(x,t)$ which satisfies the properties~\ref{As:DI:20}--\ref{As:DI:23},
 $0\leq u\leq 1$ and the uniform estimates $\sup_{t\in[0,\infty)} \norm{u(\cdot,t)}_{C^2(\R)}<\infty$.
 We observe that a solution $u$ of the IVP with initial datum $u_0\in L^\infty(\R)$ and $0\leq u_0\leq 1$ almost everywhere
  becomes smooth for positive times and its $C^k_b(\R)$-norm for any $k\in\N$ can be uniformly bounded.
\end{enumerate}
The modifications in the estimates in~\ref{As:Chen:C3b} are due to our replacement of a second-order derivative with a Riesz-Feller operator,
 which demand a local estimate versus a global estimate see Proposition~\ref{prop:RieszFeller:estimate}.
Furthermore, we prefer to work in a $C_b$ setting instead of a $L^\infty$ setting.
\begin{thm}
Theorems \ref{thm:Chen:1}-\ref{thm:Chen:3} still hold if each term $K_2\|v''\|_{C^0(\R)}$ is replaced by
\benn
\tilde{K}_2\|v'\|_{C_b(\R)}+K_2\|v''\|_{C_b(\R)}
\eenn
occurring in the inequalities \eqref{eq:Chen:B3} and \eqref{eq:Chen:C3_1}.
\end{thm}

\begin{proof}
Precise statements and details are given in Appendix~\ref{ap:uniq} for uniqueness,
 Appendix~\ref{ap:stab} for stability and Appendix~\ref{ap:exist} for existence.
\end{proof}

Finally we can prove the main result.
\begin{thm}
Suppose $1<\alpha\leq 2$, $\abs{\theta} \leq \min\{\alpha,2-\alpha\}$ and $f\in C^\infty(\R)$ satisfies~\eqref{As:f:0}.
Then equation~\eqref{RD} admits a traveling wave solution $u(x,t)=U(x-ct)$ satisfying~\eqref{As:TWS}.
In addition, a traveling wave solution of~\eqref{RD} is unique up to translations.
Furthermore, traveling wave solutions are globally asymptotically stable in the sense that
 there exists a positive constant $\kappa$ such that 
 if $u(x,t)$ is a solution of~\eqref{RD} with initial datum~$u_0\in C_b(\R)$ satisfying $0\leq u_0\leq 1$ and
 \[ \liminf_{x\to\infty} u_0(x) > a\,, \qquad \limsup_{x\to-\infty} u_0(x) < a\,, \]
 then, for some constants $\xi$ and $K$ depending on $u_0$,
 \[ \norm{u(\cdot,t)-U(\cdot- ct + \xi)}_{L^\infty(\R)} \leq K e^{-\kappa t} \qquad \forall t\geq 0\,. \] 
\end{thm} 
\begin{proof}
Under the assumption of this theorem,
 we studied at the beginning of this subsection
 Chen's original conditions~\ref{As:Chen:C1}--\ref{As:Chen:C4}.
We noticed that only in condition~\ref{As:Chen:C3} one estimate has to be modified.
This implies that the same estimate has to be changed also in condition~\ref{As:Chen:B3}.
However, in the Appendices we verify that his approach can be modified
 to obtain the stated results on existence in Theorem~\ref{thm:existence},
 uniqueness in Theorem~\ref{thm:uniqueness}
 and stability in Theorem~\ref{thm:stability} of traveling wave solutions of~\eqref{RD}.
\end{proof}

%% file: GLE-Uniqueness.tex
The problem~\eqref{RD} under consideration fulfills the assumptions~\ref{As:Chen:A2} and~\ref{As:Chen:A3}
 due to the discussion in Section~\ref{sec:TWP}.
For nonlinear functions satisfying the assumptions~\ref{As:Chen:A1},
 the uniqueness result in Theorem~\cite[Theorem 2.1]{Chen:1997} is applicable.
In the following we reproduce the proof with all modifications.

Consider wave speed $c\in\R$ and traveling wave variable $\xi:=x-c t$.
\begin{defn}
  A traveling wave solution of~\eqref{RD} is a solution of the form $u(x,t)=U(\xi)$, 
  for some function~$U$ that connects different endstates $\lim_{\xi\to\pm\infty} U(\xi) = \upm$.
\end{defn}
A traveling wave solution satisfies the \textit{traveling wave equation} $- c U'(\xi) = \RieszFeller U + f(U)$.
\begin{thm} \label{thm:uniqueness}
 Suppose~\ref{As:Chen:A1} holds and $(U,c)$ is a traveling wave solution of~\eqref{RD} satisfying
 \begin{multline} \label{A:TWS}
  U\in C^1(\R)\,, \quad
   \lim_{\xi\to -\infty} U(\xi) = 0 =:\um \,, \quad
   \lim_{\xi\to +\infty} U(\xi) = 1 =:\up \,, \\
  U'(\xi)>0 \Xx{on} \R \,, \quad
   \lim_{\abs{\xi}\to\infty} U'(\xi) = 0 \,.
 \end{multline}
 Then for any traveling wave solution $(\tilde{U},\tilde{c})$ of~\eqref{RD}
 with 
 \[
  \tilde{U}\in C(\R) \,, \quad
  \lim_{\xi\to\pm\infty} \tilde{U}(\xi) = \upm \XX{and}
  \um\leq \tilde{U} \leq\up \Xx{on} \R \,,
 \]
 we have $\tilde{c}=c$ and $\tilde{U}(\cdot)=U(\cdot+\xi_0)$ for some $\xi_0\in\R$.
\end{thm}

First we need to construct sub- and supersolutions.
\begin{lem} \label{lem:subsolution}
 Suppose $(U,c)$ is a traveling wave solution of~\eqref{RD} satisfying~\eqref{A:TWS}.
 Then there exists a small positive constant $\delta_*$ (which is independent of $U$)
 and a large positive constant $\sigma^*$ (which depends on $U$)
 such that for any $\delta\in(0,\delta_*]$ and every $\xi_0\in\R$,
 the functions $w^+$ and $w^-$ defined by
 \begin{equation} \label{function:w:pm}
  w^\pm(x,t) := U\big(x-c t+\xi_0 \pm\sigma^*\delta [1-\exp(-\beta t)]\big) \pm \delta\exp(-\beta t)
 \end{equation}
 with $\beta := \frac12 \min\{-f'(0),-f'(1)\}$ are a supersolution and a subsolution of~\eqref{RD}, respectively.
\end{lem}
\begin{proof}
 The function $w^+(x,t)$ with $y:= x-c t+\xi_0 +\sigma^*\delta [1-\exp(-\beta t)]$ satisfies
 \begin{align*}
  \diff{}{t}w^+ - \RieszFeller w^+ - f(w^+) 
  &= U'(y) \big(-c + \sigma^*\delta\beta \exp(-\beta t)\big) \\
   & \qquad - \delta\beta \exp(-\beta t) - \RieszFeller w^+ - f(w^+) \,; 
 \intertext{a traveling wave satisfies $- c U' = \RieszFeller U + f(U)$
  as well as $\RieszFeller U (y)= \RieszFeller w^+ (x,t)$, hence}
  &= \RieszFeller U(y) + f(U) - \RieszFeller w^+(x,t) - f(w^+) \\
   & \qquad + U'(y) \sigma^*\delta\beta \exp(-\beta t) - \delta\beta \exp(-\beta t) \\
  &= f(U) - f\big(U+\delta\exp(-\beta t)\big) + \delta\beta \exp(-\beta t) \big(U'(y) \sigma^* - 1\big) \,.
 \end{align*}
 Due to the properties~\eqref{A:TWS} of $U$, 
 there exists for any $\delta_*\in (0,\frac12)$ a constant $M=M(U)$ such that
 \begin{equation} \label{A:M}
  U(\xi)>1-\delta_* \Xx{for all} \xi\geq M\,, \quad 
  U(\xi)<\delta_* \Xx{for all} \xi\leq -M\,.
 \end{equation}
 We consider three cases 
 \[ \abs{y}\leq M \,, \quad y<-M \XX{and} y>M \,. \]
\begin{enumerate}
\item 
  In case $\abs{y}\leq M$, the estimate 
  \begin{multline*}
    f(U) - f\big(U+\delta\exp(-\beta t)\big)
    = - \delta\exp(-\beta t) \, \integrall{0}{1}{ f'\big(U+\theta\delta\exp(-\beta t)\big)}{\theta} \\
    \geq - \norm{f'}_{C([-1,2])} \, \delta\exp(-\beta t)    
  \end{multline*}
  yields 
  \begin{multline*}
    f(U) - f\big(U+\delta\exp(-\beta t)\big) + \delta\beta \exp(-\beta t) \big(U'(y) \sigma^* - 1\big) \\
    \geq \delta\exp(-\beta t) \big( - \norm{f'}_{C([-1,2])} + \beta \big(U'(y) \sigma^* - 1\big)\big) \,.
  \end{multline*}
  The last expression is non-negative, if $\sigma^*$ is chosen according to
  \begin{equation} \label{A:sigma*}
      \sigma^* \geq \sup_{\abs{y}\leq M} \frac{ \norm{f'}_{C([-1,2])} + \beta}{\beta U'(y)}
      = \frac{ \norm{f'}_{C([-1,2])} + \beta}{\beta \inf_{\abs{y}\leq M} U'(y)} \,,
  \end{equation}
  where $\inf_{\abs{y}\leq M} U'(y)$ is positive,
  since $U'$ is a continuous positive function and $\abs{y}\leq M$ is a compact subset.
  For $\sigma^*$ in~\eqref{A:sigma*}, $\diff{}{t}w^+ - \RieszFeller w^+ - f(w^+)\geq 0$ for all $\abs{y}\leq M$.
\item 
In case $y\geq M$,
 \begin{multline*}
  f(U) - f\big(U+\delta\exp(-\beta t)\big) + \delta\beta \exp(-\beta t) \big(U'(y) \sigma^* - 1\big) \\
    = \delta\exp(-\beta t) \bigg( \integrall{0}{1}{ - f'\big(U(y)+\theta\delta\exp(-\beta t)\big) - \beta}{\theta}
				    + \beta \sigma^* U'(y) \bigg)
 \end{multline*}
The last expression is non-negative, if $\delta\in(0,\delta_*]$
 and $\delta_*$ is chosen sufficiently small according to 
 \begin{equation} \label{A:delta_0:1a}
  \min_{u\in [1-\delta_*,1+\delta_*]} - f'(u) \geq \beta = \frac12 \min\{ -f'(0),-f'(1) \} \,,
 \end{equation}
 since $\beta \sigma^* U'(y)$ is non-negative anyway.
\item
In case $y\leq -M$,
 \begin{multline*}
  f(U) - f\big(U+\delta\exp(-\beta t)\big) + \delta\beta \exp(-\beta t) \big(U'(y) \sigma^* - 1\big) \\
    = \delta\exp(-\beta t) \bigg( \integrall{0}{1}{ - f'\big(U(y)+\theta\delta\exp(-\beta t)\big) - \beta}{\theta}
				    + \beta \sigma^* U'(y) \bigg)
 \end{multline*}
The last expression is non-negative, if $\delta\in(0,\delta_*]$
 and $\delta_*$ is chosen sufficiently small according to 
 \begin{equation} \label{A:delta_0:2a}
  \min_{u\in [0,2 \delta_*]} - f'(u) \geq \beta = \frac12 \min\{ -f'(0),-f'(1) \} \,,
 \end{equation}
 since $\beta \sigma^* U'(y)$ is non-negative anyway.
\end{enumerate}
Choosing $\delta_*$ sufficiently small 
 such that~\eqref{A:delta_0:1a} and~\eqref{A:delta_0:2a},
 then $M$ sufficiently large such that~\eqref{A:M} 
 and finally $\sigma^*$ sufficiently large such that~\eqref{A:sigma*} are satisfied, respectively, 
 we deduce that 
 \[ \diff{}{t}w^+ - \RieszFeller w^+ - f(w^+) \geq 0 \,. \]

In contrast, to prove that $w^-$ is a subsolution, i.e. 
 \[ \diff{}{t}w^- - \RieszFeller w^- - f(w^-) \leq 0 \,, \]
 we have to choose $\delta_*$ sufficiently small 
 such that
 \begin{equation} \label{A:delta_0:1b}
  \min_{u\in [1-2\delta_*,1]} - f'(u) \geq \beta = \frac12 \min\{ -f'(0),-f'(1) \} \,,
 \end{equation}
 and
 \begin{equation} \label{A:delta_0:2b}
  \min_{u\in [-\delta_*,\delta_*]} - f'(u) \geq \beta = \frac12 \min\{ -f'(0),-f'(1) \} \,,
 \end{equation}
 then $M$ sufficiently large such that~\eqref{A:M} 
 and finally $\sigma^*$ sufficiently large such that~\eqref{A:sigma*} are satisfied, respectively. 

Together, the result follows 
 if we choose $\delta_*$ sufficiently small 
 such that
 \begin{equation} \label{A:delta_0:1}
  \min_{u\in [1-2\delta_*,1+\delta_*]} - f'(u) \geq \beta = \frac12 \min\{ -f'(0),-f'(1) \} \,,
 \end{equation}
 and
 \begin{equation} \label{A:delta_0:2}
  \min_{u\in [-\delta_*,2\delta_*]} - f'(u) \geq \beta = \frac12 \min\{ -f'(0),-f'(1) \} \,,
 \end{equation}
 then $M$ sufficiently large such that~\eqref{A:M} 
 and finally $\sigma^*$ sufficiently large such that~\eqref{A:sigma*} are satisfied, respectively. 
\end{proof}

\begin{proof}[Proof of Theorem~\ref{thm:uniqueness}]
 The proof is taken from the article~\cite[Proof of Theorem 2.1]{Chen:1997},
 whereat we will highlight the differences.

 The problem~\eqref{RD} under consideration fulfills the assumptions~\ref{As:Chen:A2} and~\ref{As:Chen:A3}
  due to the discussion in Section~\ref{sec:TWP}.

 \emph{Step 1}.
 Since $U(\xi)$ and $\tilde{U}(\xi)$ have the same limits as $\xi\to\pm\infty$,
 there exist $\xi_1\in\R$ and $h\gg 1$ such that
 \[ U(\cdot +\xi_1)-\delta_* < \tilde{U}(\cdot ) < U(\cdot +\xi_1+h)+\delta_* \Xx{on} \R \,, \]
 where $\delta_*$ is taken from Lemma~\ref{lem:subsolution}.
 Considering the translated profile $U(\cdot +\xi_1)$ instead of $U$,
 we can set $\xi_1=0$ without loss of generality.
 Comparing $\tilde{U}(x-\tilde{c} t)$ with $w^\pm$ in~\eqref{function:w:pm} 
 (with $\xi_0=0$ for $w^-$ and $\xi_0=h$ for $w^+$),
 we obtain from Lemma~\ref{lem:subsolution} and Lemma~\ref{lem:comparison:Cb}
 \begin{multline*}
    U\big(x-c t -\sigma^*\delta_* [1-\exp(-\beta t)]\big) - \delta_* \exp(-\beta t) \\
    < \tilde{U} (x-\tilde{c} t)
    < U\big(x-c t +h +\sigma^*\delta_* [1-\exp(-\beta t)]\big) +\delta_* \exp(-\beta t)
 \end{multline*}
 for all $x\in\R$ and $t>0$. 
 Keeping $\xi := x-\tilde{c} t$ fixed, 
 sending $t\to\infty$,
 and using $\lim_{\xi\to\pm\infty} U(\xi) = \lim_{\xi\to\pm\infty} \tilde{U}(\xi) = \upm$,
 we then obtain from the first inequality that $c\geq\tilde{c}$ 
 and from the second inequality that $c\leq\tilde{c}$,
 so that $c=\tilde{c}$.
 In addition,
 \begin{equation} \label{ineq:profiles}
  U (\xi -\sigma^*\delta_*) < \tilde{U} (\xi) < U (\xi +h +\sigma^*\delta_*) \qquad \forall \xi\in\R \,.
 \end{equation} 

 \emph{Step 2}.
 Due to~\eqref{ineq:profiles}, the shifts 
 \[ \xi^* := \inf\set{\xi\in\R}{\tilde{U}(\cdot ) \leq U(\cdot +\xi)} \geq -\sigma^*\delta_* \]
 and
 \[ \xi_* := \sup\set{\xi\in\R}{\tilde{U}(\cdot ) \geq U(\cdot +\xi)} \leq h +\sigma^*\delta_* \]
 are well-defined and satisfy $\xi_*\leq\xi^*$. 
 To finish the proof, it suffices to show that $\xi_*=\xi^*$.
 To do this, we use a contradiction argument.
 Hence, we assume that $\xi_*<\xi^*$ and $\tilde{U}(\cdot )\not \equiv U(\cdot +\xi^*)$.

 Since we assume $\lim_{\abs{\xi}\to\infty} U'(\xi) = 0$,
 there exists a large positive constant $M_2=M_2(U)$ such that
 \begin{equation} \label{ineq:M2}
  2 \sigma^* U'(\xi) \leq 1 \Xx{if} \abs{\xi}\geq M_2 \,.
 \end{equation}
 The definition of $\xi^*$ implies $\tilde{U}(\cdot ) \leq U(\cdot +\xi^*)$.
   The functions $\tilde{U}(\cdot )$ and $U(\cdot +\xi^*)$ are stationary solutions of~\eqref{RD}
   whereat $\tilde{U}(\cdot )-U(\cdot +\xi^*)\in C([0,T];C_0(\R))$.
   Thus the comparison result in Lemma~\ref{lem:comparison:Cb} implies $\tilde{U}(\cdot ) < U(\cdot +\xi^*)$ on $\R$.
 Consequently, by the continuity of $U$ and $\tilde{U}$,
 there exists a small constant $\hat h\in (0,\frac1{2\sigma^*}]$ such that
 \begin{equation} \label{ineq:profiles:hat-h}
  \tilde{U}(\xi) < U(\xi+\xi^*-2\sigma^* \hat h) \qquad \forall \xi \xx{with} \abs{\xi+\xi^*}\leq M_2+1 \,.
 \end{equation}
 When $\abs{\xi+\xi^*}\geq M_2+1$, then for some $\theta\in[0,1]$  
 \begin{multline*}
  U(\xi+\xi^*-2\sigma^* \hat h) - \tilde{U}(\xi) 
   > U(\xi+\xi^*-2\sigma^* \hat h) - U(\xi+\xi^*) 
   = -2\sigma^* \hat h U'(\xi+\xi^*-2\theta\sigma^* \hat h)
   > -\hat h 
 \end{multline*}
 by the definition of $M_2$.
 Hence, in conjunction with~\eqref{ineq:profiles:hat-h},
 $U(\xi+\xi^*-2\sigma^* \hat h) +\hat h > \tilde{U}(\xi)$ on $\R$.
 Due to Lemma~\ref{lem:subsolution} and Lemma~\ref{lem:comparison:Cb}, for all $x\in\R$ and $t>0$,
 \begin{equation}
   U\big(x -c t +\xi^* -2\sigma^*\hat h +\sigma^*\hat h [1-\exp(-\beta t)]\big) + \hat h \exp(-\beta t) 
    > \tilde{U} (x-c t).
 \end{equation} 
 Keeping $\xi := x-c t$ fixed and sending $t\to\infty$,
 we obtain $U\big(\xi +\xi^* -\sigma^*\hat h \big) \geq \tilde{U} (\xi)$ for all $\xi\in\R$.
 But this contradicts the definition of $\xi^*$.
 Hence, $\xi_*=\xi^*$, which completes the proof of the theorem. 
\end{proof}

%% file: GLE-Stability.tex
We follow the proof of Chen in~\cite[Section 3]{Chen:1997}. 
In Section~\ref{sec:TWP}
 we studied the properties~\ref{As:Chen:C1}--\ref{As:Chen:C4} in case of $\cA[u] := \RieszFeller u + f(u)$.
Indeed the properties~\ref{As:Chen:C1}, \ref{As:Chen:C2} and~\ref{As:Chen:C4} are satisfied,
 whereas one estimate in~\ref{As:Chen:C3} has to be modified.
This implies that the properties~\ref{As:Chen:A1}--\ref{As:Chen:A3} and~\ref{As:Chen:B1}--\ref{As:Chen:B2} hold,
 whereas the estimate in property~\ref{As:Chen:B3} has to be modified. 

\begin{thm} \label{thm:stability}
Assume that~\ref{As:Chen:A1}--\ref{As:Chen:A3}, \ref{As:Chen:B1}--\ref{As:Chen:B2} and~\ref{As:Chen:C3b} hold.
Also assume that~\eqref{RD} has a traveling wave solution $(U,c)$ satisfying~\eqref{A:TWS},
 and 
\begin{equation} \label{As:delta} 
 0< \delta \leq \min\bigg\{ \min\{1,1/\sigma^*\} \frac{\delta_*}2\,, \frac{a^-}2\,, \frac{1-a^+}{2} \bigg\} \,. 
\end{equation}
Then there exists a positive constant $\kappa$ 
 such that for any $u_0\in C_b(\R)$ satisfying $0\leq u_0\leq 1$ and
 \begin{equation} \label{As:liminf}
  \liminf_{x\to +\infty} u_0(x) > 1-\delta > a^+\,, \qquad \limsup_{x\to -\infty} u_0(x) < \delta < a^-\,,
 \end{equation}
 the solution $u(x,t)$ of~\eqref{RD} with initial condition $u(\cdot,0)=u_0(\cdot)$ has the property that
 \[ \norm{u(\cdot,t)-U(\cdot-ct+\xi)}_{L^\infty(\R)} \leq K e^{-\kappa t} \quad \forall t\geq 0 \]
 where $\xi$ and $K$ are constants depending on $u_0$.
\end{thm}

\begin{proof}
We follow the four step procedure in~\cite[Proof of Theorem 3.1]{Chen:1997}.

\textbf{Step 1.} We prove that for any admissible $\delta>0$,
 there exist large positive constants $T$ and $H$ such that
 \begin{equation} \label{ineq:U:u}
   U(x-cT-H/2)-\delta \leq u(x,T) \leq U(x-cT+H/2)+\delta \qquad \forall x\in\R\,.
 \end{equation}
First, auxiliary smooth functions $w^\pm(x,t)$ are introduced in Lemma~\ref{lem:subsolution:auxiliary}
 which are constant except on a bounded interval.
The functions
 \begin{multline*}
  w^+(x,t) := (1+\delta) - [ 1-(a^- -2\delta) e^{-\epsilon t} ] \zeta(-\epsilon(x-\xi^+ +\cK t)) \\
  \XX{and} w^-(x,t) := -\delta + [ 1-(1-a^+ -2\delta) e^{-\epsilon t} ] \zeta(\epsilon(x-\xi^- -\cK t))
 \end{multline*}
 are a supersolution and a subsolution of~\eqref{RD}, respectively,
 for any $\delta\in (0\,, \min\{a^-/2\,, (1-a^+)/2\} ]$, $\xi^\pm\in\R$
 and appropriate constants $\epsilon=\epsilon(\delta)$ and $\cK=\cK(\delta)$.
Thus
 \begin{equation} \label{ineq:auxiliary}
   w^-(x,T) \leq u(x,T) \leq w^+(x,T) \qquad \forall x\in\R 
 \end{equation}
 will follow for a suitable choice of the parameters $\xi^\pm$ and Lemma~\ref{lem:comparison:Cb}.
In particular, we have to choose $\xi^\pm$ such that
 \begin{equation} \label{As:ID} 
   w^-(x,0) \leq u(x,0)=u_0(x) \leq w^+(x,0) \qquad \forall x\in\R\,.
 \end{equation}
Due to assumption~\eqref{As:liminf},
 the biggest $x^*\in\R$ such that $u_0(x^*)=1-\delta>a^+$ is a finite number.
Moreover, $w^-(x,0)\leq a^+ +\delta$ for all $x\in\R$
 where assumption~\eqref{As:delta} implies $a^+ +\delta\leq 1-\delta$.
Thus the choice $\xi^-=x^*$ implies the estimate
 \[ w^-(x,0)\leq \begin{cases}
                   -\delta &\xx{for all} x\leq x^*\,, \\
                   a^+ +\delta &\xx{for all} x\geq x^*\,,
                 \end{cases}
 \]
 hence $w^-(x,0)\leq u_0(x)$ for all $x\in\R$.

Again, due to assumption~\eqref{As:liminf},
 the smallest $x_*\in\R$ such that $u_0(x_*)=\delta$ is a finite number.
Moreover, $w^+(x,0)\geq a^- -\delta$ for all $x\in\R$
 where assumption~\eqref{As:delta} implies $\delta\leq a^- -\delta$.
Thus the choice $\xi^+=x_*$ implies the estimate
 \[ w^+(x,0)\geq \begin{cases}
                   a^- -\delta &\xx{for all} x\leq x_*\,, \\
                   1 +\delta &\xx{for all} x\geq x_*\,,
                 \end{cases}
 \]
 hence $u_0(x)\leq w^+(x,0)$ for all $x\in\R$. 
Consequently, for our choice of parameters $\xi^\pm$,
 Lemma~\ref{lem:comparison:Cb} implies estimate~\eqref{ineq:auxiliary} for all $T>0$.

Finally, we determine $T>0$, $H>0$ and $\delta_U>0$ such that
 \[ U(x-cT-H/2) -\delta_U \leq w^-(x,T) \XX{and} w^+(x,T) \leq U(x+cT+H/2) +\delta_U \]
 for all $x\in\R$ hold.
The functions $U(\cdot )$ and $w^\pm$ are continuous differentiable and monotone.
Moreover $w^-(x,t) := -\delta + [ 1-(1-a^+ -2\delta) e^{-\epsilon t} ] \zeta(\epsilon(x-\xi^- -\cK t))$
 satisfies $w^-(x,T)=1-\delta -(1-a^+ -2\delta) e^{-\epsilon T}\geq a^+ +\delta$ for all $x\geq \xi^- +\cK T+\frac4\epsilon$.
Choose $H^->0$ and $\delta_U$ such that
 \[ U(x-cT-H^-/2) -\delta_U = -\delta \Xx{for} x=\xi^- +\cK T+\frac4\epsilon\,. \]
Then $U(x-cT-H^-/2) -\delta_U \leq w^-(x,T)$ for all $x\in\R$, if 
 \begin{equation} \label{C:1}
  -\delta>-\delta_U \XX{and} a^+ +\delta \geq 1-\delta_U \,.
 \end{equation}
In contrast, $w^+(x,t) := (1+\delta) - [ 1-(a^- -2\delta) e^{-\epsilon t} ] \zeta(-\epsilon(x-\xi^+ +\cK t))$ 
 satisfies
 \[ w^+(x,T) = \begin{cases}
                 \delta +(a^- -2\delta) e^{-\epsilon T} &\xx{for all} x\leq \xi^+ -\cK T-\frac4\epsilon\,, \\
                 1+\delta &\xx{for all} x\geq \xi^+ -\cK T\,.
               \end{cases}
 \]
Choose $H^+>0$ and $\delta_U$ such that
 \[ U(x-cT+H^+/2) +\delta_U = 1+\delta \Xx{for} x=\xi^+ -\cK T-\frac4\epsilon\,. \]
Then $U(x-cT+H^+/2) +\delta_U \geq w^+(x,T)$ for all $x\in\R$, if 
 \begin{equation} \label{C:2}
  \delta_U \geq \delta +(a^- -2\delta) e^{-\epsilon T} \XX{and} 1+\delta_U>1+\delta \,.
 \end{equation}
The conditions $-\delta>-\delta_U$ and $1+\delta_U>1+\delta$ are equivalent to $\delta_U>\delta$.
We consider $d:=\delta_U-\delta>0$.
Inequality $\delta_U \geq \delta +(a^- -2\delta)$ implies condition $\delta_U \geq \delta +(a^- -2\delta) e^{-\epsilon T}$,
 hence we choose $d\geq a^- -2\delta> 0$.
Condition $a^+ +\delta \geq 1-\delta_U = 1-\delta-d$ is equivalent to $d\geq 1-a^+-2\delta> 0$.
On the one hand,
 we have to choose $\delta_U$ according to $\delta_U=\delta+d$ where $d\geq \max\{a^- -2\delta\,, 1-a^+-2\delta\}> 0$.  
On the other hand,
 we have to choose $\delta_U$ small enough
 such that the assumption $\delta\in(0,\min\{1,1/\sigma^*\} \delta_*/2]$ in Lemma~\ref{lem:bootstrap} is fulfilled.
These objectives can be met, if we consider $a^\pm$ as parameters which can be chosen sufficiently small.
Finally, due to the monotonicity of the functions $U(\cdot )$ and $w^\pm$,
 inequality~\eqref{ineq:U:u} will hold for the choice $H=\max\{H^-\,, H^+\}$.

Having established estimate~\eqref{ineq:U:u},
 the subsequent steps of the proof - where the exponential rate~$\kappa$ is derived - apply verbatim.
\end{proof}

In the sequel, $\zeta\in C^\infty(\R)$ is a fixed function having the following properties:
\begin{equation} \label{fct:zeta}
\begin{cases}
 \zeta(s)=0 &\Xx{if} s\leq 0\,, \\
 \zeta(s)=1 &\Xx{if} s\geq 4\,, \\
 0<\zeta'(s)<1 \xx{and} \abs{\zeta''(s)}\leq 1 &\Xx{if} s\in (0,4)\,.
\end{cases}
\end{equation}
\begin{lem} \label{lem:subsolution:auxiliary}
 Assume that~\ref{As:Chen:B1} holds.
 Then for every $\delta\in (0\,, \min\{a^-/2\,, (1-a^+)/2\} ]$,
 there exists a small positive constant $\epsilon=\epsilon(\delta)$ and 
 a large positive constant $\cK=\cK(\delta)$ such that, for every $\xi\in\R$,
 the function $w^+(x,t)$ and $w^-(x,t)$ defined by
 \begin{align*}
   w^+(x,t) &:= (1+\delta) - [ 1-(a^- -2\delta) e^{-\epsilon t} ] \zeta(-\epsilon(x-\xi+\cK t))\,, \\
   w^-(x,t) &:= -\delta + [ 1-(1-a^+ -2\delta) e^{-\epsilon t} ] \zeta(\epsilon(x-\xi-\cK t))\,,
 \end{align*}
 are respectively a supersolution and a subsolution of~\eqref{RD} in $\R\times (0,\infty)$.
\end{lem}
\begin{proof}
%  \emph{ 
%    We only prove the assertion of the lemma for $w^-$.
%    The proof for $w^+$ is analogous and is omitted.
%    By translational invariance, we need only consider the case $\xi=0$.
%    Since $\cA[w^-(x,t)1]=f(w^-(x,t))1$,
%    \begin{multline*}
%      \diff{w^-}{t}(x,t)-\cA[w^-(\cdot ,t)](x) \\
%        = - C\epsilon [ 1-(1-a^+ -2\delta) e^{-\epsilon t} ] \zeta'
%          + \epsilon (1-a^+ -2\delta) e^{-\epsilon t} \zeta
%          - \cA[w^-(\cdot ,t)](x) \\
%        \leq - C\epsilon a^+ \zeta' + \epsilon - f(w^-(x,t)) - \big[ \cA[w^-(\cdot ,t)](x)-\cA[w^-(x,t)1](x) \big] \,.
%    \end{multline*}
%  }
%  We estimate for $(x,t)\in \R\times (0,\infty)$ the quantity $\abs{\cA[w^-(\cdot ,t)](x)-\cA[w^-(x,t)1](x)}$,
%  whereat $w^-(\cdot ,t)$ denotes a function and $w^-(x,t)1$ a constant.
%  Thus
%  \begin{align*}
%   \Abs{\cA[w^-(\cdot ,t)](x)-\cA[w^-(x,t)1](x)}
%     &= \Abs{ \RieszFeller[w^-(\cdot ,t)](x) + f(w^-(x,t)) - \underbrace{\RieszFeller[\overbrace{w^-(x,t)1}^{\text{a constant}}](x)}_{=0} - f(w^-(x,t)) } \\
%     &= \Abs{ \RieszFeller[w^-(\cdot ,t)](x) }
%      \leq C \big[ \norm{\difff{w^-}{x}{2}}_{C_b} + \norm{\diff{w^-}{x}}_{C_b} \big] 
%      \leq C \epsilon \,,
%  \end{align*}
 We only prove the assertion of the lemma for $w^-$.
 The proof for $w^+$ is analogous and is omitted.
 By translational invariance, we need only consider the case $\xi=0$.

 We estimate 
 \[ \diff{w^-}{t}(x,t) - \cA[w^-(\cdot ,t)](x) = \diff{w^-}{t}(x,t) - \RieszFeller[w^-(\cdot ,t)](x) - f(w^-(x,t)) \,. \]
 On the one hand 
 \[
   \diff{w^-}{t}(x,t) 
     = - \cK\epsilon [ 1-(1-a^+ -2\delta) e^{-\epsilon t} ] \zeta'
       + \epsilon (1-a^+ -2\delta) e^{-\epsilon t} \zeta
     \leq - \cK\epsilon a^+ \zeta' + \epsilon \,.
 \]
 due to the assumptions on $\zeta$ and $\delta$.
 On the other hand,
 \begin{equation*}
   \Abs{ \RieszFeller[w^-(\cdot ,t)](x) }
     \leq \cK \Bigg[ \Norm{\difff{w^-}{x}{2}}_{C_b(\R\times [0,T])} + \Norm{\diff{w^-}{x}}_{C_b(\R\times [0,T])} \Bigg] 
     \leq \cK \epsilon \,,
 \end{equation*}
 due to proposition~\ref{prop:RieszFeller:estimate}, the assumptions on $\zeta$ and 
 \begin{align*} 
  \diff{w^-}{x} = \epsilon [ 1-(1-a^+ -2\delta) e^{-\epsilon t} ] \zeta'(\epsilon(x-\xi+\cK t))
 \intertext{as well as} 
  \difff{w^-}{x}{2} = \epsilon^2 [ 1-(1-a^+ -2\delta) e^{-\epsilon t} ] \zeta''(\epsilon(x-\xi+\cK t)) \,.
 \end{align*}
 Consequently, the estimate 
 \begin{equation} \label{ineq:w-}
   \diff{w^-}{t}(x,t) - \cA[w^-(\cdot ,t)](x)
     \leq - \cK_1 \epsilon a^+ \zeta' - f(w^-) + \cK_2 \epsilon 
 \end{equation}
 for some positive constants $\cK_1$ and $\cK_2$ follows.
 To show that $w^-$ is a subsolution,
 we have to find $\epsilon$ and $\cK_1$ such that the right-hand side of~\eqref{ineq:w-} is negative.
 This is possible by the same arguments as in~\cite[Proof of Lemma 3.2]{Chen:1997}.
\end{proof}
\begin{remark}
 The notions of sub- and supersolutions are only meaningful in presence of a comparison principle.
\end{remark}

\begin{lem} \label{lem:bootstrap}
Assume that the hypothesis of Theorem~\ref{thm:stability} hold
 and the constants $\delta_*$ and $\sigma^*$ are taken from Lemma~\ref{lem:subsolution}.
Then there exist a small positive constant $\epsilon^*$ (independent of $u_0$)
 such that if, for some $\tau\geq 0$, $\xi\in\R$, $\delta\in(0,\min\{1,1/\sigma^*\} \delta_*/2]$, and $h>0$, there holds
 \begin{equation} \label{ineq:sandwich}
  U(x-c\tau+\xi) -\delta \leq u(x,\tau) \leq U(x-c\tau+\xi+h)+\delta \qquad \forall x\in\R\,,
 \end{equation}
 then for every $t>\tau+1$,
 there exist $\hat\xi(t)$, $\hat\delta(t)$, and $\hat h(t)$ satisfying
 \begin{align*}
  & \hat\xi(t) \in [\xi-\sigma^*\delta\,, \xi+h+\sigma^*\delta]\,, \\ %\Xx{($\sigma^*$ is as in~\eqref{ineq:M2}),} \\
  & \hat\delta(t) \leq e^{-\beta(t-\tau-1)} [\delta+\epsilon^* \min\{h\,, 1\}]\,, \\
  & \hat h(t) \leq [h-\sigma^*\epsilon^* \min\{h\,, 1\}] + 2\sigma^* \delta\,,
 \end{align*}
 such that~\eqref{ineq:sandwich} holds with $(\tau,\xi,\delta,h)$ replaced by $(t,\hat\xi(t),\hat\delta(t),\hat h(t))$.
\end{lem}
\begin{proof}
Equation~\eqref{RD} is autonomous, 
 hence invariant with respect to spatial translations and time shifts.
Thus we set $\xi=0$ and $\tau=0$ without loss of generality and obtain
 \[ U(x) -\delta \leq u(x,0) \leq U(x+h)+\delta \qquad \forall x\in\R\,. \]
We want to deduce 
 \begin{equation} \label{ineq:sandwich:2}
    U\big(x-c t-\sigma^*\delta [1-e^{-\beta t}]\big) - \delta e^{-\beta t}
      \leq u(x,t)
      \leq U\big(x-c t+h+\sigma^*\delta [1-e^{-\beta t}]\big) + \delta e^{-\beta t}\,,
 \end{equation}
 with the help of Lemma~\ref{lem:comparison:Cb}.
First, the functions
 \begin{equation} \label{w:minus:lemma33}
    w^-(x,t) := U\big(x-c t-\sigma^*\delta [1-e^{-\beta t}]\big) - \delta e^{-\beta t}
     \xx{with} w^-(x,0) = U(x) -\delta\,,
 \end{equation}
 and
 \begin{equation} \label{w:plus:lemma33}
    w^+(x,t) := U\big(x-c t+h+\sigma^*\delta [1-e^{-\beta t}]\big) + \delta e^{-\beta t}
     \xx{with} w^+(x,0) = U(x+h)+\delta\,,
 \end{equation}
 are a subsolution and a supersolution of~\eqref{RD}, respectively, due to Lemma~\ref{lem:subsolution}.
Second, a solution $u(x,t)$ of~\eqref{RD} with initial datum $u(\cdot ,0)=u_0(\cdot )\in C_b(\R)$ and $0\leq u_0\leq 1$  
 satisfies $0\leq u(x,t)\leq 1$ for all $(x,t)\in\R\times(0,T]$,
 due to Theorem~\ref{thm:CP:existence:bistable}.
Finally, inequality~\eqref{ineq:sandwich:2} follows from the comparison principle in Lemma~\ref{lem:comparison:Cb}. 

Define $\overline h:=\min\{h,1\}$ and $\epsilon_1:= \frac12 \min_{\xi\in[0,2]} U'(\xi)$,
 such that $\integrall{0}{1}{U(y+\overline h)-U(y)}{y}\geq 2\epsilon_1 \overline h$.
Due to~\eqref{ineq:sandwich:2}, 
 at least one of the estimates
 \[ \integrall{0}{1}{u(y,0)-U(y)+\delta}{y}\geq \epsilon_1 \overline h +\delta
      \Xx{or} \integrall{0}{1}{U(y+\overline h)+\delta-u(y,0)}{y}\geq \epsilon_1 \overline h +\delta
 \]
 is true.
Here the first case is considered, whereas the second case is similar and omitted.
Comparing $u$ with $w^-$ in~\eqref{w:minus:lemma33}
 and using property~\ref{As:Chen:B2} - see also Lemma~\ref{lem:comparison:Cb} - yields 
 \begin{multline} \label{HI:lemma33}
    \min_{x\in [-M_2-2-\abs{c}\,, M_2+2+\abs{c}]} 
       u(x,1) -\big[ U\big(x-c-\sigma^*\delta [1-e^{-\beta}]\big) - \delta e^{-\beta} \big] \\
     \geq \eta(M_2+2+\abs{c}) \integrall{0}{1}{ u(y,0) -(U(y) - \delta) }{y} 
       \geq \eta \epsilon_1 \overline h + \eta \delta
 \end{multline}
with $\eta:=\eta(M_2+2+\abs{c})$.
Defining $\xi_1:= c+\sigma^*\delta [1-e^{-\beta}]$, which satisfies
 \[ -\abs{c} \leq \xi_1 = c+\sigma^*\delta [1-e^{-\beta}] \leq \abs{c} + \sigma^* \min\{1,1/\sigma^*\} \delta_*/2 \leq \abs{c} + 1\,,  \]
and
 \begin{equation} \label{def:epsilon:star}
    \epsilon^*:= \min\bigg\{ \frac{\delta_*}{2}\,,
        \frac1{2\sigma^*}\,, \min_{x\in [-M_2-2-2\abs{c}\,, M_2+2+2\abs{c}]} \frac{\eta\epsilon_1+\eta\delta/\overline{h}}{2\sigma^* U'(x)} \bigg\}
 \end{equation}
 yields
 \[
  U(x-\xi_1+2\sigma^*\epsilon^*\overline h) - U(x-\xi_1)
   = U'(\theta) 2\sigma^*\epsilon^* \overline h
   \leq \eta\epsilon_1 \overline h +\eta\delta
 \]
 for all $x\in [-M_2-1-\abs{c}\,, M_2+1+\abs{c}]$
 and some $\theta\in[x-\xi_1\,, x-\xi_1+2\sigma^*\epsilon^*\overline h]$.
Consequently, together with~\eqref{HI:lemma33},
 \[ u(x,1) \geq U(x-\xi_1+2\sigma^*\epsilon^*\overline h) - \delta e^{-\beta}
     \qquad \forall x\in [-M_2-1-\abs{c}\,, M_2+1+\abs{c}]\,.
 \]
In contrast, for $\abs{x}\geq M_2+1+\abs{c}$, the definition of $M_2$ in~\eqref{ineq:M2} yields 
 $U (x-\xi_1)\geq U (x-\xi_1+2\sigma^* \epsilon^*\overline h)-\epsilon^* \overline h$.
 Together with~\eqref{ineq:sandwich:2} for $t=1$ and the previous estimate, we obtain 
 \[ u(x,1) \geq U(x-\xi_1) - \delta e^{-\beta}
     \geq U(x-\xi_1+2\sigma^*\epsilon^*\overline h) - \big(\delta e^{-\beta} + \epsilon^* \overline h\big)
     \qquad \forall x\in \R\,.
 \]
Next, we want to show
 \[ u(x,1+\tau) \geq U(x-c\tau-\xi_1+2\sigma^*\epsilon^*\overline h-\sigma^* q(1-e^{-\beta\tau}))-q e^{-\beta\tau} =:w^-_2(x,\tau) \]
 for $q:= \delta e^{-\beta} + \epsilon^* \overline h$ and all $\tau\geq 0$.
The estimate $q= \delta e^{-\beta} + \epsilon^* \overline h \leq \delta_*$ and Lemma~\ref{lem:subsolution} imply that 
 \[
    w^-_2(x,\tau) %:= U(x-c\tau-\xi_1+2\sigma^*\epsilon^*\overline h-\sigma^* q(1-e^{-\beta\tau}))-q e^{-\beta\tau}
     \XX{with} w^-_2(x,0) := U(x-\xi_1+2\sigma^*\epsilon^*\overline h)-q
 \]
%  \[
%     w^-_2(x,\tau) := U(x-c\tau-\xi_1+2\sigma^*\epsilon^*\overline h-\sigma^* q(1-e^{-\beta\tau}))-q e^{-\beta\tau}
%      \xx{with} w^-_2(x,0) := U(x-\xi_1+2\sigma^*\epsilon^*\overline h)-q
%  \]
 is a subsolution of~\eqref{RD}.
Thus we deduce from Lemma~\ref{lem:comparison:Cb}, $u(x,1+\tau) \geq w^-_2(x,\tau)$ for all $\tau\geq 0$.
Furthermore we conclude
 \begin{multline*}
  u(x,1+\tau) \geq w^-_2(x,\tau) = U(x-c\tau-\xi_1+2\sigma^*\epsilon^*\overline h-\sigma^* q(1-e^{-\beta\tau}))-q e^{-\beta\tau} \\
    \geq U(x-c\tau-c+\sigma^*\epsilon^*\overline h-\sigma^* \delta)-e^{-\beta\tau}\big(\delta + \epsilon^* \overline h \big)\,,
 \end{multline*}
 using the definitions of $\xi_1$ and $q$, and the monotonicity of $U$.
Hence, setting $t=1+\tau$, $\hat\xi(t):= \sigma^*\epsilon^*\overline h-\sigma^* \delta$,
 and $\hat\delta(t)=e^{-\beta(t-1)}(\delta+\epsilon^*\overline h)$,
 we obtain from the last inequality the lower bound.
Whereas, estimate~\eqref{ineq:sandwich:2}
 with $\hat h(t):= [h+\sigma^*\delta(1-e^{-\beta t})]-\hat\xi (t) = h-\sigma^*\epsilon^*\overline{h}+\sigma^*\delta[2-e^{-\beta\tau}]$,
 implies the upper bound.
\end{proof}
Finally we prove that the speed $c$ of the traveling wave is bounded following the proof of~\cite[Theorem 3.5]{Chen:1997}.
Given the existence of a traveling wave solution, Hans Engler also proved that the wave speed has to be finite~\cite{Engler:2010}. 
\begin{thm}
Assume that~\ref{As:Chen:B1}, \ref{As:Chen:A2} and~\ref{As:Chen:C3b} hold.
Then for any traveling wave solution $(U,c)$ of~\eqref{RD}, 
 the wave speed $c$ satisfies
 \begin{equation} \label{est:speed} 
   \abs{c} \leq \speedI := \dfrac{\norm{f}_{C([0,1])}}{\epsI} \dfrac{3+a}{a}
 \end{equation}
 where $a:=\min\{a^-,1-a^+\}$ and $\epsI$ is a positive constant defined implicitly by
 \[ \rho(\epsI) := \cK \Big[ \epsI^2 \norm{\zeta''}_{C_b(\R\times [0,T])} + \epsI \norm{\zeta'}_{C_b(\R\times [0,T])} \Big]
      = \min\set{\abs{f(s)}}{ s\in [\tfrac{a}{3},\tfrac{2a}{3}]\cup [1-\tfrac{2a}{3},1-\tfrac{a}{3}]} \] 
 whereat the constant $\cK$ is determined in Proposition~\ref{prop:RieszFeller:estimate} and function $\zeta(s):=\tfrac12 [1+\tanh(s)]$.
\end{thm}
\begin{proof}
Estimate~\eqref{est:speed} will be proven with the help of explicit sub- and supersolutions in traveling wave form.
Due to assumption~\ref{As:Chen:B1}, $0<a\leq\tfrac12$ and 
 $\min\set{\abs{f(s)}}{ s\in [\tfrac{a}{3},\tfrac{2a}{3}]\cup [1-\tfrac{2a}{3},1-\tfrac{a}{3}]}>0$.
The traveling wave $U$ takes only values in $[0,1]$,
 hence we can modify $f$ without loss of generality such that $\norm{f}_{C([0,1])}= \norm{f}_{C([1,2])}$ as well as
 $f(u)=-f(\tfrac{a}{3})>0$ for $u\in[-1,-\tfrac{a}{3}]$ and $f(u)=-f(1-\tfrac{a}{3})<0$ for $u\in[1+\tfrac{a}{3},2]$.

To prove the upper bound $c\leq \speedI$, we will use a subsolution $w^-(x,t)$.
We recall the definition of $\epsI$ and $\speedI$ in the statement of the theorem
 and define $\zeta(s):=\tfrac12 [1+\tanh(s)]$, $\delta=\tfrac{a}{3}$, $w^-(x,t):= -2\delta + (1+\delta) \zeta(\epsI(x-\speedI t))$.
A direct calculation like in the proof of Lemma~\ref{lem:subsolution:auxiliary} yields
 \begin{multline} \label{est:speed:sub}
    \diff{w^-}{t}(x,t) - \cA[w^-(\cdot ,t)](x) = \diff{w^-}{t}(x,t) - \RieszFeller[w^-(\cdot ,t)](x) - f(w^-(x,t)) \\
     \leq -\epsI\speedI (1+\delta) \zeta'(\epsI(x-\speedI t)) + \rho(\epsI) - f(w^-(x,t))  
 \end{multline}
 where 
 \begin{multline*}
    \Abs{ \RieszFeller[w^-(\cdot ,t)](x) }
     \leq \cK \Bigg[ \Norm{\difff{w^-}{x}{2}}_{C_b(\R\times [0,T])} + \Norm{\diff{w^-}{x}}_{C_b(\R\times [0,T])} \Bigg] \\
     \leq \cK \Big[ \epsI^2 \norm{\zeta''}_{C_b(\R\times [0,T])} + \epsI \norm{\zeta'}_{C_b(\R\times [0,T])} \Big] 
     =: \rho(\epsI) \,,
 \end{multline*}
 due to proposition~\ref{prop:RieszFeller:estimate}.
To show that $w^-(x,t)$ is a subsolution, i.e. $\diff{w^-}{t}(x,t) - \cA[w^-(\cdot ,t)](x) \leq 0$ for all $(x,t)\in\R\times(0,\infty)$,
 we consider three subcases $\zeta\in(0,\tfrac{\delta}{1+\delta}]$, $\zeta\in(\tfrac{\delta}{1+\delta},\tfrac{1}{1+\delta})$
 and $\zeta\in [\tfrac{1}{1+\delta},1)$.
First $\zeta\in(0,\tfrac{\delta}{1+\delta}]$ implies that
 $w^-(x,t)= -2\delta + (1+\delta) \zeta(\epsI(x-\speedI t)) \in (-2\delta, -\delta]$
 hence $f(w^-(x,t))=-f(\tfrac{a}{3})>0$ and $f(w^-(x,t))\geq \rho(\epsI)$.
In this case the right hand side of~\eqref{est:speed:sub} is nonnegative.
Second $\zeta\in(\tfrac{\delta}{1+\delta},\tfrac{1}{1+\delta})$ implies that
 $w^-(x,t)= -2\delta + (1+\delta) \zeta(\epsI(x-\speedI t)) \in (-\delta,1-2\delta)$
 hence $f(w^-(x,t))$ has no definite sign.
However the right hand side of~\eqref{est:speed:sub} satisfies
 \[ -\epsI\speedI (1+\delta) \zeta'(\epsI(x-\speedI t)) + \rho(\epsI) - f(w^-(x,t))
      \leq -\epsI\speedI (1+\delta) \tfrac{2\delta}{(1+\delta)^2} + 2 \norm{f}_{C([0,1])} \leq 0
 \]
 for our choice of $\epsI$ and $\speedI$, 
 and the identity 
 \[ \min \set{ \zeta'(s) }{\zeta(s)\in(\tfrac{\delta}{1+\delta},\tfrac{1}{1+\delta})} = \frac{2\delta}{(1+\delta)^2} \]
 using $\zeta'(s)=\tfrac12 (1-\tanh^2(s)) = \tfrac12 (1-(2\zeta(s)-1)^2) = - 2 \zeta (\zeta-1)$.
Third $\zeta\in [\tfrac{1}{1+\delta},1)$ implies 
 $w^-(x,t)= -2\delta + (1+\delta) \zeta(\epsI(x-\speedI t)) \in [1-2\delta, 1-\delta)$
 hence $f(w^-(x,t))>0$ and $f(w^-(x,t))\geq \rho(\epsI)$.
In this case the right hand side of~\eqref{est:speed:sub} is nonnegative.
Therefore $\diff{w^-}{t}(x,t) - \cA[w^-(\cdot ,t)](x) \leq 0$ for all $(x,t)\in\R\times(0,\infty)$, hence $w^-(x,t)$ is a subsolution.

Like in the first step of the proof of Theorem~\ref{thm:stability},
 we can find $X\gg 1$ such that $U(\cdot )\geq w^-(\cdot -X,0)$
 and deduce $U(x-ct)\geq w^-(x-X,t)= w^-(x-\speedI t-X,0)$ for all $(x,t)\in\R\times(0,\infty)$
 from Lemma~\ref{lem:comparison:Cb}.
Setting $\xi=x-ct$ yields $U(\xi)\geq w^-(\xi +(c - \speedI)t -X,0)$ for all $(\xi,t)\in\R\times (0,\infty)$.
In case of $c\geq \speedI$ taking the limit $t\to\infty$ would lead to a contradiction with $U(\cdot )\geq w^-(\cdot -X,0)$,
 hence the estimate $c\leq \speedI$ follows.  

To prove the lower bound $-\speedI\leq c$,
 we use a supersolution $w^+(x,t):=\delta+(1+\delta) \zeta(\epsI(x+\speedI t))$.
\end{proof}

%% file: GLE-Existence.tex
\begin{thm} \label{thm:existence}
Assume that the assumptions~\ref{As:Chen:C1}, \ref{As:Chen:C2}, \ref{As:Chen:C3b} and~\ref{As:Chen:C4} hold.
There exists a traveling wave solution $(U,c)$ of~\eqref{RD} that satisfies~\eqref{A:TWS}.
\end{thm}
\begin{proof}
Step 1:
Consider the IVP 
\begin{equation} \label{IVP:zeta}
 \begin{cases}
  \diff{v}{t} = \cA[v] &\xx{in} \R\times(0,\infty) \,, \\
  v(\cdot ,0) = \zeta(\cdot )    &\xx{in} \R\,,
 \end{cases}
\end{equation}
where the function $\zeta$ is defined in~\eqref{fct:zeta}.
The idea is to show that for some diverging sequence $(t_j)_{j\in\N}$
 the sequence $(v(\cdot +\xi(t_j),t_j))_{j\in\N}$ - where $v(\xi(t),t)=a$ for all $t\geq 0$ - 
 has a pointwise limit $U(\cdot )$
 which is the profile of a traveling wave solution of~\eqref{RD}. 

The IVP~\eqref{IVP:zeta} has a unique solution $v\in C^\infty_b(\R\times (t_0,\infty))$ for any $t_0>0$
 due to Theorem~\ref{thm:CP:existence:bistable}, which satisfies 
 \begin{equation} \label{x-limits-v} 
  0\leq v(\cdot ,t)\leq 1\,, \qquad \lim_{x\to-\infty} v(x,t)=0 \XX{and} \lim_{x\to+\infty} v(x,t)=1 \Xx{for all} t\geq 0 
 \end{equation}
 due to Theorem~\ref{thm:far-field-behaviour}.
The function $v$ is monotone increasing in $x$, 
 since $v(x,0)=\zeta(x)\leq\zeta(x+h)=v(x+h,0)$ and the comparison principle~\ref{As:Chen:C2}.
The function $v$ is smooth for positive times, hence $v_x(x,t)>0$ for all $(x,t)\in\R\times(0,\infty)$;
 actually $v_x(x,t)\geq \eta(\abs{x},t)\zeta(1)>0$ for all $(x,t)\in\R\times(0,\infty)$ 
 follows from studying the difference quotients $\frac{v(x+h,t)-v(x,t)}{h}$ with the help of~\ref{As:Chen:C2}.
Then the implicit function theorem implies the existence of a smooth function $z:(0,1)\times (0,\infty)\to \R$, %$(a,t)\mapsto z(a,t)$, 
 such that $v(z(\tilde a,t),t)=\tilde a$ for all $(\tilde a,t)\in (0,1)\times (0,\infty)$.
The following three lemmas can be proved in the same way as in Step 2 of the proof of~\cite[Theorem 4.1]{Chen:1997}.
\begin{lem} \label{lem:est:z}
Under the assumptions of Theorem~\ref{thm:existence},
 there exist a large positive constant $\delta_1$ and a function $m_1:(0,\delta_1/2]\to (0,\infty)$ 
 such that for all $\delta\in(0,\delta_1/2]$ the function $z:(0,1)\times (0,\infty)\to \R$ satisfies
 \begin{equation} \label{est:z}
  z(1-\delta,t)-z(\delta,t) \leq m_1(\delta) \qquad \forall t\geq 0\,.
 \end{equation}
\end{lem}

\begin{lem}
Under the assumptions of Theorem~\ref{thm:existence},
 for every $M>0$ there exists a constant $\hat\eta(M)>0$ such that
 \begin{equation} \label{est:dvdx}
  \diff{v}{x}(x+z(a,t),t)\geq \hat\eta(M) \qquad \forall t\geq 1\,, \quad x\in[-M,M]\,.
 \end{equation}
\end{lem}

Similar to Lemma~\ref{lem:subsolution} sub- and supersolutions of~\eqref{RD} are constructed.
\begin{lem} \label{lem:subsolution:W}
Under the assumptions of Theorem~\ref{thm:existence},
 there exists a small positive constant $\delta_0$ %(which is independent of $U$)
 and a large positive constant $\sigma_2$ %(which depends on $U$)
 such that for any $\delta\in(0,\delta_0]$ and every $\xi\in\R$,
 the functions $W^+$ and $W^-$ defined by
 \begin{equation} \label{function:W:pm}
  W^\pm(x,t) := v\big(x+\xi \pm\sigma_2\delta [1-\euler^{-\beta t}]\big) \pm \delta\euler^{-\beta t}
 \end{equation}
 with $\beta := \frac12 \min\{-f'(0),-f'(1)\}$ are a supersolution and a subsolution of~\eqref{RD}, respectively.
\end{lem}

\begin{lem}
Under the assumptions of Theorem~\ref{thm:existence},
 there exists a sequence $(t_j)_{j\in\N}$ and a non-decreasing function $U:\R\to(0,1)$,
 such that $(t_j)_{j\in\N}$ diverges to $+\infty$ as $j\to+\infty$ and 
 \[ \lim_{j\to\infty} v(\xi+z(a,t_j),t_j) = U(\xi) \Xx{for all} \xi\in\R\,. \]
Moreover $U$ satisfies $\lim_{\xi\to -\infty} U(\xi) = 0$ and $\lim_{\xi\to +\infty} U(\xi) = 1$. %the properties~\eqref{A:TWS}.
\end{lem}
\begin{proof}
The sequence $\{f_k(\cdot ):=v(\cdot +z(a,k),k)\}_{k\in\N}$ of real-valued functions on $\R$
 consists of bounded functions which are uniformly equicontinuous.
Due to the Arzela-Ascoli theorem there exists a subsequence $\{k_j\}_{j\in\N}$
 and a bounded continuous function $U\in C_b(\R)$
 such that $f_{k_j}(\cdot )=v(\cdot +z(a,k_j),k_j)\to U(\cdot )$ for $j\to\infty$
 uniformly on compact subsets of $\R$.
Obviously, the function $U$ inherits from the function $v$ the properties $U(0)=a$, $0\leq U\leq 1$,
 and to be non-decreasing in~$x$.
For sufficiently small positive $\delta$ estimate~\eqref{est:z} implies
 $U(-m_1(\delta))\leq \delta$ and $U(m_1(\delta))\geq 1-\delta$
 consequently $\lim_{\xi\to-\infty} U(\xi)=0$ and $\lim_{\xi\to\infty} U(\xi)=1$.

First we show $U(\xi)\leq \delta$ for all $\xi\leq -m_1(\delta)$:
Estimate~\eqref{est:dvdx} implies
 $U(\xi+h)-U(\xi)\geq \hat\eta(\abs{\xi}+1)h\geq 0$ for all $h\in[0,1]$ and all $\xi\in\R$.
Therefore we only need to show $U(-m_1(\delta))\leq \delta$,
 whereat $v(-m_1(\delta)+z(a,k_j),k_j)\to U(-m_1(\delta))$ for $j\to\infty$.
The function $v(x,t)$ is monotone increasing in the first argument, 
 hence the function $z(\tilde a,t)$ is monotone increasing in its first argument as well.
Due to Lemma~\ref{lem:est:z} for $\delta\in(0,\delta_1/2]$ with $\delta<a<1-\delta$ we deduce $z(\delta,t)<z(a,t)<z(1-\delta,t)$,
 \[ -m_1(\delta)+z(a,t)\leq -z(1-\delta,t)+z(\delta,t)+z(a,t)< z(\delta,t)\,, \]
 $v(-m_1(\delta)+z(a,k_j),k_j)< v(z(\delta,t),t) = \delta$ and finally
 \[ v(-m_1(\delta)+z(a,k_j),k_j)\to U(-m_1(\delta)) \leq \delta \Xx{for} j\to\infty\,. \] 
In a similar way we show $v(\xi+z(a,t))> 1-\delta$ for all $\xi> m_1(\delta)$
 and deduce $U(\xi)\geq 1-\delta$ for all $\xi> m_1(\delta)$. 

Moreover the convergence $\lim_{j\to\infty} v(\xi+z(a,t_j),t_j) = U(\xi)$ is uniform on $\R$:
For sufficiently small $\delta>0$ we deduce for all $j\in\N$ that
\[ \abs{ U(\xi) - v(\xi+z(a,t_j),t_j) } \leq \abs{U(\xi)} + \abs{v(\xi+z(a,t_j),t_j)} \leq \delta \qquad \forall \xi\leq -m_1(\delta/2) \]
and $\abs{ U(\xi) - v(\xi+z(a,t_j),t_j) } \leq \abs{1-U(\xi)} + \abs{1-v(\xi+z(a,t_j),t_j)} \leq \delta \qquad \forall \xi\geq m_1(\delta/2)$. 
Due to the uniform convergence on compact intervals, 
 we can choose $J(\delta)$ sufficiently large such that
\[ \abs{ U(\xi) - v(\xi+z(a,t_j),t_j) } \leq \delta \qquad \forall \xi\in [-m_1(\delta/2),m_1(\delta/2)] \XX{and} \forall j\geq J(\delta) \]
hence - using the short hand notation $w(\xi,t_j):=U(\xi) - v(\xi+z(a,t_j),t_j)$ - 
\[ \sup_{\xi\in\R} \abs{w(\xi,t_j)} =
     \max \{ \sup_{\xi\in (-\infty,-m_1(\frac{\delta}{2}))} \abs{w(\xi,t_j)},
             \sup_{\xi\in[-m_1(\frac{\delta}{2}),m_1(\frac{\delta}{2})]} \abs{w(\xi,t_j)}, \sup_{\xi\in(m_1(\frac{\delta}{2}),\infty)} \abs{w(\xi,t_j)} \}
     \leq \delta
\]
for all $j\geq J(\delta)$.

More precisely, the solution $v$ is a smooth function for positive times and has uniformly bounded derivatives
 due to Theorem~\ref{thm:CP:existence:bistable}.
Therefore the Arzela-Ascoli Theorem implies that $U\in C^m_b(\R)$ of any order $m\in\N$
 and the existence of a diverging sequence $\{k_j\}_{j\in\N}$
 such that $f_{k_j}(\cdot )=v(\cdot +z(a,k_j),k_j)\to U(\cdot )$ for $j\to\infty$
 uniformly w.r.t. the $C^m$ norm on compact subsets of $\R$.
Moreover the function $v$ converges to constant endstates, 
 whereat its spatial derivative of any order converge to zero in the limits $x\to\pm\infty$.
These properties are passed on to the function $U$ and
 - as before with the help of Lemma~\ref{lem:est:z} -
 the convergence $f_{k_j}(\cdot )=v(\cdot +z(a,k_j),k_j)\to U(\cdot )$ for $j\to\infty$
 turns out to be uniform on $\R$.
\end{proof}
\begin{lem}
Under the assumptions of the previous lemma,
 the function $U$ is the profile of a traveling wave solution of~\eqref{RD} and satisfies~\eqref{A:TWS}. 
\end{lem}
\begin{proof}
The IVP 
\begin{equation} \label{IVP:U}
 \begin{cases}
  \diff{\tilde{U}}{t} = \cA[\tilde{U}] &\xx{in} \R\times(0,\infty) \,, \\
  \tilde{U}(\cdot ,0) = U(\cdot )    &\xx{in} \R\,,
 \end{cases}
\end{equation}
has a unique solution $\tilde U\in C^\infty_b(\R\times (t_0,\infty))$ for any $t_0>0$
 due to Theorem~\ref{thm:CP:existence:bistable}.
First we need to establish  
 \begin{equation} \label{limit:tildeU}
  \lim_{j\to\infty} v(\xi+z(a,t_j),t_j+t) = \tilde{U}(\xi,t) \Xx{for all} (\xi,t)\in\R\times(0,\infty)\,.
 \end{equation}
For any $\hat\epsilon>0$ there exists $J(\hat\epsilon)$ such that if $j>J(\hat\epsilon)$ then
 \[ v(\cdot -\hat\epsilon+z(a,t_j),t_j) -\hat\epsilon < U(\cdot ) < v(\cdot +\hat\epsilon+z(a,t_j),t_j) +\hat\epsilon\,. \]
Considering these functions as the initial data of the IVP~\eqref{IVP:U},
 we obtain from Lemma~\ref{lem:subsolution:W} the estimate
 \begin{multline*}
   v(\cdot -\hat\epsilon+z(a,t_j)-\sigma_2\hat\epsilon [1-\euler^{-\beta t}],t_j+t) -\hat\epsilon\euler^{-\beta t} \leq \tilde U(\cdot ,t) \\
     \leq v(\cdot -\hat\epsilon+z(a,t_j)+\sigma_2\hat\epsilon [1-\euler^{-\beta t}],t_j+t) +\hat\epsilon\euler^{-\beta t}\,.
 \end{multline*}
Noticing that $\tilde U$ is smooth and taking the limit $\hat\epsilon\to 0$
 and then $j\to\infty$ yields statement~\eqref{limit:tildeU}.
The first estimate is rewritten as 
 \[ v(\cdot +z(a,t_j),t_j+t) -\hat\epsilon\euler^{-\beta t} \leq \tilde U(\cdot +\hat\epsilon+\sigma_2\hat\epsilon [1-\euler^{-\beta t}],t) \] 
 taking the limits yields 
 \[ \limsup_{j\to\infty} v(\cdot +z(a,t_j),t_j+t) \leq \tilde U(\cdot ,t)\,. \]
Using the second estimate yields 
 \[ \liminf_{j\to\infty} v(\cdot +z(a,t_j),t_j+t) \geq \tilde U(\cdot ,t)\,. \]
Taken together
 \[ \tilde U(\cdot ,t) \leq \liminf_{j\to\infty} v(\cdot +z(a,t_j),t_j+t) \leq \limsup_{j\to\infty} v(\cdot +z(a,t_j),t_j+t) \leq \tilde U(\cdot ,t)\]
 we deduce statement~\eqref{limit:tildeU}. 
The monotonicity of $v$ w.r.t to $x$ and its limiting behavior
 allow to find a large positive constant $m_0$ such that $v(\cdot -m_0,1)-\delta_0\leq v(\cdot ,0)\leq v(\cdot +m_0,1)+\delta_0$.
Again a comparison principle and Lemma~\ref{lem:subsolution:W} imply 
 \[
  v(\cdot -m_0-\sigma_2\delta_0(1-\euler^{-\beta t}),t+1)-\delta_0\euler^{-\beta t}\leq v(\cdot ,t)
    \leq v(\cdot +m_0+\sigma_2\delta_0(1-\euler^{-\beta t}),t+1)+\delta_0\euler^{-\beta t}
 \]
 whereat evaluating at $\xi+z(a,t)$, setting $t=t_j$ and taking the limit $j\to\infty$ yields
 \begin{equation} \label{estimate:4.10}
  \tilde{U} (\xi-m_0-\sigma_2 \delta_0, 1) \leq U(\xi) \leq \tilde{U} (\xi+m_0+\sigma_2 \delta_0, 1)
    \Xx{for all} \xi\in\R\,. 
 \end{equation}
To prove the statement we show that $\tilde{U}(\cdot ,t)=U(\cdot -ct)$ for some $c\in\R$ and all $t$.
Due to estimate~\eqref{estimate:4.10} the numbers 
 \[ \xi_*:=\sup\set{\xi\in\R}{\tilde{U}(\cdot +\xi,1)\leq U(\cdot )}
     \Xx{and} \xi^*:=\inf\set{\xi\in\R}{U(\cdot )\leq \tilde{U}(\cdot +\xi,1)} 
 \]
are well-defined and satisfy $ -m_0-\sigma_2 \delta_0 \leq \xi_*\leq\xi^* \leq m_0+\sigma_2 \delta_0$.
However $\xi_*=\xi^*$ arguing as in the proof of Theorem~\ref{thm:uniqueness}.
In particular we noted that $U\in C^\infty_b(\R)$ 
 and for some diverging sequence $\{k_j\}_{j\in\N}$
 the convergence $v(\cdot +z(a,k_j),k_j)\to U(\cdot )$ for $j\to\infty$
 is uniform w.r.t. the $C^m_b(\R)$-norm for any order $m\in\N$.
In a similar way we can establish that $\lim_{x\to\pm\infty} \tilde U_x(x,t) = 0$ for all $t\geq 0$ 
 and the uniform convergence $v(\cdot +z(a,t_j),t_j+t)\to \tilde U(\cdot ,t)$ w.r.t. the $C^1_b(\R)$-norm for all $t>0$. 

Comparing $\tilde{U}(\cdot ,t)$ with $U(\cdot )$ for $t\in(1,2]$ in the same way
 one obtains the existence of a function $c:[1,2]\to\R$ with $c(1)=\xi_*=\xi^*$
 such that $\tilde{U}(\cdot ,t)=U(\cdot -c(t))$.
The function $c$ is differentiable and equation $\diff{\tilde{U}}{t} = \cA[\tilde{U}]$ implies 
 $-c'(t)U'(\xi) = \cA[U](\xi)$.
The left hand side of the identity does not depend on $t$ explicitly (only through $\xi$),
 hence $c'(t)$ is constant for all $t$ and $(U,c')$ is a traveling wave solution of~\eqref{RD}.
To establish the properties of $U'$ in~\eqref{A:TWS},
 we notice that $\tilde{U}$ and hence $U$ are bounded smooth functions 
 approaching constant endstates in the limits $\xi\to\pm\infty$.  
\end{proof}
\end{proof}

%% file: GLE-LevyProcesses+Semigroups.tex
The following section is a verbatim excerpt of the book ``L\'evy processes and infinitely divisible distributions''
 by Ken-iti Sato~\cite{Sato:1999}.

\begin{defn}\cite[Definition 7.1]{Sato:1999}
 A probability measure $\mu$ on $\R^d$ is \emph{infinitely divisible} if, for any positive integer $n$, 
 there is a probability measure $\mu_n$ on $\R^d$ such that $\mu$ is the $n$-fold convolution of~$\mu_n$,
 i.e. $\mu = \underbrace{\mu_n \ast \ldots \ast \mu_n}_{{=n-times}}$.
\end{defn}

The Fourier transform of a convolution of probability measures is the product of their characteristic functions.
Thus a probability measure $\mu$ is infinitely divisible if and only if, 
for each $n$, an $n$th root of the characteristic function $\hat{\mu}(z)$ can be chosen in such a way 
that it is the characteristic function of some probability measure. 
Here we use the short hand notation $\hat\mu(z)=\Fourier[\mu](z)$
 for the (extension of the) Fourier transform (to measures).
\newline

A stable probability measure is a special case of an infinitely divisible probability measure~\cite[Definition 13.1]{Sato:1999},
whose characteristic function are determined by the L\'evy-Khintchine formula~\eqref{eq:LKF}.
Let $D:=\{x\in \R^d \; | \; |x|\leq 1 \}$ be the closed unit ball.
\begin{thm}\cite[Theorem 8.1]{Sato:1999} \label{thm:LK}
 \begin{enumerate}
  \item \label{LKformula} 
   If $\mu$ is an infinitely divisible distribution on $\R^d$, then
   \begin{multline} \label{eq:LKF}
    \hat{\mu}(z)= \exp\bigg[ -\frac{1}{2} <z,Az> + i <\gamma,z> \\
     + \int_{\R^d} (e^{i<z,x>} - 1 - i<z,x> 1_D(x)) \nu(dx) \bigg], \quad z\in\R^d \,,                 
   \end{multline}
   where $A$ is a symmetric nonnegative-definite $d \times d$ matrix, $\gamma\in \R^d$, 
   $\nu$~is a measure on $\R^d$ satisfying
   \begin{equation} \label{eq:LevyMeasure}
    \nu(\{0\}) = 0 \quad \text{and} \quad \int_{\R^d} min(1,|x|^2) \nu(dx) < \infty \,.
   \end{equation}
  \item The representation of $\hat{\mu}(z)$ in \eqref{LKformula} by $A$, $\nu$, and $\gamma$ is unique.
  \item Conversely, if $A$ is a symmetric nonnegative-definite $d \times d$ matrix, 
   $\nu$ is a measure satisfying \eqref{eq:LevyMeasure}, and $\gamma\in\R^d$, 
   then there exists an infinitely divisible distribution $\mu$ whose characteristic function is given by~\eqref{eq:LKF}.
 \end{enumerate}
\end{thm}
The triplet $(A,\nu,\gamma)$ in Theorem~\ref{thm:LK} is called the \emph{generating triplet} of $\mu$. 
The matrix $A$ and measure $\nu$ are called, respectively, 
 the \emph{Gaussian covariance matrix} and the \emph{L\'evy measure} of $\mu$. 
When $A=0$, $\mu$ is called \emph{purely non-Gaussian}~\cite[Definition 8.2]{Sato:1999}.
\begin{remark} \label{remark:LevyMeasure1D}
If $d=1$, then $S=\{-1,1\}$ and any non-trivial $\alpha$-stable distribution with $0<\alpha<2$ has absolutely continuous L\'evy measure 
\begin{equation}
\label{eq:LevyMeasure1D}
\nu(dx) = 
 \begin{cases}
  c_1 x^{-1-\alpha} & \text{on } (0,\infty)\,, \\
  c_2 |x|^{-1-\alpha} & \text{on } (-\infty,0)\,,
 \end{cases}
\end{equation}
with $c_1\geq 0$, $c_2\geq 0$, $c_1+c_2 >0$ due to~\cite[Theorem 14.3]{Sato:1999}.
\end{remark}
\begin{remark}\cite[Remark 8.4]{Sato:1999} \label{remark:AlternativeRepresentation}
 The integrand of the integral in the right-hand side of~\eqref{eq:LKF} is integrable with respect to $\nu$, 
 because it is bounded outside of any neighborhood of $0$ and 
 \begin{equation}
  e^{i<z,x>} - 1 - i<z,x> 1_D(x) = O(|x|^2) \quad \text{as} \quad |x|\to 0
 \end{equation}
 for fixed $z$. 
 There are many other ways of getting an integrable integrand. 
 Let $c(x)$ be a bounded measurable function from $\R^d$ to $\R$ satisfying 
 \begin{align}
  c(x) &= 1+o(|x|) && \text{as } |x|\to 0 \,, \\
  c(x) &= O(1/|x|) && \text{as } |x|\to \infty \,.
 \end{align}
 Then~\eqref{eq:LKF} is rewritten as 
 \begin{multline} \label{eq:LKF:C}
  \hat{\mu}(z)= \exp\bigg[ -\frac{1}{2} <z,Az> + i <\gamma_c,z> \\
   + \int_{\R^d} (e^{i<z,x>} - 1 - i<z,x> c(x)) \nu(dx) \bigg]\,, \quad z\in\R^d\,,                 
 \end{multline}
 with $\gamma_c\in\R^d$ defined by
 \begin{equation}
  \gamma_c = \gamma + \int_{\R^d} x(c(x)-1_D(x))\; \nu(dx) \,.
 \end{equation}
 We denote the triplet in~\eqref{eq:LKF:C} by $(A,\nu,\gamma_c)_c$. 
 It is also a generating triplet and~\eqref{eq:LKF:C} is also called a L\'evy-Khintchine representation. 
 \newline 

Alternative representations: 
\begin{enumerate}
 \item 
  If $\nu$ satisfies the additional condition 
  \begin{equation} \label{eq:condition:0}
    \int_{|x|\leq 1} |x| \, \nu(dx) < \infty\,, 
  \end{equation}
  then using the zero function as $c$, we [are] getting
  \begin{multline} \label{eq:LKF:0}
   \hat{\mu}(z)= \exp\bigg[ -\frac{1}{2} <z,Az> + i <\gamma_0,z> 
    + \int_{\R^d} (e^{i<z,x>} - 1) \nu(dx) \bigg]\,, \quad z\in\R^d \,,                
  \end{multline}
  with $\gamma_0 \in\R^d$. 
  This is the representation by the triplet $(A,\nu,\gamma_0)_0$. 
  The constant $\gamma_0$ here is called the \emph{drift} of $\mu$. 
 \item 
  If $\nu$ satisfies the additional condition 
  \begin{equation} \label{eq:condition:1}
   \int_{|x|>1} |x| \, \nu(dx) < \infty\,,
  \end{equation} 
  then letting $c(x)$ be the constant function $1$, 
  we have the representation by the triplet $(A,\nu,\gamma_1)_1$:
  \begin{multline} \label{eq:LKF:1}
   \hat{\mu}(z)= \exp\bigg[ -\frac{1}{2} <z,Az> + i <\gamma_1,z> \\
    + \int_{\R^d} (e^{i<z,x>} - 1 - i<z,x>) \nu(dx) \bigg], \quad z\in\R^d,                 
  \end{multline}
  We will call the constant $\gamma_1$ the \emph{center} of $\mu$. 
  It will be shown in~\cite[Chapter 5, Example (25.23)]{Sato:1999} 
  that finiteness of $\int_{|x|>1} |x| \, \nu(dx)$ is equivalent to 
  finiteness of the mean of $\mu$, $\int_{\R^d} x \, \mu(dx)$, 
  and that $\gamma_1=\int_{\R^d} x \, \mu(dx)$. 
  Thus the center and the mean are identical. 
\end{enumerate}
 We note that [in the triplets] $A$ and $\nu$ are invariant no matter what function $c(x)$ we choose. 
\end{remark}

Moreover, for every infinitely divisible distribution $\mu$ on $\R^d$,
 there is a L\'evy process $(X_t)_{t\geq 0}$ (such that $P_{X_1}=\mu$),
 which is unique up to identity in law~\cite[Corollary 11.6]{Sato:1999}.
This L\'evy process $(X_t)_{t\geq 0}$ is called the L\'evy process corresponding to $\mu$.
The generating triplet $(A,\nu,\gamma)$ of $\mu$ is called
 the generating triplet of the L\'evy process $(X_t)_{t\geq 0}$.

Suppose that ${X_t}$ is a L\'evy process on $\R^d$ corresponding to an infinitely divisible distribution $\mu=P_{X_1}$. 
The transition function $P_t(x,B)$ is defined by
\begin{equation} \label{eq:TS} % transition semigroup
 P_t(x,B) = \mu^t(B-x) \quad \text{for } t\geq 0, \, x\in\R^d, \, B\in\mathcal{B}(\R^d)\,,
\end{equation}
as in \cite[Equation (10.8)]{Sato:1999}. 
Define for $f\in C_0(\R^d)$, i.e. $f\in C(\R^d)$ and $\lim_{|x|\to\infty} f(x)=0$,
\[
 (P_t f)(x) = \int_{\R^d} P_t(x, \dy) f(y) = \int_{\R^d} \mu^t( \dy) f(x+y) = E[f(x+X_t)]\,.
\]
Then $P_t f\in C_0(\R^d)$ by the Lebesgue convergence theorem. 
The following is a major result in the theory of L\'evy processes~\cite[Section 31]{Sato:1999}.
\begin{thm} \cite[Theorem 31.5]{Sato:1999} \label{thm:IG}
 Suppose $\{X_t\}$ is a L\'evy process on~$\R^d$ with generating triplet $(A,\nu,\gamma)$, 
 whereat $A=(A_{jk})\in\R^{d\times d}$ and $\gamma=(\gamma_j)\in\R^d$.
 The associated family of operators $\{P_t \,|\, t\geq 0 \}$ in~\eqref{eq:TS}
 is a strongly continuous semigroup on $C_0(\R^d)$ with norm $\|P_t\|=1$. 
 Let $L$ be its infinitesimal generator. 
 Then $C^{\infty}_c(\R^d)$ is a core of $L$, $C^2_0(\R^d) \subset \mathcal{D}(L)$, and
 \begin{multline} \label{eq:IG} % infinitesimal generator
  Lf(x) = \frac{1}{2} \sum_{j,k=1}^{d} A_{jk} \frac{\partial^2 f}{\partial x_j \partial x_k}(x) 
   + \sum_{j=1}^{d} \gamma_{j} \frac{\partial f}{\partial x_j}(x) + \\
   + \int_{\R^d} \bigg( f(x+y) - f(x) - \sum_{j=1}^{d} y_{j} \frac{\partial f}{\partial x_j}(x) 1_D(y) \bigg) \; \nu( \dy)
 \end{multline}
 for $f\in C^2_0(\R^d)$ and $D=\{x\in\R^d \;|\; |x|\leq 1\}$.
\end{thm}
The semigroup $\{P_t\}$ on $C_0(\R^d)$ is called transition semigroup of $\{X_t\}$. 
\begin{remark} \label{remark:semigroup:differences}
  If the probability measure $\mu^t$ has a probability density function $K(\cdot ,t)$,
  then the transition semigroup $\{P_t\}$ of the L\'evy process $\{X_t\}$ satisfies
  \begin{align*}
    (P_t f)(x) 
      &= \int_{\R^d} P_t(x, \dy) f(y) = \\
      &= \int_{\R^d} \mu^t( \dy) f(x+y) = \int_{\R^d} f(x+y) K(y,t) \dy
    \intertext{and using the substitution $z=-y$} 
      &= \int_{\R^d} f(x-z) K(-z,t) \d[z] = (f\ast K(-\cdot ,t))(x) \,. 
  \end{align*}
  Thus we have to be careful,
  when we compare this transition semigroup with convolution semigroup generated by $\RieszFeller$.
\end{remark}

\subsection{Semigroup generated by a Riesz-Feller operator}
In the following, we collect the most important facts on Feller and Markov semigroups 
 from the books~\cite{Applebaum:2004} and~\cite{Jacob:2001}.
% Feller semigroup
$B_b(\R^d)$ is the space of bounded Borel measurable functions from $\R^d$ to $\R$,
 and is a Banach space with the supremum norm.
$C_0(\R^d)$ is the space of continuous functions from $\R^d$ to $\R$ that vanish as $\abs{x}\to\infty$.
\begin{defn}
A family $(S_t)_{t\geq 0}$ of linear operators $S_t: C_0(\R^d)\to C_0(\R^d)$ is a \emph{Feller semigroup}
 if it is a strongly continuous semigroup of contractions such that
\begin{enumerate}
%\item $S_0 = \Id$.
%\item $S_s S_t = S_{s+t}$ for all $s,t\geq 0$.
\item $f\geq 0 \quad \Rightarrow \quad S_t f\geq 0 \quad$ for all $t\geq 0$, $f\in C_0(\R^d)$.
%\item $S_t$ is a contraction, i.e. $\norm{S_t}\leq 1$ for all $t\geq 0$.
\item $S_t 1 = 1$ for all $t\geq 0$.
%\item $\lim_{t\to 0} \norm{S_t f-f} = 0$ for all $f\in C_0(\R^d)$.
\end{enumerate}
\end{defn}
The definition of Feller semigroups is not consistent throughout the literature.
\begin{defn}
Let $(\tau_a,a\in\R^d)$ be the translational group acting in $B_b(\R^d)$,
 so that $(\tau_a f)(x)=f(x-a)$ for all $a,x\in\R^d$, $f\in B_b(\R^d)$.
A semigroup $(S_t)_{t\geq 0}$ is called \emph{translational invariant},
 if $S_t \tau_a = \tau_a S_t$ for each $t\geq 0$, $a\in\R^d$.
\end{defn}
The Green functions $\Green$ are L\'{e}vy strictly $\alpha$-stable distributions due to Lemma~\ref{lem:SSPM},
 hence there exists a L\'evy process $(X_t)_{t\geq 0}$ (such that $P_{X_1}=\mu$),
 which is unique up to identity in law~\cite[Corollary 11.6]{Sato:1999}.
This L\'evy process $(X_t)_{t\geq 0}$ is a Feller process, 
 hence the associated transition semigroup $(P_t)_{t\geq 0}$ is a Feller semigroup.
This Feller semigroup is translational invariant due to~\cite[Theorem 3.3.1]{Applebaum:2004}
 and satisfies~\cite[Theorem 3.3.3]{Applebaum:2004},
 see also~\cite[Theorem 31.5]{Sato:1999}.
% L^p Markov semigroup
\begin{defn}[{\cite[Section 3.4]{Applebaum:2004}}] \label{def:MarkovSG}
We fix $1\leq p<\infty$
 and let $(S_t, t\geq 0)$ be a strongly continuous contraction semigroup of operators in $L^p(R^d)$.
We say that it is \emph{sub-Markovian} if $f\in L^p(R^d)$ and
 \[ 0\leq f\leq 1 \text{  a.e.} \quad \Rightarrow \quad 0\leq S_t f\leq 1 \text{  a.e.} \]
 for all $t\geq 0$.
Any semigroup on $L^p(R^d)$ can be restricted to the dense subspace $C_c(R^d)$.
If this restriction can then be extended to a semigroup on $B_b(R^d)$
 that satisfies $S_t 1 = 1$ then the semigroup is said to be \emph{conservative}.
A semigroup that is both sub-Markovian and conservative is said to be $L^p$-Markov.
\end{defn}
\begin{prop} \label{prop:MarkovSG}
For $0<\alpha\leq 2$, $\abs{\theta} \leq \min\{\alpha,2-\alpha\}$ and $1\leq p<\infty$, 
 the Riesz-Feller operator~$\RieszFeller$ generates an $L^p$-Markov semigroup  
 \benn
  S_t: L^p(\R) \to L^p(\R)\,, \quad u_0 \mapsto S_t u_0 = \Green(\cdot,t)\ast u_0 \,. 
 \eenn
If $p=2$ then 
\[ (S_t f)(x) = \FourierInv[e^{t \psi^\alpha_\theta(\cdot)} \hat f(.)](x) \]
for all $t\geq 0$, $x\in\R^d$, $f\in L^2(\R^d)$.
Moreover,
 the infinitesimal generator $A$ of $S_t$ satisfies 
 \[ (A f)(x) = \FourierInv [\psi^\alpha_\theta \hat f](x) \]
 for all $f\in D_A$ whereat $D_A=H^\alpha(\R^n)$.
\end{prop}
\begin{proof}
The Green functions $\Green$ are L\'{e}vy strictly $\alpha$-stable distribution due to Lemma~\ref{lem:SSPM},
 hence there exists a L\'evy process $(X_t)_{t\geq 0}$ (such that $P_{X_1}=\mu$),
 which is unique up to identity in law~\cite[Corollary 11.6]{Sato:1999}.
This L\'evy process $(X_t)_{t\geq 0}$ defines a transition semigroup $(P_t)_{t\geq 0}$,
 which is an $L^p$-Markov semigroup for $1\leq p<\infty$ due to~\cite[Theorem 3.4.2]{Applebaum:2004}.
This transition semigroup $(P_t)_{t\geq 0}$ can be identified with the convolution semigroup $(S_t)_{t\geq 0}$,
 see also Remark~\ref{remark:semigroup:differences}. 
Moreover, for $p=2$ the additional properties are discussed in~\cite[Exercise 3.4.3]{Applebaum:2004}
 and~\cite[Theorem 3.4.4]{Applebaum:2004}.
\end{proof}

Due to Theorem~\ref{thm:RieszFeller:extension},
 a non-degenerate Riesz-Feller operator $\RieszFeller$ operator can be extended to a linear operator from $C^2_b(\R)$ to $C_b(\R)$.
Theorem~\ref{thm:IG} states that a non-degenerate Riesz-Feller operator $\RieszFeller$ is the generator of a Feller semigroup,
 which we extend to a semigroup on $B_b(\R)$ ($L^\infty(\R)$) in Theorem~\ref{thm:Bb:semigroup}.
This extended semigroup is not strongly continuous,
 since $\Green$ is a continuous probability density,
 see also the discussion in~\cite[page 427 ff.]{Jacob:2001}.
A priori it is not clear,
 how the extension of the semigroup $(S_t)_{t\geq 0}$
 and the extension of the infinitesimal generator $\RieszFeller$ are related.
This issue is discussed in~\cite{Schilling:1998}, see also~\cite[Section 4.8]{Jacob:2001}.

Specifically, if the extended semigroup $\tilde S_t$ is not strongly continuous on $B_b(\R)$,
 then the infinitesimal generator cannot be defined as $\tilde A u = \lim_{t\to 0} \frac{\tilde S_t u - u}{t}$,
 where the limit is understood in the strong sense on $B_b(\R)$.
However, the following result is true.
\begin{lem}[{\cite[Lemma 4.8.7]{Jacob:2001}}]
Let $(S_t)_{t\geq 0}$ be a Feller semigroup with extension $(\tilde S_t)_{t\geq 0}$ onto the space $B_b(\R)$.
Then we have for all $u\in C_b(\R)$
\[ \lim_{t\to 0} \tilde S_t u(x) = u(x) \]
uniformly on compact sets $K\subset\R$.
\end{lem}
\begin{defn}[{\cite[Definition 4.8.6]{Jacob:2001}}]
A Feller semigroup $(S_t)_{t\geq 0}$ is called a \emph{strong Feller semigroup}
 if for all $t>0$ the operator $\tilde S_t$ maps $B_b(\R)$ into $C_b(\R)$.

We call $(S_t)_{t\geq 0}$ a \emph{$C_b$-Feller semigroup}
 if for each $t\geq 0$ the restriction of $\tilde S_t$ to $C_b(\R)$ maps $C_b(\R)$ into itself.  
\end{defn}
\cite[Example 4.8.21]{Jacob:2001} shows that the Feller semigroup $(S_t)_{t\geq 0}$ is a strong Feller semigroup.
\cite[Example 4.8.26]{Jacob:2001} establishes a relation between the $C_b$-extension of $\RieszFeller$ 
 and the $B_b$-extension of the strong Feller semigroup.

\begin{lem}[{\cite[Lemma 4.8.12]{Jacob:2001}}]
Let $\SG$ be a $\Cb$-Feller semigroup on $\Co$ with extension $\SGextension$ on $\Bb$.
For every $u\in\Cb$ there exists at most one element $g\in\Cb$ such that 
 \[ \tilde S_t u(x) - u(x) = \integrall{0}{t}{ (\tilde S_s g)(x) }{s} \]
 holds for all $t>0$ and $x\in\R^n$.
Moreover, we have 
 \[ g(x) = \lim_{t\to 0} \frac{\tilde S_t u(x) - u(x)}{t} \]
 uniformly on compact sets.
\end{lem}

\begin{defn}[{\cite[Definition 4.8.14]{Jacob:2001}}]
Let $\SG$ be a $\Cb$-Feller semigroup with generator $(A,D(A))$ on $\Co$.
Further let $\SGextension$ be the extension of $\SG$ to $\Bb$.
The $\Cb$-extension of $(A,D(A))$ or the $\Cb$-generator of $\SGextension$
 is the operator $(\tilde A,D(\tilde A))$ defined by
\[ D(\tilde A):=
    \Set{ u\in\Cb }{ \text{$\lim_{t\to 0} \frac{\tilde S_t u(x) - u(x)}{t}$ exists uniformly on compact sets} } 
\]
and 
\[ \tilde A u(x):= \lim_{t\to 0} \frac{\tilde S_t u(x) - u(x)}{t}\,, \qquad u\in D(\tilde A)\,. \]
\end{defn}

\begin{remark}[{\cite[Remark 4.8.15]{Jacob:2001}}]
Clearly $(\tilde A,D(\tilde A))$ is an extension of $(A,D(A))$,
 i.e. $D(A)\subset D(\tilde A)$ and $\tilde A|_{D(A)} = A$.
Moreover, typical relations between $(A,D(A))$ and $\SG$
 do also hold for $(\tilde A,D(\tilde A))$ and $(\tilde S_t|_{\Cb})_{t\geq 0}$.
In particular, we have for $u\in D(\tilde A)$ that $\tilde S_t u\in D(\tilde A)$ and
 \[ \Diff{}{t} \tilde S_t u = \tilde A \tilde S_t u = \tilde S_t \tilde A u\,. \]
For $u\in\Cb$ and $t\geq 0$
 it follows that $\integrall{0}{t}{\tilde S_s u}{s} \in D(\tilde A)$ and 
 \[ \tilde S_t u - u = \tilde A \integrall{0}{t}{ \tilde S_s u }{s}\,, \]
 as well as 
 \[ \tilde S_t u - u = \integrall{0}{t}{ \tilde A \tilde S_s u }{s} = \integrall{0}{t}{ \tilde S_s \tilde A u }{s} \]
 for $u\in D(\tilde A)$.
\end{remark}